 \documentclass[reqno]{amsart}
\usepackage{mathrsfs}
\usepackage{amsmath}
\usepackage{amsthm}
\usepackage{amsfonts}
\usepackage{amssymb}
\usepackage{url}
\usepackage{enumerate}
\usepackage[pdftex,bookmarks=true]{hyperref}
\usepackage[usenames, dvipsnames]{color}
\usepackage{verbatim}
\usepackage{tikz}
\usepackage{subcaption}
\usepackage{pgfplots}
\usepackage{adjustbox}
\usepackage{xcolor}
\usepackage{pdfsync}

\newcommand\Vol{{\operatorname{Vol}}}
\newcommand\rank{{\operatorname{rank}}}

\newcommand\R{{\mathbb{R}}}
\newcommand\C{{\mathbb{C}}}
\newcommand\Q{{\mathbf{Q}}}

\renewcommand\P{{\mathbf{P}}}
\newcommand\E{{\mathbf{E}}}

\newcommand\diag{{\operatorname{diag}}}

\newcommand\Z{{\mathbb{Z}}}

\newcommand\F{{\mathbb{F}}}

\newcommand\col{{\mathbf{c}}}
\newcommand\row{{\mathbf{r}}}

\newcommand\ep{\varepsilon}
\newcommand\al{\alpha}
\newcommand\la{\lambda}
\newcommand\1{\mathbf{1}}

\newcommand\Ba{{\mathbf a}}
\newcommand\Bb{{\mathbf b}}

\newcommand\Bd{{\mathbf d}}

\newcommand\Bf{{\mathbf f}}

\newcommand\Bu{{\mathbf u}}
\newcommand\Bv{{\mathbf v}}
\newcommand\Bw{{\mathbf w}}
\newcommand\Bx{{\mathbf x}}

\newcommand\Bz{{\mathbf z}}

\newcommand\BG{{\mathbf G}}

\newcommand\BY{{\mathbf Y}}

\renewcommand\Pr{{\mathbf P }}

\newcommand\CC{{\mathcal C}}
\newcommand\CE{{\mathcal E}}
\newcommand\CF{{\mathcal F}}
\newcommand\CG{{\mathcal G}}

\newcommand\CN{{\mathcal N}}
\newcommand\CO{{\mathcal O}}

\newcommand\CS{{\mathcal S}}

\newcommand\CV{{\mathcal V}}
\newcommand\CW{{\mathcal W}}

\renewcommand\mod{\ \operatorname{mod}\ }

\newcommand\supp{\operatorname{supp}}

\newcommand\eps{\varepsilon}

\newcommand\Sp{\operatorname{Sp}}
\renewcommand\a{\alpha}

\newcommand\cok{\mathbf{Cok}}
\newcommand\Aut{\mathbf{Aut}}
\newcommand\bs{\backslash}

\newcommand{\ra}{\rightarrow}
\newcommand\isom{\simeq}
\newcommand\Hom{\operatorname{Hom}}
\newcommand\Sur{\operatorname{Sur}}
\newcommand\sub{\subset}
\newcommand\im{\operatorname{im}}
\newcommand\rk{{\operatorname{rank}}}

\newcommand\CP{{\mathfrak P}}
\newcommand\CB{{\mathfrak W}}

\newcommand\tensor{\otimes}

\newcommand\Bin{\operatorname{Binom}}
\newcommand\suppi{\operatorname{supp_{index}}}

\newcommand\Sym{\operatorname{Sym}}
\newcommand{\dc}{\gamma}

\usepackage{fullpage}

\theoremstyle{plain}
\numberwithin{equation}{section}

\newtheorem{prop}[equation]{Proposition}
\newtheorem{proposition}[equation]{Proposition}
\newtheorem{theorem}[equation]{Theorem}
\newtheorem{corollary}[equation]{Corollary}
\newtheorem{cor}[equation]{Corollary}

\newtheorem{lemma}[equation]{Lemma}
\newtheorem{fact}[equation]{Fact}

 \newtheorem{claim}[equation]{Claim}

\theoremstyle{definition}
  \newtheorem{definition}[equation]{Definition}
  
\theoremstyle{remark}
\newtheorem{remark}[equation]{Remark}

  \newcommand{\melanie}[1]{{\color{blue} \sf $\clubsuit\clubsuit\clubsuit$ Melanie: [#1]}}
    \newcommand{\HQ}[1]{{\color{red} \sf $\clubsuit$ Hoi's new comment: [#1]}}
        \newcommand{\HC}[1]{{\color{BrickRed} \sf $\clubsuit$ Hoi: [#1]}}

\begin{document}

\title{Local and global universality of random matrix cokernels}

\author{Hoi H. Nguyen}
\address{Department of Mathematics\\ The Ohio State University \\ 231 W 18th Ave \\ Columbus, OH 43210 USA}
 \email{nguyen.1261@math.osu.edu}

\author{Melanie Matchett Wood}
\address{Department of Mathematics\\
Harvard University\\
Science Center Room 325\\
1 Oxford Street\\
Cambridge, MA 02138 USA}  
\email{mmwood@math.harvard.edu}

\begin{abstract} In this paper we study the cokernels of various random integral matrix models, including random symmetric, random skew-symmetric, and random Laplacian matrices. We provide a systematic method to establish universality under very general randomness assumptions. Our highlights include both local and global universality of the cokernel statistics of all these models. In particular,  we find the probability that a sandpile group of an Erd\H{os}-R\'enyi random graph is cyclic, answering a question of Lorenzini from 2008.
\end{abstract}

\maketitle

\section{Introduction} 

Given a graph $G$, there is a naturally associated abelian group $S_G$, which has gone in the literature by many names, including the sandpile group, the critical group, the Jacobian, and the Picard group (due to its independent appearance in many subjects ranging from statistical mechanics to combinatorics to arithmetic geometry), see \cite{Bacher, Dhar, Lo0,Lo1,Rush}. %
A combinatorial way to interpret this group is via Chip-Firing game \cite{Big,GK}. More precisely, a (degree zero) divisor on $G$ is a function $\delta: V(G) \to \Z$ where $\sum_v \delta(v)=0$. One can naturally define the sum $\delta_1+\delta_2$ of two such divisors $\delta_1,\delta_2$ by $(\delta_1+\delta_2)(v) = \delta_1(v)+\delta_2(v), v\in V(G)$. Let $Div^0(G)$ be the group of (degree zero) divisors equipped with this addition. There is a natural equivalence relation over the elements of $Div^0(G)$, namely we say $\delta_1 \sim \delta_2$ if there is a sequence of chip-firing moves from $\delta_1$ to $\delta_2$, where these moves include ``borrow" (a borrow move at $v_0$ changes $\delta$ to $\delta_{\Bb,v_0}$, where $\delta_{\Bb,v_0}(v_0)= \delta(v_0)+d(v_0)$ and  $d(v_0)$ is the degree of $v_0$, and $\delta_{\Bb,v_0}(v') = \delta(v')-1$ for all neighbors $v'$ of $v_0$, as well as $\delta_{\Bb,v_0}(v)= \delta(v)$ at all other vertices) and ``firing" (a firing move at $v_0$  changes $\delta$ to $\delta_{\Bf,v_0}$ where $\delta_{\Bf,v_0}(v_0) =  \delta(v_0)-d(v_0)$ and $\delta_{\Bf,v_0}(v') = \delta(v')+1$ for all neighbors $v'$ of $v_0$, as well as $\delta_{\Bf,v_0}(v)= \delta(v)$ at all other vertices). The group $S_G$ is then simply the quotient $Div^0(G)/\sim$. 

We also invite the reader to Subsection \ref{sub:Lap} for further discussion on this fascinating group. One can see that $S_G$ is the cokernel of the (combinatorial) Laplacian of $G$, and its order is the number of spanning trees of $G$. About fifteen years ago, Lorenzini  \cite{Lorenzini2008} asked how often sandpile groups of graphs are cyclic, and a specific answer was conjectured in \cite{Clancy2015} (see also Table \ref{fig:ER}).  In this paper, we answer that question, proving the conjecture.

\begin{theorem}\label{theorem:cyclic:L_G}  Let $G\in G(n,1/2)$ be an Erd\H{os}-R\'enyi graph on $n$ vertices. We have
$$\lim_{n\to \infty} \P(S_G \mbox{ is cyclic}) =  \prod_{i=1}^\infty \zeta(2i+1)^{-1}\approx .7935$$
where $\zeta(s)$ is the Riemann zeta function. 
\end{theorem}

 This theorem is one application of our development of new techniques for proving {global} statistics of random symmetric integral matrices (many of which also apply to random skew-symmetric integral matrices). %
  Previous work of the second author \cite{Wood2017} studied \emph{local} statistics, in particular the Sylow $p$-subgroups of the cokernels of random symmetric matrices for any finite set of primes $p$.  \emph{Global} properties of a finite abelian group are those that cannot necessarily be determined from a finite list of Sylow $p$-subgroups, such as cyclicty as we see above. 

Our work is the first to address the \emph{universality} aspect of these global statistics of random symmetric matrices, i.e. the extent to which the statistics, asymptotically, do not depend on the distribution of the entries of the matrices.
A significant part of Random Matrix Theory is to study the universality phenomenon of empirical spectral distributions under different matrix symmetries, under different sources of randomness, and  under different scalings. In this paper we are pursuing a  different direction of universality of random matrices, namely the behaviour of their cokernels as  abelian groups. (When the entries $m_{ij}$ of an $n\times n$ matrix $M$ are integers, the cokernel is defined as $\cok(M) = \Z^n/M( \Z^n).$)

We make the following definition to restrict the types of entries our random matrices will have. 
Let $\al>0$ be given, and fixed throughout the paper. %
We say a random integer $\xi$ is \emph{$\alpha$-balanced} if for every prime $p$ we have
\begin{equation}\label{eqn:alpha}
\max_{r \in \Z/p\Z} \P(\xi\equiv r \pmod{p}) \le 1-\alpha.
\end{equation}

For i.i.d symmetric matrices, we have a universality result analogous to Theorem~\ref{theorem:cyclic:L_G}.

\begin{theorem}\label{theorem:cyclic:sym} Let $M_n=M_{n \times n} =(x_{ij})_{1\le i,j\le n}$ be a random symmetric matrix 
with  upper triangular entries  $x_{ij}$ for $i\geq j$ each i.i.d. copies of an integral $\alpha$-balanced random variable $\xi$.  
We have
$$\lim_{n\to \infty} \P\Big(\cok(M_n) \mbox{ is cyclic}\Big) =  \prod_{i=1}^\infty \zeta(2i+1)^{-1}\approx .7935.$$
\end{theorem}

Our methods also apply to i.i.d skew-symmetric matrices,  though the distributions are different (e.g. because the rank of a skew-symmetric matrix is always even).

\begin{theorem}\label{theorem:cyclic:alt}  Let $A_n=(x_{ij})_{1\le i,j\le n}$ be a random skew-symmetric matrix where the upper diagonal entries  $x_{ij},$ for $i> j$, are i.i.d. copies of an integral $\alpha$-balanced random variable $\xi$.
Let $\mathcal{S}$ be the set of finite abelian groups of the form $H\times H$, and for such a group $G$, let 
$\Sp(G)$ be the group of automorphisms of $G$ that preserve a fixed non-degenerate skew-symmetric bilinear pairing.
We have 
$$\lim_{n\to \infty} \P\Big(\cok(A_{2n}) \mbox{ is a square of a cyclic group}\Big)  = \zeta(2)  \prod_{i=1}^\infty \zeta(2i+1)^{-1} \prod_{p \textrm{ prime}} (1-p^{-2} + p^{-3}).$$
Furthermore, for any abelian group $C$
 \begin{equation}\label{E:oddprob}\lim_{n\to \infty}  \P\Big(\cok(A_{2n+1})\isom C\Big)  =
\begin{cases}\frac{1}{|\Sp(B)|}\prod_{i=1}^\infty \zeta(2i+1)^{-1}
&C=\Z\times B, \textrm{ for }B\in \mathcal{S}\\
0 & \textrm{otherwise}.
\end{cases} .
\end{equation}
 \end{theorem}

What allows us to prove global statistics is a new method to understand the behavior of the matrices modulo primes that are very large compared to $n$.
This is made more difficult by the dependence between the upper and lower triangular entries in the above models.  Moreover, it is significantly more challenging to handle the Laplacian model in Theorem~\ref{theorem:cyclic:L_G} because of the dependence of the diagonal on the other entries of the matrix.    We discuss our new techniques to over come these difficulties later in the introduction.

\begin{table}
\centering
  \begin{tabular}{ |c | c | c | c | }
  \hline
    n $\backslash$ \ q & .3 & .5 & .7 \\ \hline
    15 & .784255 & .792895 & .775746 \\ \hline
    30 & .793807 & .793570 & .793375 \\ \hline
    45 & .793308 & .793962 & .793637 \\ \hline
    60 & .793436 & .793694 & .79354 \\ 
    \hline
  \end{tabular}
  \caption{Clancy et.al. computed in \cite{Clancy2015} the Jacobians of $10^6$ connected random graphs with $n$ vertices and edge probability $q$ and this chart gives the probability that the Jacobian is cyclic in each case.}
\label{fig:ER}
\end{table}

\begin{center}
{ \bf Past work and further results of this paper}
\end{center}

\subsection{Cokernels of random integral non-symmetric matrices}

For an abelian group $G$ and a prime $p$, we write $G_p$ for the Sylow $p$-subgroup of $G$. For a set $P$ of primes, we write $G_P:=\prod_{p\in P}G_p$, the product of the Sylow $p$-subgroups of $G$ for all $p\in P$.
Motivated by the Cohen-Lenstra heuristics for the distribution of class groups of number fields, the following has been shown by the second author.

\begin{theorem}\label{theorem:W:i.i.d.}\cite[Corollary 3.4]{W1}
Let $Q_n$ be a random matrix with entries 
 i.i.d copies of a $\alpha$-balanced random  integer $\xi$.  Let $B$ be any finite abelian group, and let $P$ be a finite set of primes including all those that divide $|B|$. Then
$$
\lim_{n\ra\infty} \P\Big((\cok(Q_{n}))_P \simeq B \Big) = \frac{1}{|\Aut(B)|}\prod_{p\in P} \prod_{k=1}^\infty (1-p^{-k}). 
$$
\end{theorem}

\begin{remark}\label{R:i.i.d.} For a particular finite abelian group $B$, by taking $P$ larger and larger and since $\prod_{p \textrm{ prime}} (1-p^{-1})=0$, we have 
$$\lim_{n\ra\infty} \P(\cok (Q_{n})\isom B)=0.$$ 
\end{remark}

While the above results hold for {\it local} statistics (i.e. when $\cok(M_{n})_P$ is isomorphic to a given finite abelian group), it is natural to study {\it global} statistics, such as how often $\cok (Q_{n})$ is cyclic \footnote{Clearly one could also ask about other global properties, but cyclicity seems to be one of the most natural ones.}. 
 For example, the following gives a main result of \cite{NgW} in the case where the matrix entries do not change with $n$ (the case of interest in this paper).

\begin{theorem}\label{theorem:cyclic:i.i.d.}\cite[Theorem 1.2]{NgW}  Let $Q_n=Q_{n \times n} =(x_{ij})_{1\le i,j\le n}$ be a random matrix where the entries $x_{ij}$ are i.i.d. copies of an integral $\alpha$-balanced random  integer $\xi$.
We have
 \begin{align*}
\lim_{n\to \infty} \P\Big(\cok(Q_{n\times n}) \mbox{ is cyclic}\Big) = \prod_{p \textrm{ prime}} (1 + \frac{1}{p(p-1)}) \prod_{k=2}^\infty \zeta(k)^{-1}.
\end{align*}
 \end{theorem}
 
%
 
 \begin{comment}
This probability has been seen in several papers studying the probability that a random lattice in $\Z^n$ is co-cyclic (gives cyclic quotient), in cases when these lattices are drawn from the nicest, most uniform distributions, e.g. uniform on lattices up to index $X$ with $X\ra\infty$ \cite{CKK,NS,Pet}, or with basis with uniform entries in $[-X,X]$ with $X\ra\infty$ \cite{SW}. We note that the i.i.d. model here is more relevant to the Cohen-Lenstra heuristic, and as the reader will see, the cyclicity probability here is different from Theorem \ref{theorem:cyclic:sym} for symmetric matrices. 
\melanie{The first part of the previous sentence I am not so sure about, and the second half seems awkward here.  Maybe this whole paragraph can be replaced by a paragraph near our new results that explains that these probabilities may have been seen before in uniform-like distributions, but that is rather different from proving them here.  This could be added to the heuristic paragraph in fact, but should be in the symmetric case. }
\end{comment}

We also showed the following generalization of Theorem~\ref{theorem:cyclic:i.i.d.}. 

\begin{theorem}\label{theorem:prodcyc:i.i.d.}\cite[Theorem 2.5]{NgW} Let $Q_n=Q_{n \times n}$ be as in Theorem~\ref{theorem:cyclic:i.i.d.}. 
Let $B$ be a finite abelian group and let $k_0$ be larger than any prime divisor of $|B|$, and define $C_B=\{B\times C\,|\, C \textrm{ cyclic, } p\nmid |C| \textrm{ for }1<p<k_0 \}$, the set of groups differing from $B$ by a cyclic group with order only divisible by primes at least $k_0$. 
 Then, we have
  $$\lim_{n\to \infty}  \P\Big(\cok(Q_{ n \times n})\in C_B\Big)  =  \frac{1}{|\Aut(B)|}   \prod_{\substack{p<k_0\\p \textrm{ prime}}} (1-p^{-1}) \prod_{\substack{p\geq k_0\\p \textrm{ prime}}} (1 + \frac{1}{p(p-1)})
  \prod_{k=2}^\infty \zeta(k)^{-1}.$$
 \end{theorem}

\subsection{Cokernels of random integral symmetric matrices} 
Since the behavior of the empirical spectral distribution is quite different for general versus symmetric matrices (see for instance \cite{BSbook, Edelman, TVcir} and \cite{Mehta, Pastur,  Wigner}), and also because the cokernels of symmetric matrices naturally have pairings \cite{CLP}, one might expect that the cokernel statistics of symmetric matrices are different from the non-symmetric case. 
For local statistics, the second author showed the following result on the local statistics.

\begin{theorem}\label{theorem:W:sym}\cite[Corollary 9.2]{Wood2017} Let $M_n=M_{n \times n} =(x_{ij})_{1\le i,j\le n}$ be a random symmetric matrix 
 with upper triangular entries  $x_{ij}$, for  $i\geq  j$, that are i.i.d. copies of an  $\alpha$-balanced random integer $\xi$.
Let $B$ be any finite abelian group, and let $P$ be a finite set of primes including all those that divide $|B|$. Then  
$$\lim_{n\to \infty} \P((\cok(M_n))_P \isom B) = \frac{\#\{ \mbox{symmetric, bilinear, perfect } \phi: B \times B \to \C^\ast \}}{|B| |\Aut(B)|} \prod_{p\in P}\prod_{k\ge 0} (1 -p^{-2k-1})$$
and for any $p\in P$
$$\lim_{n \to \infty} \P(\rank(M_n/p) = n-r) = p^{-r(r+1)/2}  \prod_{i=r+1}^\infty (1 -p^{-i}) \prod_{i=1}^\infty (1 - p^{-2i})^{-1}.$$
\end{theorem}
Here $M/p$ is the matrix of entries modulo $p$. Note that if $B = \oplus_{i} \Z/p^{\la_i}\Z$ with $\la_1\ge \la_2 \ge \dots$, and with conjugate partition $\la_1' \ge \la_2' \ge \dots $ then 
$$\# \{\mbox{sym., bilinear, perfect } \phi: B \times B \to \C^\ast \} = p^{-\sum_i \la_j'(\la_i'+1)/2} \prod_{i=1}^{\la_1} \prod_{j=1}^{\lfloor (\la_i' - \la_{i+1}')/2 \rfloor} (1-p^{-2j})^{-1}.$$

\begin{remark}\label{R:sym} Similarly to Remark \ref{R:i.i.d.}, for a particular finite abelian group $B$, by taking $P$ larger and larger, for random symmetric matrices $M_n$ as above we also have 
$$\lim_{n\ra\infty} \P(\cok (M_{n})\isom B)=0.$$ 
\end{remark}

Now for global statistics such as cyclicity, we will first explain a heuristic guess for this probability. Note that $\cok(M_{n})$ is cyclic if and only if its reduction to modulo $p$ is cyclic for all primes $p$. We then make two idealized heuristic assumptions on $M_{n}$.  (i) (uniformity assumption) Assume that for each prime $p$ the entries of $M_{n}$ are uniformly distributed modulo $p$. In this case, it is classical that the probability that $M_n$ is cyclic is $(1-p^{-n-1})\prod_{i=1}^{(n-2)/2} (1-p^{-2i-1})$ when $n$ is even and 
$\prod_{i=1}^{(n-1)/2} (1-p^{-2i-1})$ when $n$ is odd.  (ii) (independence assumption) We next assume that the statistics of $M_{n}$ reduced to modulo $p$ are asymptotically mutually independent for all primes $p$. Under these assumptions, as $n \to \infty$, the probability that $M_{n}$ is surjective would be asymptotically the product of all of the surjectivity probabilities modulo $p$, which leads to the conjecture that  $\cok (M_{n})$ is cyclic with asymptotic probability $\prod_{i=1}^\infty \zeta(2i+1)^{-1}$, which is around 0.7935. 
The matrices in this paper do not have to satisfy either assumption, and indeed they can violate them dramatically.  For example, if the matrix entries only take values $0$ and $1$, then they cannot be uniformly distributed mod any prime $>2$, and the matrix entries mod $3$ are not only not independent from the entries mod $5$, but they are in fact determined by the entries mod $5$.  

One of the main goals of this paper is to show that the heuristic gives a correct prediction even when the matrices fail the above assumptions dramatically.
The result for the probability of cyclicity was stated in Theorem~\ref{theorem:cyclic:sym}.
We can in fact show a little bit more,  similarly to Theorem \ref{theorem:prodcyc:i.i.d.}, and one of our main results is the following.
\begin{theorem}\label{theorem:prodcyc:sym} Let $M_n=M_{n \times n}$ be random symmetric matrices as in Theorem~\ref{theorem:cyclic:sym}. 
Let $B$ be a finite abelian group and let $k_0$ be larger than any prime divisor of $|B|$, and define $C_B=\{B\times C\,|\, C \textrm{ cyclic, } p\nmid |C| \textrm{ for }1<p<k_0 \}$, the set of groups differing from $B$ by a cyclic group with order only divisible by primes at least $k_0$. 
 Then, we have
  $$\lim_{n\to \infty}  \P\Big(\cok(M_{ n \times n})\in C_B\Big)  =  \frac{\#\{ \mbox{symmetric, bilinear, perfect } \phi: B \times B \to \C^\ast \}}{|B| |\Aut(B)|}   \prod_{\substack{p<k_0\\p \textrm{ prime}}} (1-p^{-1})  \prod_{i=1}^\infty \zeta(2i+1)^{-1}.$$
 \end{theorem}

\subsection{Cokernels of skew-symmetric matrices}
Now we discuss another random matrix model whose cokernel universality aspect has not been addressed in the literature \footnote{We refer the reader to \cite{MR,SS} and the references therein for universality aspects of the empirical spectral distribution of these random matrices.}. Let $A_n=(x_{ij})_{1\le i,j\le n}$ be an  skew-symmetric (i.e.  alternating) random matrix where $x_{ij},$ for  $1\le i< j\le n$ are i.i.d. copies of an $\alpha$-balanced random integer $\xi$ (and $x_{ii}=0$ and $x_{ji}=-x_{ij}$). In the special case that $\xi_{ij}$ are i.i.d. taking values in $\{-1,0,1\}$, we can view $A_n$ as the adjacency matrix of a random tournament graph on $n$ vertices. 
The cokernel distribution of Haar distributed skew-symmetric matrices over $\Z_p$ was studied in \cite[Theorem 3.9]{Bhargava2015b} in connection to heuristics for various statistics of elliptic curves.

To introduce our results, we will need some more notation. We let ${\bf 1}_{n \textrm{ odd}}$ be $1$ if $n$ is odd and $0$ otherwise.   For a finite set of primes $P$, 
we say a group is a \emph{$P$-group} if its order is a product of powers of primes in $P$, and we
let $\mathcal{S}_P$ be the set of $P$-groups in $\mathcal{S}$ (squares of abelian groups). A finite abelian group has a non-degenerate skew-symmetric linear pairing to $\C^*$ if and only if it is of the form $H\times H$  (see \cite[Proposition 2]{Delaunay2001}).  A group $G$ of this form has a unique such pairing up to isomorphism, and we let $\Sp(G)$ be the group of automorphisms of $G$ that preserve the pairing.  If $A$ is a skew-symmetric integral matrix, then it has even rank and   
$\cok_{\operatorname{tors}}(A)$, the set of elements of $\cok(A)$ of finite order, is in $\mathcal{S}$ \cite[Sections 3.4 and 3.5]{Bhargava2015b}.

In this paper, we first show the following local statistics analogous to Theorems \ref{theorem:W:i.i.d.} and \ref{theorem:W:sym}.
\begin{theorem}\label{theorem:W:alt} Let $A_n=(x_{ij})_{1\le i,j\le n}$ be a random skew-symmetric matrix where the upper diagonal entries are $x_{ij},$ for $i> j$ are i.i.d. copies of an  $\alpha$-balanced random integer $\xi$.
Let $B$ be any finite abelian group, and let $P$ be a finite set of primes including all those that divide $|B|$.  Then
\begin{align*}
\lim_{n\ra\infty} \P(\rank (A_n) =n-{\bf 1}_{n \textrm{ odd}} )&=1\\
\lim_{n\ra\infty} \P((\cok (A_{2n}))_P\isom B) &=
\begin{cases}
\frac{|B|}{|\Sp(B)|} \prod_{p\in P} \prod_{i=0}^\infty (1-p^{-2i-1})
&B\in \mathcal{S}_P\\
0 & \textrm{otherwise}
\end{cases}
\\
\lim_{n\ra\infty} \P((\cok_{\operatorname{tors}}(A_{2n+1}))_P\isom B) &=\begin{cases}\frac{1}{|\Sp(B)|} \prod_{p\in P} \prod_{i=1}^\infty (1-p^{-2i-1})
&B\in \mathcal{S}_P\\
0 & \textrm{otherwise}.
\end{cases}
\end{align*}

Furthermore, for any prime $p$ and non-negative integer $r$, we have
\begin{align*}
\lim_{n\ra\infty} \P(\rank ( A_{2n}/p)=2n-2r)&=p^{-r(2r-1)}\frac{\prod_{i=r}^{\infty}(1-p^{-2i-1}) }{ \prod_{k=1}^r (1-p^{-2k})} \\
\lim_{n\ra\infty} \P(\rank ( A_{2n+1}/p)=2n-2r)&=p^{-r(2r+1)}\frac{\prod_{i=r+1}^{\infty}(1-p^{-2i-1}) }{ \prod_{k=1}^r (1-p^{-2k})}.
\end{align*}
\end{theorem}

The distribution of the torsion of the cokernels is different for odd and even dimensional matrices.  Note the additional factor of $|B|$ in the even case, and well as the normalization constants starting their products in different places.

\begin{remark}\label{R:even0} Similarly to Remarks \ref{R:i.i.d.} and \ref{R:sym}, for a particular finite abelian group $B$ we have 
$$\lim_{n\ra\infty} \P(\cok(A_{2n})\isom B)=0.$$ 
\end{remark}

Then, for global statistics we prove Theorem~\ref{theorem:cyclic:alt}, giving the cyclicity probability in the even dimensional case, and the non-zero probabilities of each group in the odd dimensional case.

\begin{remark}\label{R:odd1}
In contrast to Remark~\ref{R:even0}, we have that the probabilities of Theorem~\ref{theorem:cyclic:alt} for the odd dimensional case in \eqref{E:oddprob} sum to $1$ \cite[Theorem 9]{Delaunay2001}.  Thus one can, using Fatou's lemma, determine the asymptotic probability that the cokernel has any property (e.g. cyclicity) by summing the probabilities of the groups with that property (see \cite[Lemma 2.4]{NgW}).
\end{remark}

Additionally, analogously to Theorems \ref{theorem:prodcyc:i.i.d.} and \ref{theorem:prodcyc:sym}, one of our main results is the following.
\begin{theorem}\label{theorem:prodcyc:alt} Let $A_n$ be as in Theorem~\ref{theorem:cyclic:alt}. 
Let $B \in \CS$ be a finite abelian group and let $k_0$ be larger than any prime divisor of $|B|$, and define $C_{B}=\{B \times C \times C\,|\, C \textrm{ cyclic, } p\nmid |C| \textrm{ for }1<p<k_0 \}$. Then, we have 
 $$\lim_{n\to \infty}  \P\Big(\cok(A_{2n})\in C_B\Big)  =  \frac{|B|}{|\Sp(B)|}    \prod_{i=1}^\infty \zeta(2i+1)^{-1}   \prod_{\substack{p< k_0\\p \textrm{ prime}}} (1-p^{-1}) \prod_{\substack{p\geq k_0\\p \textrm{ prime}}} (1-p^{-2})^{-1}(1-p^{-2} + p^{-3}).$$

 \end{theorem}

\subsection{Laplacian of random graphs}\label{sub:Lap} Now we turn to one of our main motivating applications. Given a matrix  $M=(x_{ij})_{1\le i,j\le n}$, let $L_M$ be the Laplacian corresponding to $M$,  with entries
\begin{equation}\label{eqn:Lap}
L_{ij}=\begin{cases}
x_{ij} & \mbox{ if } j \neq i\\
-\displaystyle{\sum_{\substack{1\leq k \leq n\\ k\neq i}} x_{ki}}  & \mbox{ if } j=i,
\end{cases}
\end{equation}
and note the columns of $L_M$ all sum to $0$.
Let $\Z_0$ be the set of vectors in $\Z^{n}$ of zero-sum. We write $S_M$ for the cokernel $\Z_0^{n}/L_M\Z^{n}$. 
In this paper we will be focusing on the case when $M$ is the adjacency matrix of a graph $G$ (directed or undirected), in which case we write $L_G$ for $L_M$. In this case this abelian group, denoted by $S_G$, is called the {\it sandpile group} (or the Jacobian, the Picard group, the critical group) of $G$. It is well known that the order of $S_G$ is the number of spanning trees of $G$ when $G$ is undirected, and there is a similar result for rooted spanning trees when $G$ is directed.
  Note that the study of sandpile groups for directed and undirected graphs and their implications has been an extremely active research direction in recent years. (There is a vast literature on this, which is impossible to list even a small portion of it; we refer the reader to for instance \cite{BN1, CC,FL,FL2,GK} and the references therein.) 
  The behavior of $S_G$ in general is highly non-trivial, for instance it is not clear if there is a way to relate $S_G$ to $S_{G'}$ if they differ by only a few graph theoretic operations! The task of describing $S_G$ for a given $G$ is naturally daunting, and it is perhaps too complicated to wish for a theory to describe $S_G$ properly for most $G$. In another direction, it is natural to trade off a detailed description of each $G$ to have a general pictures on which group structures they might have.

Motivated by this, the current authors showed  the following analog of Theorem \ref{theorem:W:i.i.d.} for sandpile groups of random directed graphs.

\begin{theorem}\label{theorem:W:S_G}\cite[Theorem 1.6]{NgW} Let $G$ be a random directed Erd\H{o}s-R\'enyi graph where each directed edge is chosen independently with probability $1/2$.
  Let $B$ be a finite abelian group.  %
  Then
$$
\lim_{n\ra\infty} \P\Big(S_G \simeq B \Big) = \frac{1}{|B||\Aut(B)|} \prod_{i=2}^{\infty} \zeta(i)^{-1}.
$$
\end{theorem}

\begin{remark}
As in Remark~\ref{R:odd1}, these limiting  probabilities sum to $1$ \cite[Lemma 3.2]{W1} and thus imply limiting probabilities of any other properties.  
\end{remark}

This result has an interesting combinatorial application 
 that  about $45.58\%$ of simple digraphs $G=(V,E)$ has the property that  a chip configuration $\sigma$ on $\Gamma$ stabilizes after a finite number of legal firings if and only if $|\sigma| \le |E|-|V|$; we refer the reader to \cite{FL,NgW} for further details. Notice that the results in \cite{NgW} also extend to the regime that $\al$ is allowed to depend on $n$ (such as $\al_n \ge n^{-1+o(1)}$ and to rectangular matrices, and other important statistics). %

Now we turn to undirected graphs.
For local statistics, the second author showed the following analog of Theorem \ref{theorem:W:sym}

\begin{theorem}\cite[Theorem 1.1]{Wood2017}\label{theorem:W:lap} Let $G\in G(n,1/2)$ 
be an Erd\H{os}-R\'enyi graph on $n$ vertices. Let $B$ be a finite abelian group. Let $P$ be a finite set of primes including all those dividing
$|B|$. Then  
$$\lim_{n\to \infty} \P((S_G)_P \isom B) = \frac{\#\{ \mbox{symmetric, bilinear, perfect } \phi: B \times B \to \C^\ast \}}{|B| |\Aut(B)|} \prod_{p\in P}\prod_{k\ge 0} (1 -p^{-2k-1}).$$
\end{theorem}

{\bf Definition of $L_n$:}
Sometimes it is more convenient to work with the model $L_n$, obtained from $L_G=L_{M}$ (see \eqref{eqn:Lap}) where $M$ is the adjacency matrix of $G$, an Erd\H{os}-R\'enyi graph on $n+1$ vertices, by deleting the last row and column.   More precisely,  we let $x_{ij}$ be independent copies of a uniform random element of $\{0,1\}$ for $1\leq i <j \leq n+1$ and let $L_n$ be the $n\times n$ matrix with entries 
\begin{equation}\label{eqn:redLap}
L_{ij}=\begin{cases}
x_{ij} & \mbox{ if } 1\leq i<j\leq n\\
x_{ji} & \mbox{ if } 1\leq j < i \leq n\\
-\displaystyle{\sum_{\substack{1\leq k \leq n+1\\ k\neq i}} x_{ki}}  & \mbox{ if } 1\leq i=j \leq n.
\end{cases}
\end{equation}
Unlike $L_G$, the determinant of $L_n$ is not necessarily zero. 
When $G$ is an undirected graph, projection onto the first $n$ coordinates gives an isomorphism $\Z^{n+1}_0\ra \Z^n$ that induces an isomorphism $\cok(L_n)\isom S_G$ (where 
$\cok(L_n)=\Z^n/L_n\Z^n$).  (This happens when $G$ is undirected because $L_G$ has rows summing to $0$, so the last column does not contribute to the column space.)

We note that the cyclicity event considered in Theorem \ref{theorem:cyclic:L_G} applied to $L_n$ (and its Smith Normal Form) is equivalent with the event that the greatest common divisor of all $(n-1)\times (n-1)$ minors of $L_n$ is 1. Hence Theorem \ref{theorem:cyclic:L_G} asserts that this event has probability $\approx .7935$ as well. 
The global statistics of Theorem~\ref{theorem:cyclic:L_G} are one of the central applications of the new methods of this paper.
We can also extend Theorem \ref{theorem:prodcyc:sym} to the model $S_G$ and $L_n$.

\begin{theorem}\label{theorem:prodcyc:lap}\label{theorem:cyclic:lap} %
 Let $G\in G(n,1/2)$ 
be an Erd\H{os}-R\'enyi graph on $n$ vertices. 
Let $B$ be a finite abelian group and let $k_0$ be larger than any prime divisor of $|B|$, and define $C_B=\{B\times C\,|\, C \textrm{ cyclic, } p\nmid |C| \textrm{ for }1<p<k_0 \}$. Then, we have
  $$\lim_{n\to \infty} \P(S_G \in C_B) =  \frac{\#\{ \mbox{symmetric, bilinear, perfect } \phi: B \times B \to \C^\ast \}}{|B| |\Aut(B)|}   \prod_{\substack{p<k_0\\p \textrm{ prime}}} (1-p^{-1})  \prod_{i=1}^\infty \zeta(2i+1)^{-1}$$
  and similarly for $\lim_{n\to \infty} \P(\cok(L_n)\in C_B)$.
 \end{theorem}

\subsection{Proof methods} 
 Our approach is to study the cokernels over all primes in order to understand the cokernels over $\Z$.   For each $n$, the approach we use for each prime $p$ depends on the size of $p$ relative to $n$.
\begin{itemize}
\item (Small primes, local statistics, Sections \ref{section:small:alt:1} and \ref{section:small:alt:2}) In our first interval of primes we use Theorem \ref{theorem:W:sym}, Theorem \ref{theorem:W:alt},  and Theorem \ref{theorem:W:lap} to study the $P$-parts of the cokernels, where $P$ is a product of a few fixed primes, and $n \to \infty$. 
In this paper, we only need to prove Theorem \ref{theorem:W:alt} for the skew-symmetric case.  In \cite{Wood2017}, the second author proved the local statistics for symmetric matrices by determining the (group-theoretic) moments of the distribution and proving that moments determine a unique distribution when they don't grow too quickly.  In this paper, we find the moments in the skew-symmetric case, but they grow too quickly to determine a unique distribution.  Indeed, even and odd dimensional skew-symmetric matrices have the same moments but very different distributions.  However, one can leverage further deterministic information about the group structures that can arise to prove that under those restrictions (which are different in the odd and even dimensional cases) that the moments indeed determine a unique distribution.
\vskip .2in
\item (Moderate primes, dynamics and rank statistics, Sections \ref{section:lemmas}, \ref{section:normalvector}, and \ref{section:rankevolving}) Next we study the cokernel modulo a prime $p$ as long as $p$ is sufficiently large and $p \le \exp(n^c)$ for some small constant $c$. Here we provide a very fine approximation of the rank evolution by a combinatorial method, which uses some ingredients from \cite{FJ, FJLS,KNg} and  \cite{LMNg}. This method centers around the study of non-structureness of the normal vectors of random subspaces spanned by the columns (Propositions \ref{prop:structure:subexp:sym}, \ref{prop:structure:subexp:lap}). One of the most challenging parts here is to find the right notion of structures, for which we can estimate very precisely the number of structured vectors even when $p$ is sub-exponentially large and when the matrix entries are not dependent, especially in the Laplacian case. 
\vskip .2in
\item (Large primes, simultaneously trivial statistics, Sections  \ref{section:bilinear} and \ref{section:largeprimes}) In the last stage we study primes $ \exp(n^c)< p < n^{n/2}$. We provide an inverse result characterizing quadratic forms of large concentration probability by building on our previous works \cite{Ng, NgW}. Two highlights of this part for the Laplacian case include the passing of certain rare events from all such large $p$ simultaneously to an event over $\Z$ (Lemma \ref{lemma:passing}), and the innovative part of passing from the Laplacian model to the random symmetric one with prescribed diagonal entries via Lemmas \ref{lemma:strucnotstruc}, \ref{lemma:strucnotstruc:sym}. Here, unlike the combinatorial structures used in the moderate prime parts, our structures are arithmetic (GAP), an extremely important structure that we must have to pass to all primes.

\end{itemize} 
 
As mentioned, for moderate and large primes, the Laplacian model poses a significant challenge because the diagonal entries of this model depend on all other entries. For instance, although this is not our main focus, to our best understanding it is not even known before our work that the Laplacian matrix $L_n$ is non-singular with sub-exponentially high probability, see Corollary \ref{cor:singularity}. Beside the highlights above, among our other technical contributions, the proofs of Lemma \ref{lemma:sparse:2},  Proposition \ref{prop:structure:subexp:lap'}, and Lemma \ref{lemma:strucnotstruc:sym} involve ``structure propagation", a way to deal with partially structured vectors by a series of conditionings. This method seems to be useful and of independent interest.

\subsection{Notations} We use $[n]$ for $\{1,\dots,n\}$. For an index set $I \subset [n]$, we write $I^c$ for the complement of $I$ in $[n]$. We denote the order of groups and sets using either absolute value signs $|\cdot|$ or $\#$. 

{\bf Probability:} We write $\P$ for probability and $\E$ for expected value. For an event $\mathcal{E}$, we write $\bar{\mathcal{E}}$ for its complement. 
We use $\wedge$ for logical \emph{and}.
\vskip .05in

{\bf Analysis:} We write $\exp(x)$ for the exponential function $e^x$. We write $\|.\|_{\R/\Z}$ to be the distance to the nearest integer. Throughout this paper, if not specified otherwise, $\al, A, C, C_\ast, c, c', c_i, K, T, \delta, \eta, \eps, \eps_0$, etc, 
will denote positive constants. When it does not create confusion, the same letter may denote different constants in different parts of the proof. 
 The value of the constants may depend on other constants we have chosen, but will never depend on the dimension $n$, which is regarded as an asymptotic parameter going to infinity. More specifically,  we say ``$f(n,\dots )= O_S(g(n,\dots))$'', or ``$f(n,\dots ) \ll_S g(n,\dots)$'', where $S$ is a subset of the parameters, to mean for any values $v_1,\dots,v_m$ of the parameters in $S$,
there is exists a constant $K>0$ depending on $v_1,\dots,v_m$, such that for all $n$, $|f(n,\dots )|\leq Kg(n,\dots).$ 
Also, we write $f(n,\dots )= \Theta_S(g(n,\dots))$ if $f(n,\dots )= O_S(g(n,\dots))$ and $g(n,\dots )= O_S(f(n,\dots))$. In many cases $S$ can be empty, in which case ``$f(n,\dots )= O(g(n,\dots))$", or ``$f(n,\dots ) \ll g(n,\dots)$",  means $|f(n,\dots )|\leq K g(n,\dots)$ where $K$ is an absolute positive constant. We also write $k=\omega(1)$ if $k\to \infty$ with $n$.  
\vskip .05in

{\bf Linear Algebra:} For a vector $\Bw=(w_1,\dots,w_n)$ we let $\supp(\Bw)=\{i\in[n] | w_i\ne 0\}$. We will also write $X \cdot \Bw$ for the dot product $\sum_{i=1}^n x_i w_i$.
We say $\Bw$ is a \emph{normal} vector for a subspace $H$ if $X\cdot \Bw=0$ for every $X\in H$. For a given index set $J \subset [n]$ and a vector $X= (x_1,\dots, x_n)$, we write $X|_J$ or sometimes $X_J$ to be the subvector of $X$ of components indexed from $J$. In this case we say $\suppi(X_J)=J$ (that is $\suppi(.)$ gives the index set for all the coordinates of a vector, even those that are $0$), and $|X_J|$, the dimension of $X_J$, is simply  $|J|$. 
 Similarly, if $H$ is a subspace of $\F_p^n$ then $H|_J$ or $H_J$ is the subspace spanned by $X|_J$ for $ X\in H$.  Finally, for $I, J \subset [n]$, the matrix $M_{I \times J}$ is the submatrix of the rows and columns indexed from $I$ and $J$ respectively. Sometimes we will also write $M_n$ for $M_{n\times n}$ if there is no confusion. Sometimes, for a matrix $M$ we write $\row_i(M)$ and $\col_i(M)$ for the $i$-th row and column respectively.
 
 For a prime $p$, we write $\F_p$ for the finite field with $p$ elements.
 For a matrix $M$ with coefficients in $\Z$, we write $M/p$ for the matrix with coefficients in $\F_p$ obtained from $M$ by reduction mod $p$.

\vskip .05in

{\bf Group Theory}: The exponent of a finite abelian group is the smallest positive integer $a$ such that $aG=0$. For a prime $p$, a finite abelian $p$-group is isomorphic to 
$\bigoplus_{i=1}^r \Z/p^{\lambda_i}\Z$ for some positive integers $\lambda_1\geq \lambda_2\geq\dots\geq \lambda_r$.
We call the partition $\lambda$ the \emph{type} of the abelian $p$-group. 
The {symmetric power} $\Sym^2 G$ is defined to be the quotient of $G\tensor G$ by the subgroup generated by elements of the form $g_1\tensor g_2-g_2\tensor g_1$.
If $G_p$ is type $\lambda$, generated by $e_i$ with relations $p^{\lambda_i}e_i=0$, then
$\Sym^2 G_p $ is generated by the $e_i  e_j$ for $i\leq j$ with relations
$p^{\lambda_j} e_i e_j=0$.
Similarly, $\wedge^2 G$ is defined to be the quotient of $G\tensor G$ by the subgroup generated by elements of the form $g\tensor g$, and $\wedge^2 G_\lambda=
\oplus_{i} (\Z/p^{\lambda_i} \Z)^{ i-1}$.  We write $\langle g_1,\dots \rangle$ for the subgroup generated by $g_1,\dots$.

{\bf Random groups}: For a random group $X$, the \emph{$G$-moment} of $X$ is $\E(\#\Sur(X,G))$, where $\Sur(X,G)$ denotes the surjective group homomorphisms from $X$ to $G$.

\subsection{Laplacian sampling}

We now introduce introduce various equivalent models of Laplacian matrices to be used.

{\bf Laplacian models:} for undirected graphs we will sample the Laplacian as follows.

\begin{definition}[Laplacian for random graphs]\label{def:lap} The model $L_n=(L_{ij})$ can be obtained via two phases of randomness.
\begin{itemize}
\item Phase 1: Assume that the vertices of $G=G_{n+1}$ are ordered as $\CO=(v_1,\dots, v_n,v_{n+1})$. We first sample $x_{ij}$ (via $G(n+1,1/2)$) and compute the degrees $d_1,\dots, d_{n+1}$ of the vertices, and then subtract those from the diagonals of $M_G$ to form $L_G$. We delete the last row and column to form  $L_n$. Notice that hence the column vectors of $L_n$ are not necessarily orthogonal to $\1=(1,\dots,1)$. 
\vskip .1in
\item Phase 2: Given an ordering $\CO'=(v_1',\dots, v_n',v_{n+1})$ of the vertices, we reshuffle the neighbors of the consecutive vertex pairs adapted to $\CO'$ as follows. 
For $N=0,\dots, n-1$,
we make the following modifications to the graph.
We  consider the set $I_{N}$ of indices $1\le j\le N$ (or vertices $v_i'$ from $\{v_1',\dots, v_{N}'\}$) where $v_j'$ is connected to {\it exactly one} of $v_{N+1}'$ or $v_{N+2}'$. Then for each $j \in I_{N}$, 
we flip a fair coin to either keep or swap whether each of $(v_j',v_{N+1})$ and $(v_j',v_{N+2})$ is an edge (see Figure \ref{figure:swap}).
 In other words, if $X=(x_1,\dots,x_{N})$ and $Y=(y_1,\dots, y_{N})$ are the (restricted) column vectors associated to $v_{N+1}'$ and $v_{N+2}'$, then $I_{N}$ is the collection of indices $j$ where $(x_j,y_j) =(0,1)$ or $(1,0)$. We then flip a fair coin to reassign $(x_j,y_j)$ to $(0,1)$ or $(1,0)$. We iterate this process until $N=n$, and call this model $ L_n(\CO')$.
\end{itemize}
\end{definition}

\begin{definition}
Let $L_n$ be the random matrix (reduced Laplacian) defined above in terms of the random integers $x_{ij}$.  For $A,B$ disjoint subsets of $[n+1]$, and an involution $b\mapsto \bar{b}$ on $B$,
we define the \emph{$AB$-shuffle} of $L_n$ to be the $n\times n$ matrix $L'_n$ with entries as follows.
For each $i\in A$, we (independently) make a random set $B_i$ by putting  each orbit of $B$ under the involution in $B_i$ independently with probability $1/2$.  Then, for $1\leq i\ne j \leq n+1$ we define random integers
 $x'_{ij}=x_{i\bar{j}}$ for $i\in A$ and $j\in B_i$; and
  $x'_{ij}=x_{\bar{i}j}$ for $j\in A$ and $i\in B_j$; and $x'_{ij}=x_{ij}$ for all other $i\neq j$.
  We then define $L'_n$ as in \eqref{eqn:redLap} with the $x'_{ij}$ replacing the $x_{ij}$.
\end{definition}

One can see that for any $A,B,$ the $AB$-shuffle of $L_n$ has the same distribution as $L_n$ by  considering the probability of obtaining any particular matrix.
Critically for our applications, the $AB$-shuffle leaves all the diagonal entries of the $([n]\setminus B)\times([n]\setminus B)$ submatrix of $L_n$ fixed.    
We can use the same definition to apply the $AB$-shuffle to our other models $M_n$ and $A_n$, and we sometimes do in order to give a proof for all three cases at once, but the shuffle is never necessary in these cases.

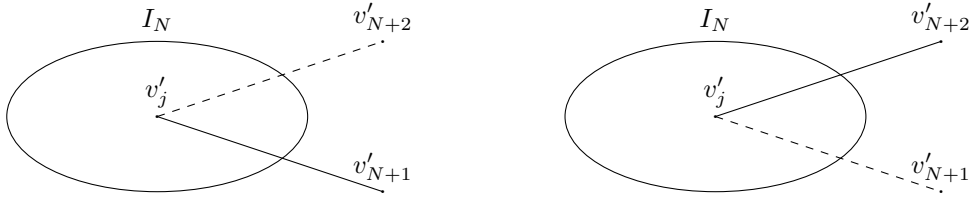
\begin{figure}\centering

\begin{tikzpicture}
\draw (2,2) ellipse (2cm and 1cm) node[above = .4in]{$I_{N}$};
\draw (2,2)  node{.} node[above]{$v_j'$} -- (5,1) node{.} node[above]{$v_{N+1}'$};
\draw [dashed] (2,2)  -- (5,3) node{.} node[above]{$v_{N+2}'$};

\end{tikzpicture}
\hfil
\begin{tikzpicture}
\draw (2,2) ellipse (2cm and 1cm) node[above = .4in]{$I_{N}$};
\draw  [dashed] (2,2)  node{.} node[above]{$v_j'$} -- (5,1) node{.} node[above]{$v_{N+1}'$};
\draw (2,2)  -- (5,3) node{.} node[above]{$v_{N+2}'$};

\end{tikzpicture}

\caption{Swapping neighbors.}
\label{figure:swap}

\end{figure}

To some extent, our reshuffling is similar to \cite{Mckay} and \cite{Cook,Tik} where shufflings/switchings were used within random graphs and random matrices. However our implementation here is rather straightforward. In Phase 1, by Chernoff's bound  
\begin{equation}\label{eqn:degreesq}
\P(\wedge_{i=1}^{n+1}(n/2 - t\sqrt{n} \le d_i \le n/2 + t\sqrt{n}) \ge 1-2n \exp(-2t^2), t>0.
\end{equation}
It is clear that the neighbor reshuffling process in Phase 2 does not change the distribution of $L_n$,  which we summarize below
\begin{fact}[Equivalence of models]\label{fact:EoM} The distribution of %
$L_n$ %
obtained in Phase 1, is the same as the distribution of $L_n(\CO')$
obtained in Phase 2, for any ordering $\CO'$. 
\end{fact}
\begin{proof} It suffices to show for $\CO' = (v_1',\dots, v_n',v_{n+1})$, 
that step $N$ of our reshuffling does not change the distribution of the Erd\H{o}s-R\'enyi graph $G(n+1,1/2)$ or its adjacency matrix.  
However, it is easy to see that conditioned on the upper-left $N\times N$ submatrix of the adjacency matrix,  after the shuffling at step $N$,
 every possible upper-left $(N+2)\times (N+2)$ submatrix of the adjacency matrix is equally likely, which proves the claim.
\end{proof}
Additionally, we will also rely on the following Chernoff's bound which says that the degrees are near $n/2$
and the reshuffling process significantly creates extra randomness,.  For $c>0$, we have
\begin{equation}\label{eqn:concentration} 
\P\Big(\wedge_{k \ge c n} |I_k| \in (k/2 - t\sqrt{n}, k/2+ t\sqrt{n}) \Big) \ge 1- n\exp(-\Theta_c(t^2)), t>0.
\end{equation}

In later applications we will choose either $t=c\sqrt{n}$ or $t=n^{1/2-c}$ for a small constant $c$. To end the discussion, in Section \ref{section:largeprimes} we will also make use of the following model (of random graphs of given degree sequence). Let $c>0$ be a constant, and assume that $(1/2-c)n \le d_i \le (1/2+c)n$ for all $1\le i\le n+1$. The number of such degree sequences $\Bd$ is simply bounded by $(2c n)^{n+1}$.  Thus with $\CG_{\Bd}$ being the collection of all simple graphs on $n+1$ vertices with degrees $\Bd$, we have 
$$\sum_{\Bd, \P(G(n+1,1/2) \in \CG_{\Bd}) < n^{-4n} }\P\big(G(n+1,1/2) \in \CG_{\Bd}\big) \le (2c n)^{n+1}n^{-4n} \le n^{2n}.$$
Hence, conditioning on the event $(1/2-c)n \le d_i \le (1/2+c)n$ for all $1\le i\le n+1$, it is natural to only focus on degree sequences $\Bd$ where 
$$\P\big(G(n+1,1/2) \in \CG_{\Bd}\big) \ge n^{-4n}.$$ 
Let $\mathcal{D}_g=\mathcal{D}_g(c,n)$ denote the collection of such degree sequences.

\begin{definition}[Laplacian for random graphs of good degree sequence]\label{def:lap'} Let $c>0$ be given sufficiently small. 
\begin{itemize}
\item Phase 1: we will choose each $\Bd$ from the degree sequence $\mathcal{D}_g$ with probability 
$$\P(\Bd \text{ is sampled}) = \frac{\P\big(G(n+1,1/2)\in \CG_{\Bd}\big)}{\sum_{\Bd' \in \mathcal{D}_g} \P\big(G(n+1,1/2) \in \CG_{\Bd'}\big)}.$$ 
\vskip .1in
\item Phase 2:  we sample a random graph $G$ from $\CG_{\Bd}$ with probability 
$$\frac{\P\big(G(n+1,1/2)= G\big)}{\P\big(G(n+1,1/2) \in \CG_{\Bd}\big)}.$$ 
\item We then obtain $L_n$ from the Laplacian of $G$ as usual. 
\end{itemize}
\end{definition}

\section{Proof of the global statistics from the main inputs}\label{section:method}

In this section we give the proofs of Theorems \ref{theorem:cyclic:sym}, \ref{theorem:cyclic:alt} and \ref{theorem:cyclic:L_G}
from their main inputs, which will then be the focus of the rest of the paper. 
As mentioned, for the small primes  we will use Theorems \ref{theorem:W:sym} and \ref{theorem:W:lap} on symmetric matrices and  graph Laplacians, and we will prove
 Theorem \ref{theorem:W:alt} on skew-symmetric matrices in Section~\ref{section:small:alt:2}.  For the moderate primes and for the symmetric and Laplacian models we show the following.  
\begin{prop}\label{prop:moderate:sym}\label{prop:moderate:lap} Let $G_n$ be either $M_n$ as in Theorem~\ref{theorem:W:sym} or $L_n$ as in Theorem~\ref{theorem:prodcyc:lap}. There exists a sufficiently small $c>0$  such that the following holds for every prime  $p \le e^{n^c}$ 
\begin{equation}\label{eqn:medium}
\P(\rk(G_n/p)\le n-2) =  O(\frac{1}{p^3} +e^{-n^{c}}).
\end{equation}
\end{prop}

The behavior for the skew-symmetric case is slightly different, as the ranks are always even.
\begin{prop}\label{prop:moderate:alt} 
Let $A_n$ be as in Theorem~\ref{theorem:W:alt}.
There exists a sufficiently small $c>0$ such that the following holds for every prime $p \le e^{n^c}$: 
\begin{equation}\label{eqn:medium:alt}
\P(\rk(A_n/p)\le n-3) =  O(\frac{1}{p^3} +e^{-n^{c}}).
\end{equation}
\end{prop}

 Even such a small error bound cannot be summed over all primes, and so for very large primes we consider all primes together.  Here we separate into  two cases, for the symmetric and Laplacian model we show the following.

\begin{prop}\label{prop:large:sym}\label{prop:large:lap} Let $G_n$ be either $M_n$ as in Theorem~\ref{theorem:W:sym} or $L_n$  as in Theorem~\ref{theorem:prodcyc:lap}. For any given $C,c>0$ the following holds for sufficiently large $n$ 
\begin{equation}\label{eqn:large:lap}  
\P\Big(\exists \mbox{ prime } p\geq e^{n^{c}} : \rk(G_n/p)\le n-2
\Big) \le n^{-C}.
\end{equation}
\end{prop}
For the skew-symmetric case we show the following. 
\begin{prop}\label{prop:large:alt} Let $A_n$ be as in Theorem~\ref{theorem:W:alt}. For any given $C,c>0$, the following holds for sufficiently large $n$
\begin{equation}\label{eqn:large:alt}  
\P\Big(\exists \mbox{ prime } p\geq e^{n^{c}} :  \rk(A_n/p)\le n-3
\Big) \le n^{-C}.
\end{equation}
\end{prop}

Propositions \ref{prop:moderate:sym} and \ref{prop:moderate:alt} will follow from our work in Sections \ref{section:lemmas}, \ref{section:normalvector}, and \ref{section:rankevolving}, while Propositions \ref{prop:large:sym} and \ref{prop:large:alt} will be justified in Sections  \ref{section:bilinear} and \ref{section:largeprimes}.
We will now show how these results can be combined to obtain our main results.  

\begin{proof}[Proof of Theorems ~\ref{theorem:cyclic:L_G} and ~\ref{theorem:cyclic:sym}] We will use the fact that $\cok(G_n)$ (where $G_n$ is either the symmetric $M_n$ or the Laplacian $L_n$) is cyclic if and only if for every prime $p$, the matrix $G$ mod $p$ has rank at least $n- 1$. 
Fix an integer $k_0$.
By Proposition \ref{prop:moderate:lap}, we have
\begin{align*}
\P\Big( \rk(G_n/p)\le n-2
\mbox{ for some $k_0 \le p \le e^{n^{c/2}}$ } \Big) &= \sum_{p=k_0}^{e^{n^{c/2}}} O(p^{-3} +e^{-n^{c}}) \\  
&= O(\frac{1}{k_0^2}+e^{-n^{c/2}}).
\end{align*} 
Combined  with Proposition~\ref{prop:large:lap} we obtain
$$\P\Big(  \rk(G_n/p)\le n-2
\mbox{  for some }  p\geq k_0 \Big) =  O( \frac{1}{k_0^2}+e^{-n^{c/2}}+ n^{-C}).$$

Let $P$ be the set of primes $<k_0$, and let $\mathcal{C}$ be the set of all cyclic abelian $P$-groups.
We note that the probabilities in Theorems~\ref{theorem:W:sym} and \ref{theorem:W:lap} sum to $1$ over all $P$-groups (e.g., from \cite[Proposition 7]{Clancy2015}
and the orbit-stabilizer theorem as in the proof of \cite[Corollary 9.2]{Wood2017}).
Thus, as in Remark~\ref{R:odd1}, we can determine the asymptotic probability that $(\cok (G_n))_P\in \mathcal{C}$ by summing the probabilities of Theorem~\ref{theorem:W:sym} over all groups in $\mathcal{C}$, which is done in \cite[Proposition 9]{Clancy2015} (note the sum factors over primes $p$).  
So we conclude, 
\begin{equation}\label{eqn:B:comp}
\lim_{n\ra\infty} \P((\cok (G_n))_P\in \mathcal{C})=\prod_{p<k_0} \prod_{i=1}^\infty (1-p^{-2i-1}).
\end{equation}

As this is true for any fixed $k_0$, we can take $k_0\ra\infty$  and combine with the bounds for $p\ge k_0$ to obtain
$$ \liminf_{n \to \infty} \P\Big(\cok(G_{n}) \mbox{ is cyclic}\Big) \geq   \prod_{p \textrm{ prime}} \prod_{i=1}^\infty (1-p^{-2i-1}) =  \prod_{i=1}^\infty \zeta(2i+1)^{-1}.$$
The upper bound for $\limsup$ follows from \eqref{eqn:B:comp}. 
\end{proof}

  \begin{proof}[Proof of Theorem~\ref{theorem:cyclic:alt}]  
\begin{comment}  
   For the odd case, because $\cok(A_{2n+1})$ is always of the form $\Z^{2m+1} \times G \times G$, it is cyclic if and only if it is $\Z$, that is for every prime $p$, the matrix $A_{2n+1}$ mod $p$ has rank $2n$ (equivalently, at least $2n-1$ because the rank is always even). 
  Let $k_0$ be sufficiently large. By Proposition \ref{prop:moderate:alt}, for $n$ large enough we have
\begin{align*}
\P\Big(A_{2n+1} \mbox{ mod $p$ has rank at most $2n-2$ for some $k_0 \le p \le e^{n^{c/2}}$ } \Big) &\le \sum_{p=k_0}^{e^{n^{c/2}}} O(p^{-3} +e^{-n^{c}}) \\  
&= O(\frac{1}{k_0^2}+e^{-n^{c/2}}).
\end{align*} 
Combined  with Eq.~\eqref{eqn:large:alt} we obtain
$$\P\Big(A_{2n+1} \mbox{ mod $p$ has rank at least $2n-1$ for all $ p\geq k_0 $}\Big) \ge 1-  O( \frac{1}{k_0^2}+e^{-n^{c/2}}+ n^{-C}).$$
Now let $S_{k_0}$ be the sum of the Sylow $p$-subgroups of $\cok(A_{2n+1})$ for $p<k_0$. By Theorem \ref{theorem:W:alt} and  by \cite{C,Mac}
$$\lim_{n\to \infty} \P((S_{k_0})_p \simeq \Z/p\Z) =\prod_{i=1}^\infty (1-p^{-2i-1}).$$
Apply Theorem \ref{theorem:W:alt} again for $P=\prod_{p<k_0} p$ and take $k_0\ra\infty$, and then combine with the bounds for $p\ge k_0$ to obtain
$$ \liminf_{n \to \infty} \P\Big(\cok(A_{2n+1}) \mbox{ is cyclic}\Big) \geq   \prod_{p \textrm{ prime}} \prod_{i=1}^\infty (1-p^{-2i-1}) =  \prod_{i=1}^\infty \zeta(2i+1)^{-1}.$$
One can easily see that the upper bound for $\limsup$ follows from Theorem \ref{theorem:W:alt}, and hence we obtain the desired asymptotic statistics for the odd case.

We argue similarly for the even case. 
\end{comment}

The group $\cok (A_{2n})$ is of the form $\Z^{2m} \times G$, for some $G\in\mathcal{S}$.
The even rank follows from the fact that the rank of a skew-symmetric matrix is even.  The condition on the torsion part follows from the fact that the torsion has a skew-symmetric non-degenerate perfect pairing (\cite[Sections 3.4 and 3.5]{Bhargava2015b} and \cite[Proposition 2]{Delaunay2001}). 
Thus $\cok(A_{2n})$
is the square of a cyclic group if and only if for every prime $p$, the matrix $A_{2n}$ mod $p$ has rank at least $2n-2$. 
The rest of the proof is exactly like the proof of Theorems ~\ref{theorem:cyclic:sym} and ~\ref{theorem:cyclic:L_G} above, using
Theorem \ref{theorem:W:alt} and Propositions~\ref{prop:moderate:alt} and \ref{prop:large:alt} as input.
To see that the probabilities in Theorem~\ref{theorem:W:alt} sum to $1$, we can use an argument as in \cite[Theorem 9]{Delaunay2001}, but only taking the product over finitely many primes.  Then, to compute the sum of the probabilities over squares of cyclic $P$-groups, we can reason as in \cite[Example E]{Delaunay2001}.

The group $\cok (A_{2n+1})$ is of the form $\Z^{2m+1} \times G$, for some $G\in\mathcal{S}$, by the same reasoning as in the even dimensional case.
 If $P$ is the set of primes $<k_0$, where we take $k_0$ larger than the largest prime dividing $C$, then we have that $\cok (A_{2n+1})\isom \Z \times C$
 if and only if $(\cok_{\operatorname{tors}} (A_{2n+1}))_P\isom C$ and for all primes $p\geq k_0$ we have 
$\rank ( A_{2n+1}/p)\geq (2n+1)-2$.   Then the proof follows as above, using Theorem \ref{theorem:W:alt},
 Propositions~\ref{prop:moderate:alt} and \ref{prop:large:alt}, and \cite[Theorem 9]{Delaunay2001}.
 \begin{comment}
Let $k_0$ be sufficiently large. By Proposition \ref{prop:moderate:alt}, for $n$ large enough we have
\begin{align*}
\P\Big(A_{2n} \mbox{ mod $p$ has rank at most $2n-3$ for some $k_0 \le p \le e^{n^{c/2}}$ } \Big) &= \sum_{p=k_0}^{e^{n^{c/2}}} O(p^{-3} +e^{-n^{c}}) \\  
&= O(\frac{1}{k_0^5}+e^{-n^{c/2}}).
\end{align*} 
Combined  with Eq.~\eqref{eqn:large:alt} we obtain
$$\P\Big(A_{2n} \mbox{ mod $p$ has rank at least $2n-2$ for all $ p\geq k_0 $}\Big) \ge 1-  O( \frac{1}{k_0^2}+e^{-n^{c/2}}+ n^{-C}).$$
Now let $S_{k_0}$ be the sum of the Sylow $p$-subgroups of $\cok(A_{2n})$ for $p<k_0$. By Theorem \ref{theorem:W:alt} and  by \cite{C,Mac} 
$$ \P((S_{k_0})_p \simeq 1)+ \P((S_{k_0})_p \simeq (\Z/p\Z)^2)  =\frac{\prod_{i=0}^{\infty}(1-p^{-2i-1}) }{ 1}  
+p^{-1}\frac{\prod_{i=1}^{\infty}(1-p^{-2i-1}) }{ (1-p^{-2})}.$$
Apply Theorem \ref{theorem:W:alt} again for $P=\prod_{p<k_0} p$, take $k_0\ra\infty$, and then combine with the bounds for $p\ge k_0$ 
\begin{align*}
\liminf_{n \to \infty} \P\Big(\cok(A_{2n}) \mbox{ is square of cyclic}\Big) &\geq   \prod_{p \textrm{ prime}}  \frac{\prod_{i=0}^{\infty}(1-p^{-2i-1}) }{ 1}  
+p^{-1}\frac{\prod_{i=1}^{\infty}(1-p^{-2i-1}) }{ (1-p^{-2})}\\
&  =  \frac{\zeta(2)}{\prod_{i=1}^\infty \zeta(2i+1)}   \prod_{p \textrm{ prime}} (1-p^{-2}+p^{-3}  ).
\end{align*}
To this end, the upper bound for $\limsup$ follows from Theorem \ref{theorem:W:alt} for $\cok(A_{2n})$, and hence we have the identity in the $n \to \infty$ limit.
\end{comment}
\end{proof}

The proofs of Theorem \ref{theorem:prodcyc:sym}, Theorem \ref{theorem:prodcyc:alt}, and Theorem \ref{theorem:prodcyc:lap} are very analogous.

\section{Treatment for small primes: determination of the moments}\label{section:small:alt:1}

We will prove Theorem~\ref{theorem:W:alt} by finding the moments (as random groups, see \cite[Section 3.3]{Clancy2015}) of the distribution  $\cok(A_{2n})$, in the following theorem.
\begin{theorem}\label{T:almom}
Let $A_n$ be as in Theorem~\ref{theorem:W:alt}, and $G$ be any finite abelian group.  We have
$$
\lim_{n\ra\infty} \E(|\Sur(\cok(A_{n}),G)|)=|\Sym^2 G|.
$$
\end{theorem}

However, these moments are exactly large enough that they do not determine a unique distribution.  In particular, the moments
do not see whether $n$ is even or odd, yet we know the distributions of $\cok(A_{n})$ are quite different in these cases because they are usually finite groups when $n$ is even
and always infinite groups when $n$ is odd.  So we will prove a new theorem on the moment problem for finite abelian groups to show that when we take into account this further information, that a unique distribution is determined by the moments.

\subsection{Proof of Theorem~\ref{T:almom}} 
In fact, we will prove the rate of convergence in Theorem~\ref{T:almom} is exponential in $n$.
The proof of this follows the proof of \cite[Theorem 1.2]{Wood2017} closely, which is the analogous result for symmetric matrices. 
Only small modifications are required, so we will be brief.

Let $a$ be a positive integer and let $G$ be a finite abelian group with $aG=0$. Let $R$ be the ring $\Z/a\Z$. 
For an $R$-module $A$, let $A^*:=\Hom(A,R)$.  
We define the $R$-module $V=R^n$, with a distinguished basis $v_1,\dots,v_n$ of $V$, and a dual basis $v_1^*,\dots, v_n^*$ of $V^*$.
We have 
\begin{align}\label{E:momsum}
\E(|\Sur(\cok(A_n),G  |)=\sum_{F\in \Hom(V,G)} \P(FA_n=0).
\end{align}
%
%
%
%
%
%
\begin{comment}
Then $F$ descends to a homomorphism from $V/col_V(M)$ if and only if $col_V(M)\sub \ker(F)$, or equivalently if the composite $FM\in \Hom(W,G)$ is $0$.
So, 
$$
\# \Hom(V/col_V(M),G) = \sum_{F\in \Hom(V,G)} 1_{FM=0},
$$
and similarly,
$$
\# \Sur(V/col_V(M),G) = \sum_{F\in \Sur(V,G)} 1_{FM=0},
$$
Suppose that $p^e G=0$, i.e. $G=\bigoplus_{i=1}^{r} \Z/p^{\lambda_i}\Z$ and $e\geq \lambda_1
\geq \lambda_2\dots \geq \lambda_r$.
\end{comment}
Let $\xi$ be a primitive $a$th root of unity. %
We view $A_n$ as an element of $\Hom(V^*,V)$, and have
$$
\P(FA_n=0) =\frac{1}{|G|^n} \sum_{C\in \Hom(\Hom(V^*,G),R)) } \E (\xi^{C(FA_n)}).
$$
We have a natural isomorphism $\Hom(\Hom(V^*,G),R))\isom \Hom(V,G^*)$ and so we often view $C$ in this latter group.
We write $e: G^* \times G\ra R$ for the map that evaluates a homomorphism.  
Since $A_n$ is skew-symmetric, we have
\begin{align*}
&C(FA_n)=\sum_{i=1}^n \sum_{j=1}^n 
e(C(v_j) ,F(v_i) ) x_{ij}\\&=\sum_{i=1}^n \sum_{j=i+1}^n 
(e(C(v_j) ,F(v_i) ) -e(C(v_i) ,F(v_j) )  )x_{ij} .
\end{align*}
For $i<j$ we define, 
$E(C,F,i,j):=e(C(v_j) ,F(v_i) ) -e(C(v_i) ,F(v_j) )$, and so
\begin{equation}\label{E:factor}
\P(FA_n=0) =\frac{1}{|G|^n} \sum_{C\in \Hom(\Hom(V^*,G),R)) } \prod_{i\leq j} \E (\xi^{E(C,F,i,j)x_{ij}}).
\end{equation}
We will show many of these $E(C,F,i,j)$ coefficients are non-zero.

For $F\in \Hom(V,G)$ and $C\in \Hom(V,G^*)$, we have a map
 $\phi_{F,C} \in \Hom(V, G \oplus G^*)$ given by adding $F$ and $C$.
Similarly, we have a map $\phi_{C,F} \in \Hom(V, G^* \oplus G)$ given by adding $C$ and $F$.
Note has $V$ has distinguished submodules $V_\sigma$ generated by the $v_i$ with $i\not \in \sigma$ for each $\sigma\sub [n]$.  So $V_\sigma$ comes from not using the coordinates in $\sigma$.
Clearly, for any submodule $U$ of $V$, 
$$\ker(\phi_{F,C}|_{U} )\sub \ker(F|_{U} ).$$
Now we will define the key structural properties of $C$ and $F$ that determines if enough of the coefficients $E(C,F,i,j)$ are non-zero.

\begin{definition}
Let $0<\dc<1$ be a real number which we will specify later in the proof.
Given $F$, we say $C$ is \emph{robust} (for $F$)  if for every $\sigma\sub[n]$ with $|\sigma|<\dc n$,
$$
\ker(\phi_{F,C}|_{V_\sigma} )\ne \ker(F|_{V_\sigma} ).
$$
Otherwise, we say $C$ is \emph{weak} for $F$. 

We say that $F\in \Hom(V,G)$ is a \emph{code} of distance $w$, if for every $\sigma\sub [n]$ with $|\sigma|<w$, we have $FV_\sigma=G$.
In other words, $F$ is not only surjective, but would still be surjective if we throw out (any) fewer than $w$ of the standard basis vectors from $V$.  
\end{definition}

For $C=0$, of course all the $E(C,F,i,j)$ are $0$.  However, given $F$, there are other $C$ for which this can happen, and next we will identify those $C$.
Given an $F\in\Hom(V,G)$, we have a map
\begin{align}\label{E:mexplicit}
m_F :\Hom(V,G^*) &\ra \wedge^2 V^*\notag\\
C &\mapsto  \sum_{i=1}^n \sum_{j=i+1}^n (e(C(v_j) ,F(v_i) ) -e(C(v_i) ,F(v_j) )  )v_i^* \wedge v_j^* .
\end{align} 
We now determine some elements $C\in \Hom(V,G^*)$ that are in the kernel of $m_F$, i.e. all the $E(C,F,i,j)$ are $0$.  The following construction corrects the construction in \cite{Wood2017} and translates it to the skew-symmetric case.
Let $\Sym(G\tensor G, R)$ denote the subset of $\Hom(G\tensor G, R)$ that are symmetric, i.e. $\alpha$ such that $\alpha (g_1\tensor g_2)=\alpha(g_2\tensor g_1)$ for all $g_1,g_2\in G$.  
So we have a map
\begin{align*}
s_F: \Sym(G\tensor G, R) &\ra \Hom(V,G^*)\\
\alpha &\mapsto \left(v\mapsto (g\mapsto  \alpha(F(v)\tensor g) )  \right).
\end{align*}
Using Equation~\eqref{E:mexplicit}, we will check that  $\im(s_F)\sub \ker (m_F)$.
The $v_a^* \wedge v_b^*$ coefficient of $m_F(s_F(\alpha))$ is %
\begin{align*}
&e(C(v_b) ,F(v_a) ) - e(C(v_a) ,F(v_b) ) = 
\alpha(F(v_b)\tensor F(v_a) ) - \alpha(F(v_a)\tensor F(v_b) )
=0.
\end{align*}
We call the $C$ in  $\im(s_F)$ \emph{special} for $F$.

\begin{lemma}\label{L:sfinj}
If $FV=G$, then we have that $s_F$ is injective.  In particular, $\#\Sym^2 G| \#\ker(m_F)$.
\end{lemma}

\begin{proof}
Note $|\Sym(G\tensor G, R)|=|\Sym^2 G|$. %
It suffices to show that $\#\Sym^2 G| \#\im(s_F)$. 
Since everything in sight can be written as a direct sum of Sylow $p$-subgroups, we can reduce to the case that $G$ is a $p$-group of type $\lambda$ (and accordingly assume $R=Z/p^e\Z$). Let $r=\lambda_1'$.

By \cite[Lemma 3.4]{Wood2017}, we can find $\tau\sub [n]$ with $|\tau|=r$ such that $Fv_i$ generate $G$ for $i\in\tau$.  Let $W$ be the submodule of $V$ generated by the $v_i$ for $v\in \tau$.
Let $e_i$ generate $G$ with relations $p^{\lambda_i} e_i=0$.
Let $w_j\in W$ be such that $Fw_j=e_j$.
Let $W'\sub W$ be the subgroup of $W$ generated by the $w_j$.  As in \cite[Lemma 3.6]{Wood2017}, by Nakayama's Lemma we have that $W'=W$.
Since the $r$ elements $w_1,\dots, w_r$, generate the free rank $r$ $R$-module $W$, they must be a basis, and we have a dual basis $w_i^*$ of $W^*$.

We have that $\Hom(G\tensor G, R)$ is generated by $e_{ij}^*$ with relations
$p^{\min(\lambda_i,\lambda_j)}e_{ij}^*$ and $e_{ij}^*(e_a\tensor e_b)$ 1 if $a=i$ and $b=j$ and $0$ otherwise.  Also,
$\Sym(G\tensor G, R)$ is generated by $e_{ij}^*+ e_{ji}^*$ for $i<j$ and $e_{ii}^*$.
Let $G^*$ be generated by $e_1^*, \dots ,e_r^* $ with relations $p^{\lambda_i}e_i^*=0$, 
and such that  $e_i^* e_i =p^{e-\lambda_i}$, and for $i\ne j$ we have $e_i^* e_j =0 $.

Recall we have
$$
s_F: \Sym(G\tensor G, R) \ra \Hom(V,G^*).
$$
We can take the further quotient 
$$
s'_F: \Sym(G\tensor G, R) \ra \Hom(W,G^*).
$$
We see that by the definition of $s_F$
$$
s_F(e_{ij}^*+ e_{ji}^*)(w_a)(e_b)
= (e_{ij}^*+ e_{ji}^*)(e_a\tensor e_b)
$$
and
$$
s_F(e_{ii}^*)(w_a)(e_b)
= (e_{ii}^*)(e_a\tensor e_b)
$$
Recall that since $W$ is a free $R$-module, the natural map $W^* \tensor G^* \ra \Hom(W,G^*)$ is an isomorphism.  
By noting the values on each $w_a$ and $e_b$ above, we can confirm that for $i<j$
$$
s'_F(e_{ij}^*+ e_{ji}^*)= w_i^* \tensor  e_j^* + p^{\lambda_i-\lambda_j} w_j^* \tensor e_i^*,.
$$
which has order $p^{\lambda_j}$.
Also, $$
s'_F(e_{ii}^*)= w_i^* \tensor  e_i^*.
$$
which has order $p^{\lambda_i}$.
We can conclude that
$$
p^{\lambda_1+2\lambda_2+2\lambda_3+\dots +r\lambda_r} \mid \#\im(s'_F) \mid \#\im(s_F).
$$
\end{proof}

We now give a good bound on the probability that a code descends to a map from the cokernel of a random matrix.

\begin{lemma}\label{L:FullFcode}
Given $\delta>0$, there there is a $c>0$
and a real number $K$ (depending only on $\alpha, \delta, a,$ and $G$)
  such that for $F\in \Hom(V,G)$  a code of distance $\delta n$ and $A\in\Hom(V^*,G)$, we have, for all $n$,
\begin{align*}
 \left|\P(FA_n=0) - |\Sym^2 G||G|^{-n}\right| &\leq
 \frac{K\exp(-cn)}{|G|^{n}}
\end{align*}
and
\begin{align*}
 \P(FA_n=A) \leq  K|G|^{-n}.
\end{align*}
\end{lemma}

\begin{proof}[Proof of Lemma~\ref{L:FullFcode}]
We closely follow the proof of \cite[Lemma 4.1]{Wood2017}.
We have
$$
\P(FA_n=A) =\frac{1}{|G|^{n}}\sum_{C\in \Hom(V,G^*) } \E (\xi^{C(FA_n-A)}),
$$
and we break the sum into $3$ pieces based on when $G$ is special, not special and weak, or robust.

Given $F$, there are $|\Sym^2 G|$ special $C$ for which $\xi^{C(FA_n)}=1$ for all $A_n$ by Lemma~\ref{L:sfinj}.
In the sum above, these $C$ contribute $|\Sym^2 G||G|^{-n}$ when $A=0$ and at most $|\Sym^2 G||G|^{-n}$ in absolute value for any $A$.

From \cite[Lemma 3.1]{Wood2017}, we have
that the number of $C\in \Hom(V,G^*)$ such that $C$ is weak for $F$ is at most
$
C_G \binom{n}{\lceil \dc n \rceil -1} |G|^{\dc  n }.
$
 If  $C$ is not special for $F$, we have
$$
|\E (\xi^{C(FA_n-A)})| =|\E (\xi^{C(-A)})\prod_{1\leq i < j\leq n} \E(\xi ^{E(C,F,i,j)x_{ij}}  ) |\leq \exp(-\alpha \delta n/(2a^2)).
$$
This follows because \cite[Lemma 3.7]{Wood2017} tells us there are at least $\delta n/2$ of the $E(C,F,i,j)$ are non-zero, and then \cite[Lemma 4.1]{Wood2017}
bounds those factors by $\exp(-\alpha /a^2)$.  (The proof of \cite[Lemma 3.7]{Wood2017} goes through for our definition of $m_F$ using Lemma~\ref{L:sfinj} in place of \cite[Lemma 3.6]{Wood2017}, and
with the roles of $\Sym^2$ and $\wedge^2$ being reversed through the proof.)

Now, given a robust $C$ for $F$,  \cite[Lemma 3.5]{Wood2017} gives a lower bound on the number of non-zero $E(C,F,i,j)$.
The proof of \cite[Lemma 3.5]{Wood2017} goes through in this setting as long as we modify the definition of the map $t$ in 
\cite{Wood2017} so that now $t((a_1,b_1),(b_2,a_2)=e(b_2,a_1)-e(b_1,a_2)$.  With this definition, the proofs of \cite[Lemma 3.2, Corollary 3.3, Lemma 3.5]{Wood2017} still hold. 
We then have that at least $\dc \delta n^2/(2|G|^2|P|)$ of the $E(C,F,i,j)$ are non-zero (where $P$ is the set of primes dividing $a$). 
 So if $C$ is robust for $F$, we conclude that
$$ 
|\E (\xi^{C(FA_n-A)})| \leq \exp(-\alpha \dc \delta n^2/(2|G|^2|P|a^2)).
$$

\begin{comment} 
In conclusion
\begin{align*}
%
 &\left|\P(FX=A) - \frac{1}{|G|^{n}} \sum_{C \in \Hom(V,G^*), \textrm{ special}}|\E (\xi^{C(FX-A)})\right| \\
&\leq
 \frac{1}{|G|^{n}} \sum_{C \in \Hom(V,G^*), \textrm{ not special}}| \E (\xi^{C(FX-A)})|\\
 &\leq
 \frac{1}{|G|^{n}} 
\left( 
 C_G \binom{n}{\lceil \dc n \rceil -1} |G|^{\dc  n }
%
 \exp(-\alpha \delta n/(2a^2))+ |G|^{n} \exp(-\alpha \dc \delta n^2/(2|G|^2|P|a^2))
 \right).
\end{align*}
%
\end{comment}
Putting these bounds together, for any $c>0$ such that $c< \alpha \delta /(2a^2)$, given
 given $\delta, \alpha, G,c$, we can choose $\dc$ sufficiently small so that we have
\begin{align*}
& \left|\P(FX=A) -  \frac{1}{|G|^{n}} \sum_{C \in \Hom(V,G^*), \textrm{ special}} \E (\xi^{C(FX-A)})\right| \\&\leq
 \frac{1}{|G|^{n}} \left(C_G \exp(-cn)  + \exp(\log(|G|)n-\alpha \dc \delta n^2/(2|G|^2|P|a^2)) \right),
\end{align*}
and the lemma follows.
\end{proof}

\begin{definition}
For an integer $D$ with prime factorization $\prod_i p_i^{e_i}$, let $\ell(D)=\sum_i e_i$.
 The \emph{depth} of an $F\in\Hom(V,G)$ is the maximal positive $D$ such that
there is a $\sigma\sub [n]$ with $|\sigma|< \ell(D)\delta n$ such that $D=[G:FV_\sigma]$, or is $1$ if there is no such $D$. 
\end{definition}

Now we will complete the proof of Theorem~\ref{T:almom}, using the sum in \eqref{E:momsum}.
We let $K$ change in each line, as long as it is a constant depending only on $\alpha,G,\delta,a$.
We then apply \cite[Lemmas 5.2, 5.4]{Wood2017} (whose proofs go through in the current case) to bound the number of $F$ of each depth and their corresponding probabilities of $FA_n=0$, and we 
obtain
\begin{align*}
\sum_{\substack{F\in \Sur(V,G)\\ F\textrm{ not  code of distance $\delta n$}
} }
\P(FA_n=0)
&\leq 
\sum_{\substack{D>1\\ D\mid\#G}} \sum_{\substack{F\in \Sur(V,G)\\ F\textrm{  depth $D$}}}
\P(FA_n=0) \\
&\leq \sum_{\substack{D>1\\ D\mid\#G}} 
K\binom{n}{\lceil \ell(D)\delta n \rceil -1} |G|^nD^{-n+\ell(D)\delta n}
  e^{-\alpha(1-\ell(D)\delta)n} (a|G|/D)^{-(1-\ell(D)\delta)n} \\ %
 &\leq  K   e^{-cn},
\end{align*}
for any $0<c<\alpha$, and $\delta$ chosen small enough in terms of $c$.
Similarly,
\begin{align*}
\sum_{\substack{F\in \Hom(V,G)\\ F\textrm{ not  code of distance $\delta n$}
} }
|\Sym^2 G||G|^{-n}
 &\leq  K 
 e^{-cn},
\end{align*}
for any $0<c<\log(2)$, and $\delta$ small enough in terms of $c$.
Using Lemma~\ref{L:FullFcode},
\begin{align*}
\sum_{\substack{F\in \Sur(V,G)\\ F\textrm{  code of distance $\delta n$}
} }
\left| \P(FA_n=0) -|\Sym^2 G| |G|^{-n} \right|
&\leq 
 Ke^{-cn} . 
\end{align*}
Combining these estimates, we conclude the proof of Theorem~\ref{T:almom}.

\section{Treatment for small primes: Moments determining the distribution}\label{section:small:alt:2}

Now we prove Theorem \ref{theorem:W:alt} by showing that the moments in Theorem~\ref{T:almom} determine unique distributions over certain restricted families of groups.  
When the moments are bounded by quantities only slightly smaller than those in Theorem~\ref{T:almom}, the result 
\cite[Theorem 8.3]{Wood2017}  shows that moments determine a unique distribution on finite abelian $P$-groups (see \cite{Wood2017} for some of the history of this moment problem).  However, the moments in Theorem~\ref{T:almom},
famously, do not determine unique distribution.  These moments arise as the distribution of moments of the predicted distribution of Selmer groups of random elliptic curves in the heuristics of Poonen and Rains \cite{Poonen2012}, further developed by Bhargava, Kane, Lenstra, Poonen and Rains  \cite{Bhargava2015b}.
There are two different distributions predicted depending on whether the parity of the elliptic curve is even or odd, and those are the distributions we see as $\cok (A_n)$ in the even and odd dimensional cases, respectively.  With a bit more information on the groups, we show that we can recover the distribution as follows.

\begin{theorem}\label{T:mainmomlimit}
Let $P$ be a finite set of primes, and $X_n$ and $Y_n$ random abelian $P$-groups for each integer $n$
either (1) all supported on groups in $\mathcal{S}_P$, or (2) all
supported on groups of the form $\Z/a\Z \times G \times G$, for some integer $a$ and all abelian $P$-groups $G$ with $aG=0$.
If there is a constant $C$ such that for every finite abelian $P$-group $A$ we have
$$
\lim_{n\ra\infty} \E(\#\Sur(X_n,A))=\lim_{n\ra\infty} \E(\#\Sur(Y_n,A))\leq C |\Sym^2 A|,
$$
then for every finite abelian $P$-group $A$ we have $\lim_{n\ra\infty} \P(X_n\isom A)=\lim_{n\ra\infty} \P(Y_n\isom A).$
\end{theorem}

This theorem can be shown in few different ways, and the following convenient argument is based on  work of the second author with W. Sawin.

\begin{proof}
A finite abelian $P$-group $G$ is in $\mathcal{S}_P$ if and only if for each prime $p\in P$, if the Sylow $p$-subgroup $G_p$ is of type $\lambda$, then all the $\lambda_i'$ are even.  Similarly, a finite abelian $P$-group is of the form in the second condition if and only if for each $p\in P$, if the Sylow $p$-subgroup $G_p$ is of type $\lambda$, we have that $\lambda_1=e$ where $p^e$ is the maximal power of $p$ dividing $a$, and all the non-zero $\lambda_i'$ are odd. Let $\mathcal{T}$ be the set of groups under consideration (either $\mathcal{S}_P$ or the groups described in the second case).  

Given a finite abelian group $G\in \mathcal{T}$, we can determine the isomorphism type of $G$ from any quotient of $G$ by one element.
To see this, we note that the Hall polynomial $g^{\lambda}_{\mu,\nu}(p)$ counts the number of subgroups of type $\mu$ of the finite abelian $p$-group of type $\lambda$ where the quotient has type $\nu$.  A subgroup generated by one element will correspond to a $\mu$ with $\mu_2=0$.  The Hall polynomial is zero, i.e. there are no such subgroups, when the Littlewood-Richardson coefficient $c^{\lambda}_{\mu,\nu}=0$ \cite[II:(4.3)]{Macdonald2015}.  
By the Pieri rule, we have that if $\mu_2=0$, then $c^{\lambda}_{\mu,\nu}=0$ whenever $\lambda'_i-\nu'_i\geq 2$ for some $i$.
Thus given any partition $\nu$, there is at most one partition $\lambda\in \mathcal{T}$  %
such that a group of type $\nu$ is a quotient of q group of type $\lambda$ by one element.  
Let $\mathcal{A}_G$ be the (finite) set of all finite abelian groups that are a quotient of $G$ by one element.

Let $Z$ be a random group and $z$ a uniform random element of $Z$.  Note  that for every finite abelian $P$-group $A$, we have
$
\E(\#\Sur(Z/\langle z\rangle ,A   ))=\E(\#\Sur(Z ,A   ))/|A|,
$
since or each surjection $Z\ra A$, there is a $|A|^{-1}$ probability that the image of $z$ is $0$.  Let $x_n,y_n$ be uniform random elements of $X_n,Y_n$ respectively.
Then for every finite abelian $P$-group $A$, we have
$$
\lim_{n\ra\infty} \E(\#\Sur(X_n/\langle x_n\rangle,A))=\lim_{n\ra\infty} \E(\#\Sur(Y_n/\langle y_n\rangle,A))\leq C |\Sym^2 A|/|A|= C|\wedge^2 A|.
$$
These moments do grow slowly enough to determine a unique distribution.
By the proof of \cite[Theorem 8.3]{Wood2017}, we have that for all $A$
$$
\lim_{n\ra\infty} \P(X_n/\langle x_n\rangle\isom A)=\lim_{n\ra\infty} \P(Y_n/\langle y_n\rangle\isom A),
$$
 and thus since 
 $$
\lim_{n\ra\infty} \P(X_n\isom A)=\sum_{B\in \mathcal{A}_G} \lim_{n\ra\infty} \P(X_n/\langle x_n\rangle\isom A),
$$
and similarly for $Y_n$, we have
 $$
\lim_{n\ra\infty} \P(X_n\isom A)=\lim_{n\ra\infty} \P(Y_n\isom A).
$$ 
\end{proof}

\subsection{Proof of Theorem~\ref{theorem:W:alt}}
Now we show that these results can be applied to $\cok(A_{n})$ to determine their distribution.
We will need a distribution that we know about to compare moments with.  
For a prime $p$, let $Y_n(p)$ be the cokernel of a skew-symmetric matrix with entries in $\Z_p$ (the $p$-adic integers) drawn from the additive Haar measure on such skew-symmetric matrices.
Then for any finite abelian $p$-group $G$, we have $\lim_{n\ra\infty} \E(\#\Sur(Y_n(p),G))=|\Sym^2 G|$.  
This follows, for example, by letting $a$ be the exponent of $G$, and noting than if we choose our entries $\zeta$ uniform from $0$ to $a-1$, then $Y_n$ and $A_n$ have the same distribution mod $a$, and thus the same $G$-moment (though also a much simpler argument can be given in the uniform case, as in \cite[Theorem 11]{Clancy2015}). 

 It has been shown \cite[Theorem 3.9]{Bhargava2015b} that as  $n\ra\infty$ the distributions of $Y_{2n}(p)$ approach the distribution of a random finite abelian $p$-group $Y(p)$ with
$$
\lim_{n\ra\infty }\P(Y_{2n}(p)\isom G)=\P(Y(p)\isom G)=\frac{|G|}{|\Sp(G)|}\prod_{i=0}^\infty (1-p^{-2i-1}),
$$
for $G\in \mathcal{S}_p$ (and $\P(Y(p)\isom G)=0$ for all other $G$). In the proof of \cite[Theorem 8.3]{Wood2017}, it is shown that if we have sequence of random finite abelian $P$-groups $Z_n$ such that for every finite abelian $P$-group $A$, the limits $\lim_{n\ra\infty} \P(Z_n\isom A)$ and 
$\lim_{n\ra\infty} \E(\#\Sur(Z_n,A))$ exist and are finite, then
$$\sum_{B \textrm{ $P$-group}} (\lim_{n\ra\infty} \P(Z_n\isom B)) \#\Sur(B,A)=
\lim_{n\ra\infty} \E(\#\Sur(Z_n,A)).$$
 Thus, $\E(\#\Sur(Y(p),G))=|\Sym^2 G|$.
 
Further,  we have, with probability $1$, that $Y_{2n+1}(p)\isom \Z_p \times G$ for a finite group $G$.   This is because the upper $n-1\times n-1$ submatrix has determinant $0$ with probability $0$ (the Haar measure of a hypersurface in $\Z_p^{k}$ is 0, see e.g. \cite[Proposition 2.1 (b)]{Bhargava2015b} in the easy case when $X$ is affine space and the measure is Haar measure).  By \cite[Theorem 3.11]{Bhargava2015b}, we have a random finite abelian $p$ group $Z(p)$ such that
$$
\lim_{n\ra\infty }\P(Y_{2n+1}(p)\isom \Z_p \times G)=\P(Z(p)\isom G)=\frac{1}{|\Sp(G)|}\prod_{i=1}^\infty (1-p^{-2i-1}).
$$
for $G\in \mathcal{S}_p$ (and $\P(Y_{2n+1}(p)\isom \Z_p \times G)=0$ for all other $G$).  
We can consider the random finite abelian groups $Y_{2n+1}(p) \tensor \Z/p^e\Z$.  
So as above, taking $e$ so that $p^eG=0$, we have 
\begin{align*}
\E(\#\Sur(\Z_p\times Z(p) ,G))=&\E(\#\Sur((\Z_p\times Z(p))\tensor \Z/p^e\Z,G))
\\
=&\lim_{n\ra\infty} \E(\#\Sur(Y_{2n+1}(p) \tensor \Z/p^e\Z,G))\\
=&|\Sym^2 G|.
\end{align*}

Now we return to $\cok(A_{n})$. In the even dimensional case, as mentioned in the proof of Theorem~\ref{theorem:cyclic:alt}, the group $\cok(A_{2n})$ is of the form $\Z^{2m} \times G$, for some $G\in\mathcal{S}$.
Since $m$ could be positive with positive probability, we need to reduce mod $a$ for some $a$ so we are in the setting of finite abelian groups.
For a positive integer $a$, we have that $\cok(A_{2n}) \tensor \Z/a\Z$ is a random finite abelian group satisfying the first condition of Theorem~\ref{T:mainmomlimit}.
Let $P$ be the set of the primes dividing $a$, and for every finite abelian $P$-group $G$, we have
\begin{align*}
\lim_{n\ra\infty} \E(\#\Sur(\cok(A_{2n}) \tensor \Z/a\Z,G)|)=\E(\#\Sur(\prod_{p\in P} Y(p) \tensor \Z/a\Z,G))=
\begin{cases}
|\Sym^2 G| &\textrm{ if $aG=0$}\\
0 & \textrm{otherwise}.
\end{cases}
\end{align*}
Thus, by Theorem~\ref{T:mainmomlimit}, we have that, for all $G$, 
\begin{equation}\label{E:evencut}
\lim_{n\ra\infty} \P(\cok(A_{2n}) \tensor \Z/a\Z\isom G)=\P(\prod_{p\in P} Y(p) \tensor \Z/a\Z\isom G).
\end{equation}

In particular, given a finite set of primes $P$, and an abelian $P$-group $G$, we can take $a=\prod_{p\in P} p^{e+1}$, where
$(\prod_{p\in P} p^{e})G=0$.  Let $\Z_P:=(\prod_{p\in P} \Z_p)$.
Then we have $\cok(A_{2n}) \tensor \Z_P \isom G$ if and only if $\cok(A_{2n}) \tensor \Z/a\Z\isom G$, and similarly for $\prod_{p\in P} Y(p)$.
So, we conclude for every finite abelian $P$-group $G$ that
\begin{align*}
\lim_{n\ra\infty} \P(\cok(A_{2n}) \tensor \Z_P \isom G)=\P(\prod_{p\in P} Y(p) \isom G)=\begin{cases}
\frac{|G|}{|\Sp(G)|}\prod_{p\in P}\prod_{i=0}^\infty (1-p^{-2i-1})&\textrm{ if $G\in\mathcal{S}_P$}\\
0 & \textrm{otherwise}.
\end{cases}
\end{align*}
In particular, since $Y(p)$ is supported on finite groups, we have $\lim_{n\ra\infty} \P(|\cok(A_{2n})|<\infty)=1$ by Fatou's lemma.
Thus, it follows that with asymptotic probability $1$, we have $\cok(A_{2n}) \tensor \Z_P \isom (\cok(A_{2n}))_P$.

In the odd dimensional case, the group $\cok(A_{2n+1})$ is of the form $\Z^{2m+1} \times G$, for some $G\in\mathcal{S}$.  
For a positive integer $a$, we have that $\cok (A_{2n+1}) \tensor \Z/a\Z$ is a random finite abelian group satisfying the second condition of Theorem~\ref{T:mainmomlimit}.
Let $P$ be the set of the primes dividing $a$, and for every finite abelian $P$-group $G$, we have
\begin{align*}
\lim_{n\ra\infty} \E(\#\Sur(\cok (A_{2n+1}) \tensor \Z/a\Z,G))=\E(\#\Sur(\prod_{p\in P} (\Z_p\tensor Z(p)) \tensor \Z/a\Z,G))=
\begin{cases}
|\Sym^2 G| &\textrm{ if $aG=0$}\\
0 & \textrm{otherwise}.
\end{cases}
\end{align*}
Thus, by Theorem~\ref{T:mainmomlimit}, we have that, for all $G$, 
\begin{equation}\label{E:oddcut}
\lim_{n\ra\infty} \P(\cok (A_{2n+1}) \tensor \Z/a\Z\isom G)=\P(\prod_{p\in P} (\Z_p\tensor Z(p)) \tensor \Z/a\Z\isom G).
\end{equation}

In particular, given a finite set of primes $P$, and an abelian $P$-group $G$, we can take $a=\prod_{p\in P} p^{e+1}$, where
$(\prod_{p\in P} p^{e})G=0$.  
Then we have $\cok (A_{2n+1}) \tensor \Z_P\isom  \Z_P \times G$ if and only if $\cok (A_{2n+1}) \tensor \Z/a\Z\isom \Z/a\Z \times G$, and similarly for $\prod_{p\in P} \Z_p \times Z(p)$.
So, we conclude for every finite abelian $P$-group $G$ that
\begin{align*}
\lim_{n\ra\infty} \P(\cok (A_{2n+1}) \tensor \Z_P\isom \Z_P\times  G)=\P(\prod_{p\in P} Z(p) \isom G)=\begin{cases}
\frac{|G|}{|\Sp(G)|}\prod_{p\in P}\prod_{i=0}^\infty (1-p^{-2i-1})&\textrm{ if $G\in\mathcal{S}_P$}\\
0 & \textrm{otherwise}.
\end{cases}
\end{align*}
In particular, since $Z(p)$ is supported on finite groups, we have $\lim_{n\ra\infty} \P(\rank (\cok (A_{2n+1}) )=2n)=1$ by Fatou's lemma.  Thus, it follows that with asymptotic probability $1$, we have $\cok(A_{2n})  \isom \Z \times G$, and
$\cok (A_{2n+1}) \tensor \Z_P\isom \Z_P\times  G$ if and only if $(\cok_{\operatorname{tors}} A_{2n+1})_P\isom G$.

It  also follows from \eqref{E:evencut} and \eqref{E:oddcut}, using $a=p$, that the ranks of $\cok(A_{2n})$ and $\cok (A_{2n+1})$ tensored with $\F_p$, in the limit, match those of $Y(p)$ and $\Z_p\times Z(p)$ tensored with $\F_p$, respectively.  These ranks are given in \cite[Example F]{Delaunay2001} (see also \cite[Theorem 1.10]{Bhargava2015b}).  This gives the matrix rank statement of Theorem~\ref{theorem:W:alt}.

\section{Treatment for moderate primes: structures of the generalized normal vectors}\label{section:lemmas}
Throughout Sections \ref{section:lemmas}, \ref{section:normalvector} and \ref{section:rankevolving} of our treatment of the moderate primes $N$ is a running parameter where, if not specified otherwise, we will assume 
$$cn \le N \le n,$$
for some constant $c>0$.
In this first section of the treatment we will gather various useful ingredients, most of which are of different natures. Two key results are Proposition \ref{prop:structure:subexp:sym} and Proposition \ref{prop:structure:subexp:lap}. 

\subsection{Odlyzko's type results} As a warm-up let us introduce some elementary tools. The first ingredient is a variant of Odlyzko's lemma (see \cite{KKS, M1,NgW,TV}).

\begin{lemma}\label{lemma:O} Let $0<\alpha<1$ be given. Let $\F$ be a field.
For a deterministic subspace $V$ of $\F^N$ of dimension $d$ and a random vector $X\in \F^N$ with entries being i.i.d. copies of $\xi$ satisfying \eqref{eqn:alpha},  we have
$$\P(X \in V) \le (1-\alpha)^{N-d}.$$
\end{lemma}
We also have a slight variation of Lemma~\ref{lemma:O} that has the same proof. 
\begin{lemma}\label{cor:O} Let $V$ be a deterministic subspace of $\F^N$ of dimension $d$ and $I\sub[N]$, for which $V_{I}$ has full dimension $d$. 
Let $J$ be a subset of $I^c$ such that $|J|=k$, and let $X\in \F^N$ be a random vector where the entries over $J$ and $J^c$ are independent, and the entries over $J$ are i.i.d. copies of $\xi$ satisfying \eqref{eqn:alpha}. Then we have
$$\P(X \in V) \le (1-\alpha)^{k}.$$
\end{lemma} 

{\bf Definition of $L_N$:}
In what follows,  we abuse notation by writing $L_N$ for the matrix %
obtained from $L_n$ (defined in the paragraph after Theorem \ref{theorem:W:lap})
 by restricting to its first $N$ rows and columns. 
In particular, whenever we write $L_N$, there is also always an implicit $n$ and random integers $x_{ij}$ for $1\leq i< j \leq n$.
In particular,  $L_N$ is a function of {\it all} the random variables $x_{kl}, 1\le k<l \le n+1$, including those with $k,\ell>N$.
This notation only applies when the subscript of $L$ is the capital letter $N$.

\begin{remark}
 When our random matrix is symmetric or skew-symmetric, the sparsity (number of $0$ coefficients) of vectors plays a large role in the arguments.  When the matrix is the graph Laplacian, the analog of sparseness is having a large number of coefficients of the same value, even if that value is not $0$.  We will, when referring non-technically to our arguments, call this analog \emph{sparsity} as well.
\end{remark}

\begin{lemma}\label{lemma:sparse:1} Let $\eps, \delta>0$ be sufficiently small constants (given $\alpha$ from \eqref{eqn:alpha}), and let $N$
 be a sufficiently large positive integer (given $\eps, \delta$ and $\al$) and $n\geq N$. 
  Let $p$ be a given prime. 
\begin{enumerate} 
\item (full rank of thin matrices) Let $1\le k_0\le (1-\eps)n/2$. Then with probability at least $1- n(1-1/2)^{n-2k_0-1}$ the matrix generated by $k_0$ columns $X_{1},\dots, X_{k_0}$ of $L_n$ over $\F_p$ has full rank. In particular we have that $\rank(L_n/p) \ge (1-\eps)n/2$ 
with probability at least  $1- n(1-1/2)^{\eps n-1}$. 
\vskip .05in
\item  (non-sparsity of generalized normal vectors)  Let $Y_0$ be a fixed vector in $\F_p^N$. The following holds with probability at least $1- (1-1/2)^{N/2}$:  any non-zero vector $\Bv$ for which $L_{N} \Bv$  agrees with $Y_0$ in at least $(1-\eps)N$ coordinates  must have the property that $\Bv -a \1$ has at least $\delta N$ non-zero components for any $a\in \F_p$.

\end{enumerate}

The same conclusions holds for the symmetric and skew-symmetric models $M_n$ and $A_n$, where $(1-1/2)$ is replaced by $(1-\al)$, and in (2) the vector $\Bv-a \1$ is replaced by $\Bv$. 

\end{lemma} 
In the special case that $Y_0$ is the zero vector,  this shows that with high probability the (actual)
 normal vectors are non-sparse.

Note that in the proof it suffices to choose $\eps, \delta$ to be slightly smaller than $\al$. When $\eps,\delta$ get smaller the probability bounds get better in (2), but we obtain a worse bound on the number of non-zero entries. In applications, we usually assume $N\ge \delta_0 n$ for a given positive constant $\delta_0$.

\begin{proof}[Proof of Lemma \ref{lemma:sparse:1}]%

In what follows, for the Laplacian case we understand that $\al=1/2$. For (1), for each $k\le k_0$ let $\CE_k$ be the event that $X_{k+1} \in span(X_1,\dots, X_k)$. This belongs to the event $\CE_k'$ that $X_{k+1}' \in span (X_1',\dots, X_k')$, where $X_i'$ is obtained from $X_i$ by eliminating the first $(k+1)$ coordinates. Now for this event, conditioning on any realization of $X_1',\dots, X_k'$, the probability $X_{k+1}'$ belongs to $span (X_1',\dots, X_k')$, as $X_{k+1}'$ is now independent from $X_1',\dots, X_k'$, is at most $(1-\al)^{(n-1-k)-k}$ by Lemma \ref{lemma:O}. Thus
$$\P(\CE_k) \le \P(\CE_k') \le (1-\a)^{n-1-2k}.$$
Summing over $k$ we obtain (1).

We now prove (2). 
Let $L'$ be the $(N+1)\times (N+1)$ symmetric matrix with with upper-left $N\times N$ submatrix $L_N$ and rows summing to $0$.
We let $M$ be the matrix of interest,  either $L'$ or $M_N$ or $A_N$,  %
and let $X_i$ be the columns of $M$.

We will show that the complement of the event under consideration in (2)  implies that there exist $I,J\sub [N+1]$ and $K\sub J\setminus I$
with $|I|=|K|=\lfloor \delta N \rfloor +1$ and $|J|=\lceil (1-\epsilon) N \rceil$
such that either
\begin{enumerate}
\item[(a)]$M_{(J\setminus I) \times I}$ has linearly dependent columns, or

\item[(b)] $M_{K\times I}$ is full rank and $Y_{J\setminus I}$ is in the span of $(X_i)_{J\setminus I}$ for $i\in I$.
\end{enumerate}

If the event of (2) fails, in the Laplacian case we let $\Bw \in \F_p^{N+1}$ be the vector whose first $N$ coordinates agree with $\Bv -a \1$ and whose last coordinate is $-a$.  We have $(L'\Bw)_{[N]} =L_N \Bv$.  We let $\Bw=\Bv$ in the $M_N$ or $A_N$ cases.
Then we have $\Bw$ such that $(M\Bw)_{[N]}$ agrees with $Y_0$ in at least $(1-\epsilon)N$ coordinates, and let $J$ be the set of such coordinates.  Further $\Bw$ has at most $\lfloor \delta N \rfloor +1$ non-zero entries, and we let $I$ index a set containing those entries.
If (a) fails, then we let $K$ be a subset of rows of $M_{(J\setminus I) \times I}$ such that $M_{K\times I}$ is full rank.  
Then $(M\Bw)_J=(Y_0)_J$ and $\supp(\Bw)\sub I$ implies that $Y_{J\setminus I}$ is in the span of $(X_i)_{J\setminus I}$ for $i\in I$.

Now we bound the probability of (a) and (b) for fixed $I,J,K$.  
Note that $M_{J\setminus I \times I}$ has independent, $\alpha$-balanced entries, so the chance that it has linearly dependent columns is at most
$|I|(1-\alpha)^{|J\setminus I|-|I|+1}$ by applying Lemma~\ref{lemma:O} as we expose one column at a time.
In case $(b)$, we condition on the values of $M_{J\setminus I \times I}$ in rows $K$.
We have $Y_K=\sum_{i\in I} c_i (X_i)_K$ for unique (non-random, after our conditioning) values $c_i$ not all $0$ (since $M_{K\times I}$ is full rank).  Since $Y_{J\setminus I}$ is in the span of $(X_i)_{J\setminus I}$ for $i\in I$, we must 
have $Y_{J\setminus I}=\sum_{i\in I} c_i (X_i)_{J\setminus I}$.  Suppose $c_j\neq 0$ for some index $j\in I$.  Then we further condition on all values of $M_{J\setminus I \times I}$ except those in column $j$, and we see that  $Y_{J\setminus I}=\sum_{i\in I} c_i (X_i)_{J\setminus I}$ implies each coordinate of $(X_j)_{(J\setminus I)\setminus K}$ must be some fixed value, which happens with probability at most $(1-\alpha)^{|(J\setminus I)\setminus K|}$.

Taking union bounds over the choices of $I,J,K$, we have that the complement of the event in (2) happens with probability at most
$$
 \binom{N+1}{\lfloor \delta N \rfloor +1}^2 \binom{N+1}{\lfloor \eps N \rfloor}
(\delta N+2) 
 (1-\al)^{(1-\eps -2 \delta )N-2}.
$$ 
As long as $\eps, \delta$ are  sufficiently small and $N$ is sufficiently large, the above is bounded by $(1-\al)^{N/2} $.
\end{proof}

\subsection{Quadratic repulsion} In our next lemma we provide another useful estimate, which in practice will be more powerful than (1) of Lemma \ref{lemma:sparse:1}.

\begin{lemma}\label{lemma:quadratic} There exist positive constants $c,c'$ (which might depend on $\al$ in the symmetric and skew-symmetric cases) such that the following holds.  We let $ k \le N \le m \le n$ be positive integers.  Let $G_n$ be either the model $M_{n}, A_n$ or $L_n$ from Theorem \ref{theorem:cyclic:sym}, Theorem \ref{theorem:cyclic:alt} or Theorem \ref{theorem:cyclic:lap} 
respectively. Then for each prime $p$, 
 with probability at least $1 - \exp(-ck^2 + c'm)$, for any subset $I,J$ of size $N$ of $[m]$ the matrix $G_{I \times J}/p$ has rank at least $N - k$. 
\end{lemma}

Note that this result (as well as Lemma \ref{lemma:sparse:1} above) holds for any $p$. In can be seen as a finite field analog of the quadratic repulsion phenomenon of eigenvalues of Wigner matrices (see for instance \cite{ESY, NgTV, NgF}). Also, in application we will assume $k\ge n^{1/2+\eps}$ for $\eps>0$, in which case, even after taking a union bound over all choices of $I$ and $J$,
we still obtain a bound of type $1-\exp(-\Theta(k^2))$ when $n$ is sufficiently large.

\begin{comment}

\begin{figure}%

\centering
%

%
\begin{tikzpicture}
\draw (0,0) -- (4,0) -- (4,4) -- (0,4) -- (0,0);
\draw (0,2) -- (2,2)--(2,4)--(0,4) ; 
\draw (0,2) -- (2,2) node[above  left =.15in]{$L_{J_0 \times J_0}$};   
\draw (0,4) -- (1,4) node[above =.1in]{$J_0$};
\draw (0,1.5) -- (2.5,1.5)--(2.5,4)--(0,4) ; 
\draw (0,4) -- (3,4) node[above left =.05in]{$J'\bs J_0$};
\draw [line width=1pt] (2,4)--(2.5,4);
\draw (0,1) -- (3,1)--(3,4)--(0,4);
\draw [dashed] (2.5,1.5)--(2.5,1);
\draw [dashed] (0,1.25) -- (2.5,1.25); 
\draw [black, line width=2pt]  (2,1.25)--(2.25,1.25) node[above =0.05in]{$N_i\cap(J'\bs J_0)$};
\draw (1,4) node{.};
 \draw (2,4) node{.};
\end{tikzpicture}
\caption{Creating extra randomness.}
\label{figure:quadratic}
\end{figure}

\end{comment}

\begin{proof}[Proof of Lemma \ref{lemma:quadratic}] 
The complement of the event in the lemma implies that there exist $I,J\sub [m]$ with $|I|=|J|=N$  and $I_0\sub I$ and $J_0\sub J$ with $|I_0|=|J_0|<N-k$ such that
 $G_{I_0\times J_0}$ is non-singular and $\rk (G_{I\times J})=\rk  (G_{I_0\times J_0}).$ We now bound the probability of this latter event for a fixed $I,J,I_0,J_0$.  

Since $|I\setminus I_0|,|J\setminus J_0|>k$, we can choose $A\sub I\setminus I_0$ and $B\sub J\setminus J_0$ such that $A$ and $B$ are disjoint, $|B|$ is even, and $|A|\geq k/2$ and $k/2\geq |B|\geq k/2-2$.  
(Here 
$A$ takes $k/2$ or $(k+1)/2$, and then $B$ takes $k/2$ minus 0 or $1$ or $(k-1)/2$ minus 0 or 1, to be even.)
  We pick an arbitrary fixed-point-free involution on $B$. 
In the Laplacian case, we will start with $G_n$, and then do an $AB$-shuffle on $G_n$ resulting in $G'_n$ (with entries $x'_{ij}$), and we will bound the probability that $G'_{I_0\times J_0}$ is non-singular and 
$\rk (G'_{I\times J})=\rk  (G'_{I_0\times J_0}).$  In the non-Laplacian case,  we let $G'_n=G_n$.

In the Laplacian case,  let $B_0$ contain one element of each involution orbit of $B$.
For $j\in B_0$, let $A_j$  be the subset of $i\in A$ for which $x_{ij}\neq x_{i\bar{j}}$.
We call $j\in B_0$ bad if $|A_j|< |A|/4$.
By the Chernoff bound,  
 we have that  $j$ is bad with  probability at most $\exp(-|A|/8)$.
 Let $\mathcal{E}$ be the event that there are at least $|B_0|/2$ bad $j$.
  There are at most $2^{|B_0|}$ subsets of $B_0$ of size at least $|B_0|/2$, and so 
$\P(\mathcal{E})\leq 2^{|B_0|} \exp(-|A|/8)^{|B_0|/2}$.

We now condition on $G_n$ in the Laplacian case and  $G'_{I_0\times J}$ and $G'_{I\times J_0}$ in the other cases. 
Let $X_i$ be the $i$th column of $G'_n$.  If $G_{I_0\times J_0}=G'_{I_0\times J_0}$ is non-singular and $\rk (G'_{I\times J})=\rk  (G'_{I_0\times J_0}),$
then for each $j\in J\setminus J_0$, we have that there are unique constants $c_{j\ell}$ (determines by our conditioned values) such that
$(X_j)_{[I_0]}=\sum_{\ell\in J_0} c_{j\ell} (X_\ell)_{[I_0]}$.
 In order for $\rk (G'_{I\times J})=\rk  (G'_{I_0\times J_0}),$ we must also have, for each $j\in B$, that
 $(X_j)_{A}=\sum_{\ell\in J_0} c_{j\ell} (X_\ell)_{A}$.  

In the Laplacian case,  after our conditioning,
 the matrix entries of $G'_{A\times B_0}$ are independent.  Further, an entry $x'_{i,j}$ for $i\in A$ and $j\in B_0$ is $1/2$-balanced if $x_{ij}\neq x_{i\bar{j}}$.
Unless $\mathcal{E}$ occurs,  we have at least $|B_0|/2$ values $j\in B_0$ such that for at least $|A|/4$ values of $i\in A$ we have that 
$x'_{i,j}$ is $1/2$-balanced.
The probability that 
  $(X_j)_{A_j}=\sum_{\ell\in J_0} c_{j\ell} (X_\ell)_{A_j}$ for such a $j\in B_0$ is at most $(1/2)^{|A|/4}$,  and thus
  the probability this holds for all such $j\in B_0$
is at most  $(1/2)^{|A||B_0|/8}+\P(\mathcal{E})$.

 In the non-Laplacian case, the entries of $G'_{A\times B}$ are independent and $\alpha$-balanced.
 Thus the probability that 
  $(X_j)_{A}=\sum_{\ell\in J_0} c_{j\ell} (X_\ell)_{A}$ for all $j\in B$ is at most $(1-\alpha)^{|A||B|}.$
  
Thus, in any case, the  probability that, for a fixed $I,J,I_0,J_0$, we have $G'_{I_0\times J_0}$ is non-singular and 
$\rk (G'_{I\times J})=\rk  (G'_{I_0\times J_0})$ is at most
$$
2^{|B_0|}e^{-|A||B_0|/16}+ \max(1/2,1-\alpha)^{|A||B_0|/8}\leq e^{k/4-k(k/4-1)/32} +e^{ -c_0  k(k/4-1)/16 },
$$  
for some $c_0$ dependong on at most $\alpha$.
Since, there are at most $2^{4m}$ choices of $I,J,I_0,J_0$, and $G'_n$ has the same distribution as $G_n$, the lemma follows.
\end{proof}

\subsection{Concentration discrepancy for general $\al$-balanced random variables} \label{SS:ConcDisc}

We next introduce some modifications to simplify the general models of symmetric, skew-symmetric and Laplacian matrices $M_N, A_N, L_N$ considered in Theorem \ref{theorem:cyclic:sym}, Theorem \ref{theorem:cyclic:alt}, and Theorem \ref{theorem:cyclic:lap}.

Throughout Section~\ref{SS:ConcDisc}
 we assume that $p\ge 3$ is a prime number. Let $\Bw=(w_1,\dots, w_N)$ be a non-zero deterministic vector in $\F_p^N$. We use a strategy
  from \cite{NgW} to bound the probability $\P(x_1 w_1 +\dots + x_N w_N =a)$, where $x_i$ are i.i.d. copies of a random $\alpha$-balanced $\xi\in \F_p$. 

Let $\psi := \xi-\xi'$ be the symmetrization of $\xi$ and let $\psi'$ be a lazy version of $\psi$ so that
\[\P(\psi'=x) = \begin{cases}\frac{1}{2}\P(\psi=x) \mbox{ if } x\neq 0\\\frac{1}{2}\P(\psi=x) + \frac{1}{2}, \mbox{ if } x=0.\end{cases}\]
Notice that $\psi'$ is symmetric as $\psi$ is symmetric (i.e. $\P(\psi'=a)=\P(\psi'=-a)$ for all $a$), and we can check that $\max_x \P(\psi =x)  \le 1 -\alpha$, and so
$$\max_{x\in \F_p} \P(\psi'=x) \le 1 -\alpha/2.$$ 

Let $\pm t_1, \pm t_2,\dots$ be the non-zero values taken by $\psi'$ with positive probability (with $\pm t_i\ne t_j$ for $i\ne j$), and let 
$\beta_j:=2\P(\psi' =t_j) = 2\P(\psi' =-t_j)$ and $\beta_0:=\P(\psi'=0)$.
Let
\begin{equation}\label{eqn:S}
S:=x_1 w_1 +\dots + x_n w_N = X \cdot \Bw,
\end{equation} 
where $X=(x_1,\dots, x_N)$. 

Consider $a  \in \F_p$ where  $\P( S=a)$ is maximum (or minimum). Using the standard notation $e_p(x)$ for $\exp(2\pi \sqrt{-1} x/p )$, we have
\begin{equation*} \P(S=a)= \E \frac{1}{p} \sum_{t \in \F_p} e_p (t (S-a)) = \E \frac{1}{p} \sum_{t \in \F_p} e_p (t S) e_p(-t a) = \frac{1}{p}+ \E \frac{1}{p} \sum_{t \in \F_p, t \neq 0} e_p (t S) e_p(-t a).\end{equation*}
So by independence, 
\begin{align*}
|\P(S=a)- \frac{1}{p} |&\le \frac{1}{p} \sum_{  t\in \F_p, t \neq 0} |\E e_p (t S)| \le \frac{1}{p} \sum_{  t\in \F_p, t \neq 0} \prod_{i=1}^N |\E e_p(t x_i w_i)|\\ 
&\le  \frac{1}{p} \sum_{  t\in \F_p, t \neq 0} \prod_{i=1}^N (\frac{1}{2}(|\E e_p(t x_i w_i)|^2+1)) \le \frac{1}{p} \sum_{  t\in \F_p, t \neq 0}  \prod_{i=1}^N |\E e_p( t \psi'  w_i)| \\
& =  \frac{1}{p} \sum_{  t\in \F_p, t \neq 0} \prod_{i=1}^N ( \beta_0 +\sum_{j=1}^l \beta_j  \cos \frac{2\pi t t_j w_i}{p}) \le \frac{1}{p} \sum_{t \in \F_p, t \neq 0}  \prod_{i=1}^N (\beta_0 + \sum_{j=1}^l \beta_j  |\cos \frac{\pi t t_j w_i}{p} |) 
\end{align*}
where we made the change of variable $t \rightarrow t /2$ (in $\F_p$ for $p\ge 3$) and used the triangle inequality.

By convexity, we have that   $|\sin  \pi z | \ge 2 \|z\|_{\R/\Z}$ for any $z\in \R$, where $\|z\|_{\R/\Z}$ is the distance of $z$ to the nearest integer. Thus, $| \cos \frac{\pi x}{p}|  \le  1- \frac{1}{2} \sin^2 \frac{\pi x}{p}  \le 1 -2 \|\frac{x}{p} \|_{\R/\Z}^2$. Hence for each $w_i$
$$\beta_0 + \sum_{j=1}^l \beta_j  |\cos \frac{\pi t t_j w_i}{p} | \le 1 - 2 \sum_{j=1}^l \beta_j \|\frac{t t_j w_i}{p} \|_{\R/\Z}^2 \le \exp(-2 \sum_{j=1}^l \beta_j \|\frac{t t_j w_i}{p} \|_{\R/\Z}^2).$$
Consequently, by Jensen's inequality and the fact that $\sum_{j=1}^l \beta_j = 1-\beta_0 \ge \al/2$ we obtain a key inequality
\begin{align}\label{eqn:Fourier}  
|\P(S=a)- \frac{1}{p} | & \le  \frac{1}{p} \sum_{t \in \F_p, t \neq 0}  \prod_{i=1}^N (\beta_0 + \sum_{j=1}^l \beta_j  |\cos \frac{\pi t t_j w_i}{p} |)   \le  \frac{1}{p} \sum_{t \in \F_p, t \neq 0} \exp( - 2 \sum_{i=1}^N \sum_{j=1}^l \beta_j \|\frac{t t_j w_i}{p} \|_{\R/\Z}^2)  \nonumber \\
& \le  \frac{1}{p} \sum_{t \in \F_p, t \neq 0} \exp( - 2 \sum_{j=1}^l \frac{\beta_j}{\sum_{j=1}^l  \beta_j} (\sum_{j=1}^l \beta_j) \sum_{i=1}^N  \|\frac{t t_j w_i}{p} \|_{\R/\Z}^2)  \nonumber \\
&\le  \frac{1}{p} \sum_{t \in \F_p, t \neq 0} \sum_{j=1}^l   \frac{\beta_j}{\sum_{j=1}^l  \beta_j}  \exp( - 2 ( \sum_{j=1}^l \beta_j )\sum_{i=1}^N\|\frac{t t_j w_i}{p} \|_{\R/\Z}^2)  \nonumber \\
&\le  \sum_{j=1}^l   \frac{\beta_j}{\sum_{j=1}^l  \beta_j}   \frac{1}{p} \sum_{t \in \F_p, t \neq 0} \exp( - 2 (\al/2)\sum_{i=1}^N\|\frac{t t_j w_i}{p} \|_{\R/\Z}^2)  \nonumber \\
&= \frac{1}{p} \sum_{t \in \F_p, t \neq 0} \exp( - \al \sum_{i=1}^N\|\frac{t w_i}{p} \|_{\R/\Z}^2).
\end{align}
Motivated by this, we define $\rho(\Bw) = \rho_\al(\Bw)$, the {\it concentration discrepancy} of $\Bw$, to be 
\begin{equation}\label{eqn:rho}
\rho(\Bw):= \frac{1}{p} \sum_{t \in \F_p, t \neq 0} \exp( -\al \sum_{i=1}^N\|\frac{t w_i}{p} \|_{\R/\Z}^2).
\end{equation}

Similarly, when we analyze a Laplacian matrix
we will define 
$\rho_{L}(\Bw)=\rho_{L,\al}(\Bw)$ as
\begin{equation}\label{eqn:rho:L}
\rho_{L}(\Bw):= \max_w \rho(\Bw - w \1) = \max_w \frac{1}{p} \sum_{t \in \F_p, t \neq 0} \exp( -\al \sum_{i=1}^N\|\frac{t (w_i-w)}{p} \|_{\R/\Z}^2).
\end{equation}
For later use we also remark the following homogenized inequality  
\begin{align*}
\sum_{i=1}^N\|\frac{t (w_i-w)}{p} \|_{\R/\Z}^2 &\ge \sum_{k=1}^{\lfloor (N-1)/2 \rfloor}(\|\frac{t (w_{2k}-w)}{p} \|_{\R/\Z}^2 +\frac{t(w_{2k+1}-w)}{p} \|_{\R/\Z}^2)\\
& \ge \frac{1}{2} \sum_{k=1}^{\lfloor (N-1)/2 \rfloor}(\|\frac{t (w_{2k}-w_{2k+1})}{p} \|_{\R/\Z}^2).
\end{align*}
Hence with $\Bw'=(w_1-w_2, \dots, w_{2 \lfloor (N-1)/2 \rfloor} -  w_{2 \lfloor (N-1)/2 \rfloor +1})$ we have
\begin{equation}\label{eqn:w-w'}
\rho_{L,\al}(\Bw) \le \rho_{\al/2}(\Bw').
\end{equation}

An Erd\H{o}s-Littlewood-Offord type result for finite fields 
is as follows.
\begin{theorem}[{\cite[Theorem A.15]{NgP}}]\label{theorem:LO} Let $p$ be a prime and $N$ a positive integer.  Let $\Bw\in \F_p^N$. Let $0<c_{nsp}<1$ be a number that might depend on $N$
and assume that 
\begin{equation}\label{eqn:sparse}
|\supp(\Bw)| \ge c_{nsp} N.
\end{equation} 
Then for $X=(x_1,\dots, x_N)$ with $x_i$ being i.i.d. copies of an $\alpha$-balanced random integer $\xi$,
$$\sup_{r \in \F_p} |\P(X \cdot \Bw =r) -\frac{1}{p}| \le \rho(\Bw)=O(\frac{1}{\sqrt{c_{nsp} N}}).$$

Furthermore, if we assume that for all $w\in \F_p$ 
\begin{equation}\label{eqn:sparse:lap}
|\supp(\Bw-w\1)| \ge c_{nsp} N.
\end{equation} 
Then for $X=(x_1,\dots, x_N)$ with $x_i, $ for $1\le i\le N-1$, being i.i.d. copies of $\xi$ and with $x_N = -\sum_{1\le i\le N-1} x_i$ we have the following affine analog
$$\sup_{r \in \F_p} |\P(X \cdot \Bw =r) -\frac{1}{p}| \le \rho_L(\Bw)=O(\frac{1}{\sqrt{c_{nsp} N}}).$$
Here the implied constants depend on $\al$, but not on $p,N,c_{nsp}$.  
\end{theorem}
Hence, for instance in the case that $X=(x_1,\dots, x_N)$ with $x_i$ being i.i.d.  copies of $\xi$,  as long as the $w_i$ are non-zero, the random sum $X \cdot \Bw$ spreads out quickly in $\F_p$ in such a way that the discrepancy with respect to the uniform distribution is $O(1/\sqrt{N})$. 
This rate of decay is best possible (in terms of $N$) if one does not have extra information on $\Bw$. 

We next give  simple but useful bounds on $\rho(\Bw),\rho_L(\Bw)$ of the sort used in the proof of Lemma~\ref{L:FullFcode}.  

    \begin{lemma}\label{lem:deg:1} Let $N$ be a positive integer, and let $p$ a prime. Let $\Bw\in \F_p^N$.
    \begin{itemize}
    \item If $|\supp(\Bw)| \ge c_{nsp} N$
    for some $c_{nsp}>0$ then  
    $$\rho(\Bw) \le \exp\big(- \al \frac{c_{nsp}N}{p^2} \big).$$ 
    \item If %
$|\supp(\Bw-w\1)| \ge c_{nsp} N$    
    for some $c_{nsp}>0$ and for all $w\in \F_p$ then  
    $$\rho_L(\Bw) \le \exp\big(-\al \frac{c_{nsp}N}{p^2} \big).$$ 
    \end{itemize}
    \end{lemma}

    \begin{proof} 
    For any $t, w_i \neq 0$ in $\F_p$ we have
    $$\|t w_i/p\|_{\R/\Z}^2 \ge (\frac{1}{p})^2.$$
    Hence if $\Bw$ has at least $c_{nsp} N$ non-zero coordinates and $t\ne0$, then
    $$\sum_{i=1}^N\|\frac{t w_i}{p} \|_{\R/\Z}^2 \ge \frac{c_{nsp}N}{p^2}.$$
The lemma then follows from the definitions of  $\rho(\Bw),\rho_L(\Bw)$.
    \end{proof}

  \subsection{A simple structure result} Another elementary observation is that one can obtain some useful structure on the $w_i$ when $\rho(\Bw)$  (or $\rho_L(\Bw)$) are sub-exponentially small.
      	
    	\begin{lemma}\label{lem:deg:2}
    		Let $\delta  < 1$ be a positive constant.  Let $p$ be a prime number and $N$ a positive integer. Let $\Bw = (w_1, \cdots, w_N) \in \F_p^N$ such that
    		$$\rho(\Bw) \ge \exp(-N^\delta).$$ 
    		Then for any $1 \leq N' \leq N$, there is a set $W'$ of $N-N'$ components $w_i$ and an arithmetic progression $Q$ (i.e. a GAP of rank one, cf. Definition \ref{def:gap}) in $\F_p$ that contains $W'$, where 
    			\begin{equation}\label{eqn:Qsize}
    			|Q| \leq 2 \al^{-1} p\sqrt{\lceil N^\delta +1 \rceil /N'}.
    			\end{equation} 
    			Similarly, assume that 
    		$$\rho_L(\Bw) \ge \exp(-N^\delta).$$ 
    		Then there is an arithmetic progression $Q$ satisfying \eqref{eqn:Qsize} that contains $N-N'$ components $w_i$.
    				\end{lemma}    	
    	\begin{proof} It suffices to consider the first case because translation of an arithmetic progression  is also an arithmetic progression.    
    	Consider the level sets $S_m:=\{x\in \F_p \,|\, x \neq 0 \mbox{ and }   \al \sum_{i=1}^N\|\frac{x w_i}{p} \|_{\R/\Z}^2  \le m  \} $.  We have
    		$$\exp(-N^\delta) \le \rho(\Bw)  =   \frac{1}{p} \sum_{x \in \F_p, x \neq 0} \exp( -   \al \sum_{i=1}^N\|\frac{x w_i}{p} \|_{\R/\Z}^2  )  \le \frac{1}{p} \sum_{m \ge 1} \exp(-(m-1)) |S_m| .$$
		
		Since $\rho(\Bw)  \ge \exp(-N^\delta)$ and $\sum_{m \ge N^\delta +2} \exp(-(m-1)) <\exp(-N^\delta)$,  there must be a  level set $S_m$ in the range $1\le m \le \lceil N^\delta +1 \rceil $ such that $S_m$ is non-empty, and so there exists $x_0 \neq 0$ so that 
    		$$ \al \sum_{i=1}^N\|\frac{x_0 w_i}{p} \|_{\R/\Z}^2    \le m.$$
    		So, with $W'$ being the set of $w_i$ such that $ \al \|\frac{x_0 w_i}{p} \|_{\R/\Z}^2  \le \frac{m }{N'}$, we have that $W'$ has at least $N-N'$ elements. 
    	          By definition, for $w_i\in W'$ we have $ 
    		\al \|\frac{x_0  w_i}{p} \|^2 \le \frac{m }{N'},  
    		$ and this implies that after a dilation by $x_0$ the set $W'$ belongs to the arithmetic progression $P$ where 		
    		$$P:=\Big\{x \in \F_p\,|\, \|\frac{x}{p}\|_{\R/\Z} \le \al^{-1} \sqrt{\frac{m}{N'}}\Big\}.$$
        		Notice that the size of $P$ is bounded by $ 2\al^{-1} p \sqrt{m/ N'} \le 2\al^{-1} p \sqrt{\lceil N^\delta +1 \rceil/ N'}$ as desired.
    	\end{proof}

\subsection{Generalized normal vectors} Here and later we will need the concept of generalized normal vectors (that was mentioned in Lemma \ref{lemma:sparse:1}).

\begin{definition} Given a vector $Y_0$ and an index set $I\subset [N]$ (usually of size $(1-o(1))N$),  we say that a vector $\Bv$ is a {\it generalized normal vector} of $G_N$ (with respect to $Y_0$ and $I$) if 
$$(G_N \Bv)_{I} = (Y_0)_I.$$
\end{definition}

Now we mention a key result of the section, which says that as long as $p$ is not too small and not too large ($p\le \exp(n^{c})$ for some sufficiently small constant $c$), with very high probability the random walks formed by generalized normal vectors of $M_N,A_N$ and $L_{N\times N}$ spread out in $\F_p$ in such a way that the discrepancy from the uniform distribution is sub-exponentially small.

    \begin{prop}[Non-local structure of the normal vectors: symmetric and skew-symmetric cases]\label{prop:structure:subexp:sym} 
    Let $N$ be a positive integer and $G_N$ be either $M_N$ or $A_N$ from Theorem \ref{theorem:cyclic:sym} and Theorem \ref{theorem:cyclic:alt}. Let $0<\lambda\leq 1$ be a given constant. Then there exists a  %
    positive constants $c$  (given $\la$ and $\al$ from \eqref{eqn:alpha}) 
    such that the following holds for $p$ sufficiently large (depending on $\la, \al$)
    but $p \le \exp(N^c)$.  Let $I_0\subset [N]$ be an index set of size $\lfloor \la N \rfloor$. For any non-zero fixed vector $Y_0 \in \F_p^N$, with probability at least $1 - \exp(-n^{c})$ (with respect to $G_N$), for any vector $\Bv \in \F_p^N$ such that $G_N \Bv=Y_0$ in at least $N -N^{c}$ coordinates, we have (with $\rho(-)$ from \eqref{eqn:rho})
$$\rho(\Bv_{I_0}) \le \exp(-N^c).$$
\end{prop}

We also have similar result for Laplacian matrices.

    \begin{prop}[Non-local structure of the normal vectors: Laplacian case]\label{prop:structure:subexp:lap} 
Let $0<\lambda\leq 1$ be a given constant. Then there exists a positive constant $c$ (given $\la$)  such that the following holds for $p$ prime and sufficiently large (depending on $\la$).
Let $n,N$ be positive integers and let $L_n$ be as in Theorem \ref{theorem:cyclic:lap}, and let $L_N$ be the $[N]\times[N]$ submatrix of $L_n$. 
     We assume that $p \le \exp(N^c)$. Let $I_0\subset [N]$ be an index set of size $\lfloor \la N \rfloor$ that might depend on the randomness of $x_{kl},$ for $N+1 \le k \text{ or } N+1 \le l$, but is independent of the randomness of $x_{kl},$ for  $1\le k<l \le N$. 
     For any non-zero fixed vector $Y_0 \in \F_p^N$, with probability at least $1 - \exp(-N^{c})$, 
     for any vector $\Bv \in \F_p^N$ such that $L_N \Bv=Y_0$  in at least $N -N^{c}$ coordinates we have (with $\rho_L(-)$ from \eqref{eqn:rho:L})
$$\rho_L(\Bv_{I_0}) \le \exp(-N^c).$$
\end{prop}

We will present the proofs of these results in the next section and deduce Proposition \ref{prop:moderate:lap} and Proposition \ref{prop:moderate:alt} in Section \ref{section:rankevolving}. To complete this section, we remark that in fact for the symmetric and skew-symmetric case, to prove Proposition \ref{prop:moderate:lap} (or Propositions \ref{prop:rankevolution:sym} and \ref{prop:rankevolution:alt}) 
one only needs Proposition \ref{prop:structure:subexp:sym} for $\la$ close to 1. On the other hand, for the Laplacian case, the extra randomness created by neighbor reshuffling will be limited, and hence we will need the full strength of Proposition \ref{prop:structure:subexp:lap} for any given $\la$.

\section{Treatment for moderate primes: proof of Propositions \ref{prop:structure:subexp:sym}  and \ref{prop:structure:subexp:lap}}\label{section:normalvector}

To prepare for the proofs we first introduce a decomposition trick (originating from \cite{Ver}) which will be useful. Let $\Ba,\BY_0$ be fixed vectors in $\F_p^N$, and let $G_N$ be either $M_N, A_N$ or $L_{N}$. Assume that we would like to bound the probability of the event $G_N\Ba=Y_0$. For this, for any $I \subset [N]$ we can write  
$$G_N =\Big(\begin{array}{cc} G_{([N]\bs I) \times ([N]\bs I) }& G_{([N]\bs I) \times I} \\ G_{I \times ([N]\bs I)} & G_{I \times I} \end{array}\Big).$$ 
The equation $(G_N \Ba)_{[N]\bs I} = (Y_0)_{[N]\bs I}$ 
can be written as
\begin{equation}\label{eqn:i.i.d.decomposition:0}
G_{([N]\bs I) \times ([N]\bs I) } \Ba_{[N]\bs I} +  G_{([N]\bs I) \times I} \Ba_I =(Y_0)_{[N]\bs I}.
\end{equation}

\begin{figure}

\centering

\begin{tikzpicture}
\draw (0,0) -- (4,0) -- (4,4) -- (0,4) -- (0,0);
\draw (0,1) -- (4,1);
\draw (3,0) -- (3,4);
\draw (0,4) -- (4,0);
\draw (0,4) -- (1.5,4) node[above]{$[N]\bs I$};
\draw (3,4) -- (3.5,4) node[above]{$I$};
\draw [fill=black] (3,1) -- (4,1) -- (4,4) -- (3,4);
\draw (4,1)--(4,2.5) node[right=0.1in]{$G_{([N]\bs I) \times I}$}; 
\end{tikzpicture}
\caption{Decomposition into i.i.d. parts.}
\label{figure:rec}
\end{figure}
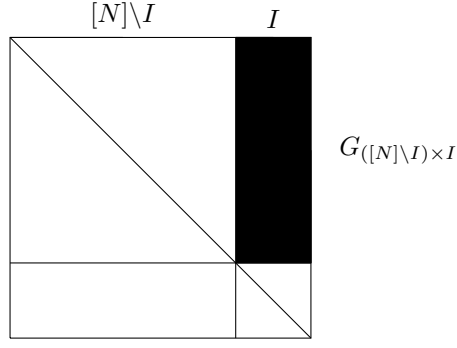

In the non-Laplacian case, we condition on the entries of $G_{([N]\bs I) \times ([N]\bs I)}$ (see Figure \ref{figure:rec}). %
Then we can view \eqref{eqn:i.i.d.decomposition:0} as $ G_{([N]\bs I) \times I} \Ba_I  = D$, where $D$ is deterministic and the entries of $ G_{([N]\bs I) \times I}$ are independent. 
In the Laplacian case, we condition on all  $x_{kl},$ for  $1\le k<l\le n$ where $(k,l), (l,k) \notin ([N]\bs I) \times I $,
 and we let $X_j$ be the $j$th row of
$ G_{([N]\bs I) \times I}$.  
After conditioning,
\eqref{eqn:i.i.d.decomposition:0} is equivalent to
\begin{align}\label{E:NI}
\sum_{i\in I}  {x_{ji}a_i}&=d_j \quad \textrm{ for all $j\in [N]\setminus I$  (non-Laplacian case), or} \\
-x_{j(n+1)}a_j+\sum_{i\in I} {x_{ji}(a_i-a_j)}&=d_j \quad \textrm{ for all $j\in [N]\setminus I$  (Laplacian case)}, \notag
\end{align}
for some deterministic $d_i$.  Also, the $x_{ji}$ above 
are all independent.
In the non-Laplacian case we have
\begin{equation}\label{eqn:i.i.d.decomposition:nL}
\P(G_N\Ba=Y_0)\le  \sup_D\P(G_{([N]\bs I) \times I} \Ba_I  = D | G_{([N]\bs I) \times ([N]\bs I)} ) \le%
 (\frac{1}{p}+\rho(\Ba_I))^{N-|I|} ,
\end{equation}
and in the Laplacian case we have
\begin{equation}\label{eqn:i.i.d.decomposition:L}
\P(G_N\Ba=Y_0)\le  \sup_{\{d_j\}_j  }\P(X_j \cdot (\Ba_I-a_j\1)=d_j  \textrm{ for all $j\in [N]\setminus I$} ) \le
 (\frac{1}{p}+\rho_L(\Ba_I))^{N-|I|}.
\end{equation}

\begin{lemma} \label{lemma:sparse:2}%
 Let $G_N$ be either the random matrix $M_N, A_N$ from Theorem \ref{theorem:cyclic:sym} or Theorem \ref{theorem:cyclic:alt}, or the random matrix $L_N$
 that is the $[N] \times [N]$ submatrix of the $L_n$ from Theorem \ref{theorem:cyclic:lap}. For every sufficiently small
  constant $\delta>0$ (where how small might depend on $\al$ in the symmetric and skew-symmetric case), for $N$ sufficiently large depending on $\al$ and $\delta$ the following holds. %
  Let $Y_0\in \F_p^N$ be a fixed vector. Let $N^{-\delta}\leq \la_N\leq 1$ be a parameter that might depend on $N$, and let $p$ be a prime where 
$$\max\{\frac{1}{\delta^3}, \frac{1}{\la_N^3}\} \leq p\leq \exp(N^{\delta}).$$
Then with probability at least $1 - \exp(-N^{1-\delta})$, for any 
 non-zero  $\Bv\in\F_p^N$ such that $G_N \Bv=Y_0$ in at least $N -N^{\delta}$ coordinates, 
\begin{enumerate}
\item $|\supp(\Bv)| \geq N-\la_N N$ in the $M_N,A_N$ case;
\vskip .05in 
\item $|\supp(\Bv-a\1)| \geq N-\la_N N$ for any $a\in \F_p$ in the $L_N$ case. In other words the highest multiplicity among the components of $\Bv$ is at most $\la_N N$.
\end{enumerate} 
\end{lemma}

\begin{remark}\label{remark:sparse:constant}
Note that we can take $\la_N$ to be a constant in the above result, which yields an analog of (2)  of Lemma \ref{lemma:sparse:1} in which $\delta$ (of that result) can be close to 1 (unlike in Lemma \ref{lemma:sparse:1} itself),   assuming that $p$ is sufficiently large but not too large.
\end{remark}

\begin{proof}[Proof of Lemma \ref{lemma:sparse:2}] It suffices to assume $\la_N \le \delta$ because the statements are weaker for $\la_N>\delta$.
Let $\tilde{\CE}$ be the event that there is a 
 non-zero  $\Bv\in\F_p^N$ such that $G_N \Bv=Y_0$ in at least $N -N^{\delta}$ coordinates and
 $\Bv$ has at least $\lambda_N N$ coordinates of the same value.
  We will show $\P(\tilde{\CE})\leq e^{-N^{1-\delta}}$, which implies  the lemma since if the event of the lemma fails, then $\tilde{\CE}$ occurs.
We take $\delta <1/4$,  and will indeed take it smaller later, depending on at most $\alpha$.  Throughout the proof, we always assume that $N$ is sufficiently large given $\alpha$ and our choice of $\delta$.

First,  we will bound the probability of the following event $\mathcal{E}$:
there exists a $\Bv$ whose first $\lceil \lambda_N N \rceil$ coordinates are the same and $G_N\Bv=Y_0$.
Let $k:=\lfloor \lambda_N N/4\rfloor$. 
We can partition $[\lceil \lambda_N N \rceil+1,N]$ into subsequences $J_1,\dots,J_\ell$ of consecutive numbers so that $|J_i|=k$ for $i<\ell$ and $k\leq |J_\ell| <2k$.
We have
$$k\geq \lambda_N N/8 \quad \textrm{and} \quad \ell<8N^{\delta}.$$

Given $\Bv$,  we partition $[N]$ into two subsets, $I_m$ (mixed) and $I_s$ (sparse).  
We let $I_s$ be the union of the $J_i$ such that $\Bv_{J_i}$ has highest entry multiplicity  larger than $|J_i|-N^{1-5\delta}$,  along with $[\lceil \lambda_N N \rceil]$. 
 We let $I_m=[N]\setminus I_s$.
We write $\Bv_m:=\Bv_{I_m}$ and $\Bv_s:=\Bv_{I_s}$.

\begin{claim}\label{C:vspossibilities}
Given $I_s$, there are at most $e^{N^{1-2\delta}}$
 possibilities for $\Bv_s\in \F_p^N$ for $\Bv$ giving that $I_s$ and whose first $ [\lceil \lambda_N N \rceil]$ coefficients are the same.
\end{claim}
\begin{proof}
The number of $\Bv_s$ is at most  %
$$\binom{2k}{\lfloor N^{1-5\delta}\rfloor}^{\ell}p^{\ell+1} (p^{N^{1-5\delta}})^{\ell}  \leq
N^{8N^{1-4\delta}} e^{16N^{2\delta}+8N^{1-3\delta} }.$$
On the left above,  the first factor is for the index choices of the coordinates that are not contributing to the highest entry multiplicity in $|J_i|$, the second
factor is  for the choices of the entry value with  highest entry multiplicity in each $J_i$ and in $ [\lceil \lambda_N N \rceil]$,  the third factor is for the choices of the other entry values in the $J_i$. 
The claim follows.
\end{proof}

{\bf Case 1.} 
Let $\mathcal{E}_1$ be the event that there exists a non-zero $\Bv$ whose first $\lceil \lambda_N N \rceil$ coordinates are the same,   $G_N\Bv=Y_0$, and the resulting $I_m$ (from $\Bv$) is empty.
By Claim~\ref{C:vspossibilities},  there are at most $e^{N^{1-2\delta}}$ choices of $\Bv$ that can arise with empty $I_m$.
For each of these choices of $\Bv$, we apply \eqref{E:NI} to bound $\P(G_N \Bv=Y_0)$, letting $I$ be the coordinate of a single non-zero coefficient in $\Bv$.
In the Laplacian case, note that in \eqref{E:NI} since $a_i\ne 0$, if $a_i-a_j=0$, then $a_j\ne 0$.
Then we have
$$
\P(G_N \Bv=Y_0)\leq (1-\alpha)^{N-1}.
$$
Summing over the possible $\Bv$, we conclude
$$
\P(\CE_1)\leq e^{N^{1-2\delta}} (1-\alpha)^{N-1}\leq e^{-c_1N}
$$
for some $c_1>0$ depending on $\alpha$.

{\bf Case 2.} 
Let $\mathcal{E}_2$ be the event that there exists a $\Bv$ whose first $\lceil \lambda_N N \rceil$ coordinates are the same,   $G_N\Bv=Y_0$, and the resulting $I_m$ is not empty.
 We have $|I_m| \geq \lambda_N N/8  \geq N^{1-\delta}/8$
 and $|I_s|\geq\lambda_N N$.%

We now describe a function $f$ from subsets of $[N]$ (that can occur as $I_s$) to subsets of $[N]$.  We describe $f(I_s)$, 
and write $I_m$ for $[N]\setminus I_s$, but note that $f$ does not depend
on $\Bv$.  If $|I_s|>|I_m|$, we let $f(I_s)$ be the first $|I_m|$ elements of $I_s$.
Otherwise, we let $J_*$ be the first $J_i$ in $I_m$ and we let $f(I_s)$ be the union of $I_s$ and the first $|I_m|-|I_s|$ elements of $I_m$ that are not in $J_*$.
This is possible because $|J_*|\leq \lambda_N N/2\leq  |I_s|.$ In either case, we have 
$|f(I_s)|=|I_m|$ and if $I_m$ is non-empty, it contains some $J_i$ that does not intersect $f(I_s).$

Let $\CE_{drop}$ be the event that  there is a square submatrix $G_{A\times B}$ of $G_N$
of dimension $\geq \lfloor N^{1/2+\delta} \rfloor$
 with rank less than $|A| - \lfloor N^{1/2+\delta} \rfloor$.
Lemma \ref{lemma:quadratic} tells us that
$
\P(\CE_{drop})\leq N e^{-c''(\lfloor N^{1/2+\delta} \rfloor)^2+c'N }
$ (for some $c'',c' >0$ depending on $\alpha$),
and thus 
$
\P(\CE_{drop})\leq  e^{-c_{drop} N^{1+2\delta} },
$
for some $c_{drop}>0$ depending on $\alpha$.%

We wish to bound the probability of $\mathcal{E}_2\setminus \CE_{drop}$.
Since for $\Bv$ causing $\mathcal{E}_2$, we have  $|I_m|> N^{1-\delta}/8$, %
we have $|I_m|\geq  \lfloor N^{1/2+\delta} \rfloor$.
So outside of $\CE_{drop}$,  the matrix $G_{f(I_s) \times I_m}$ has rank $\geq |I_m|-  \lfloor N^{1/2+\delta} \rfloor$, 
and hence it has a square submatrix $G_{{{{{I'_s}}}} \times {{I'_m}}}$ of dimension $|I_m| -  \lfloor N^{1/2+\delta} \rfloor$ which has full rank.

Given subsets $M$, ${{I'_m}}$, and ${{{{I'_s}}}}$ of $[N]$
such that $|{{I'_m}}|=|{{{{I'_s}}}}|$ and ${{I'_m}}\sub M$ and ${{{{I'_s}}}}\sub f([N]\setminus M)$, and such that
$M$ is non-empty and can occur as $I_m$,  let $\mathcal{E}_{2.1}(M,{{I'_m}},{{{{I'_s}}}})$ be the event that there is a $\Bv\in \F_p^N$ such that
the first $\lceil \lambda_N N \rceil$ entries of $\Bv$ are the same, $G_N\Bv=Y_0$, and $I_m=M$,  and $G_{{{{{I'_s}}}}\times {{I'_m}}}$ is full rank. 
In particular, above we just saw that $\mathcal{E}_2\setminus \CE_{drop}$ implies 
$\mathcal{E}_{2.1}(M,{{I'_m}},{{{{I'_s}}}})$ for some $M$, ${{I'_m}}$, and ${{{{I'_s}}}}$ with $|M|-|{{I'_m}}|=  \lfloor N^{1/2+\delta} \rfloor$
and $N-|M|\geq \lambda_N N$.

\begin{claim}\label{claim:cond:fix}
For all $M,{{I'_m}},{{{{I'_s}}}}$ such that $\mathcal{E}_{2.1}(M,{{I'_m}},{{{{I'_s}}}})$ is defined, we have
$$
\P(\mathcal{E}_{2.1}(M,{{I'_m}},{{{{I'_s}}}})) \leq e^{N^{1-2\delta}} p^{|M|-|{{I'_m}}|}\left( O(1)/\sqrt{N^{1-5\delta}}+1/p \right)^{N-|M|-\lambda_N N/2}
$$
\end{claim}
\begin{proof}
For $\Bv$ with the first $\lceil \lambda_N N \rceil$ coordinates the same and $I_m=M$, there are at most $e^{N^{1-2\delta}}$
possible values of $\Bv_s$  by Claim~\ref{C:vspossibilities}
 and there are at most $p^{|M|-|{{I'_m}}|}$ possible values for $\Bv_{I_m\setminus {{I'_m}}}.$
We fix a choice of $\Bv_s$ and $\Bv_{I_m\setminus {{I'_m}}},$ and we will bound the desired probability just for $\Bv$ with these values.
We let $F=f([N]\setminus M)$.
We condition on $G_{F\times[N]}$ and compute the conditional probability.    
If $G_{{{{{I'_s}}}}\times {{I'_m}}}$ is not full rank, then the desired (conditional) probability is $0$.  
Otherwise, the equation $(G_{{{{{I'_s}}}} \times {{I'_m}}},G_{{{{{I'_s}}}} \times ([N]\bs {{I'_m}})})\Bv=(Y_0)_{{{{{I'_s}}}}}$ implies
\begin{equation}\label{eqn:cond:fix}
\Bv_{{{I'_m}}}=(G_{{{{{I'_s}}}} \times {{I'_m}}}^{-1})\left(  (Y_0)_{{{{{I'_s}}}}} - G_{{{{{I'_s}}}} \times ([N]\bs {{I'_m}})} \Bv_{[N]\setminus {{I'_m}}} \right).
\end{equation}
Everything on the right above is determined, and thus so is $\Bv_{{{I'_m}}}$, and we are considering a single fixed value of $\Bv$.
We let $J_*$ be one of the $J_i$ that is in $M$ but has no intersection with $F$.

For each $i\in ([N]\setminus F) \setminus J_*$, we let $X_i$ be the $i$th row of $G_N$.  The equation $X_i\Bv=(Y_0)_i$ implies
\begin{equation}\label{E:toget}
\sum_{j\in J_*} x_{ij}(v_j-v_i)=(Y_0)_i -\sum_{j\in [N]\setminus (J_* \cup \{ i\})} x_{ij}(v_j-v_i)+\sum_{j\in [n+1]\setminus [N]} x_{ij}v_i,
\end{equation}
where $v_i$ are the entries of $\Bv$.

The $x_{ij}$ for $i\in  ([N]\setminus F) \setminus J_*$ and $j\in J_*$ are all independent.  
We further condition on $x_{ij}$ for $i\in  ([N]\setminus F) \setminus J_*$ and $j\not \in J_*$.
Since $J_*\cap F=\emptyset$, none of the $x_{ij}$ for $i\in ( [N]\setminus F) \setminus J_*$ and $j\in J_*$ have been conditioned on.
Thus after our conditioning, the probability of \eqref{E:toget} holding is at most $\rho_L(\Bv_{J_*}).$
We have $N-|F|-|J_*|$ independent such equations that are implied by $G_N\Bv=Y_0$.
Thus the probability that, for a given $\Bv_s$ and $\Bv_{I_m\setminus {{I'_m}}},$ we have the event in the claim, is  at most
$$
\rho_L(\Bv_{J_*})^{N-|F|-|J_*|}.
$$
Since $J_*$ is one of the $J_i$ that is a subset of $I_m$, by definition of $I_m$ and
the highest multiplicity of coordinates of $\Bv_{J_*}$ is $<|J_*|-N^{1-5\delta}$, by Theorem~\ref{theorem:LO} we have that
$$
\rho_L(\Bv_{J_*})\leq O(1)/\sqrt{N^{1-5\delta}}+1/p.
$$
We have $|J_*|\leq \lambda_N N/2$ and $|F|=|M|$.  Thus the probability that, for a given $\Bv_s$ and $\Bv_{I_m\setminus {{I'_m}}},$ we have the event in the claim, is at most
$$
\left( O(1)/\sqrt{N^{1-5\delta}}+1/p \right)^{N-|M|-\lambda_N N/2}.
$$
Summing over the possible values for $\Bv_s$ and $\Bv_{I_m\setminus {{I'_m}}}$ (which are bounded in number above), we obtain the claim.
\end{proof}

The event $\CE_2$ is the union of $\CE_{drop}$ with $\mathcal{E}_{2.1}(M,{{I'_m}},{{{{I'_s}}}})$ over $M,{{I'_m}},{{{{I'_s}}}}$ with $|M|-|{{I'_m}}|= \lfloor N^{1/2+\delta} \rfloor$
and $|N|-|M|\geq \lambda_N N$ and $M$ a possible value of $I_m$.  There are at most $2^\ell \leq 2^{8N^{\delta/2}}$ possible values of $I_m$ (and hence $M$).
Given $M$, there are at most $\binom{|M|}{\lfloor N^{1/2+\delta} \rfloor}$ choices of ${{I'_m}}$, and $\binom{|f([N]\setminus M)|}{\lfloor N^{1/2+\delta} \rfloor}=\binom{|M|}{\lfloor N^{1/2+\delta} \rfloor}$ choices of ${{{{I'_s}}}}$.  Hence,
we have
\begin{align*}
\P(\CE_2)&\leq   e^{-c_{drop} N^{1+2\delta} } + 2^{8N^{\delta/2}}  N^{2N^{1/2+\delta}}
e^{N^{1-2\delta}} e^{N^{1/2+2\delta}}\left( O(1)/\sqrt{N^{1-5\delta}}+1/p \right)^{\lambda_N N/2}\\
&\leq e^{-c_{drop} N^{1+2\delta} } + e^{N^{1-\delta}}\left( O(1)/\sqrt{N^{1-5\delta}}+1/p \right)^{\lambda_N N/2},
\end{align*}
for $\delta$ sufficiently small. %
Thus, there is some $c_0>0$ (depending only on $\alpha$) such that for $\delta$ sufficiently small and $N$ sufficiently large given $\delta$,
$$
\P(\CE)\leq  \P(\CE_1)+ \P(\CE_2) \leq e^{-c_0N} +e^{N^{1-\delta}}\left( O(1)/\sqrt{N^{1-5\delta}}+1/p \right)^{\lambda_N N/2}.
$$

To bound the probability of $\tilde{\CE}$ and prove the lemma, we need to account also for the analog of $\CE$ where the first $\lceil \lambda_N N \rceil$ entries are replaced with any $\lceil \lambda_N N \rceil$ entries, and $Y_0$ is replaced by any vector that shares at least $N-N^{\delta}$ coordinates with it.  
This gives
$$
\binom{N}{\lceil \lambda_N N \rceil} \binom{N}{\lfloor N^{\delta} \rfloor} p^{\lfloor N^{\delta} \rfloor}
$$
total events with the same probability bound as we showed above for $\CE$, whose union is $\tilde{\mathcal{E}}$.
 We then have
\begin{align}
 \P(\tilde{\CE}) \notag
\leq &\left(\binom{N}{\lceil \lambda_N N \rceil} \binom{N}{\lfloor N^{\delta} \rfloor} p^{\lfloor N^{\delta} \rfloor}\right)\left(e^{-c_0N} +e^{N^{1-\delta}}\left( O(1)/\sqrt{N^{1-5\delta}}+1/p \right)^{\lambda_N N/2} \right)\\
\leq & \left( \frac{Ne}{ \lambda_N N }   \right)^{\lambda_N N+1} N^{N^{\delta}}e^{N^{2\delta}}\notag
\left(e^{-c_0N} +e^{N^{1-\delta}}\left( O(1)/\sqrt{N^{1-5\delta}}+1/p \right)^{\lambda_N N/2} \right)\\
\leq &  e^{\lambda_N N+1 +\lambda_N N\log\lambda_N^{-1} +\log\lambda_N^{-1}  +N^{\delta}\log N
+N^{2\delta}} \label{E:tildeE}
\left(e^{-c_0N} +e^{N^{1-\delta}}\left( O(1)/\sqrt{N^{1-5\delta}}+1/p \right)^{\lambda_N N/2} \right).
\end{align}

The $e^{-c_0N}$ factor (and, e.g., its 20th root) is much smaller than all the terms it is multiplied by (in \eqref{E:tildeE}) except perhaps 
$e^{\lambda_N N+\lambda_N N\log\lambda_N^{-1}}$, but if we take $\delta$ small enough we can guarantee  $\lambda_N N+\lambda_N N\log\lambda_N^{-1}\leq c_0N/20$.  We have that
$( O(1)/\sqrt{N^{1-5\delta}}+1/p)$ is bounded by either $O(1)/\sqrt{N^{1-5\delta}}$ or $2/p$,
and so we consider these cases separately.  
We have
$$
\left(\frac{O(1)}{\sqrt{N^{1-5\delta}}}\right)^{\lambda_N N/2}\leq e^{O(\lambda_N N)- \lambda_N N(1-5\delta) \log N /4  }.
$$ 
The $e^{- \lambda_N N(1-5\delta) \log N /4 }$ factor 
(and its 20th root)
is much smaller than all the terms it is multiplied by (from \eqref{E:tildeE})  except perhaps $e^{\lambda_N N\log\lambda_N^{-1}}$.  Since $\lambda_N^{-1}\leq N^{\delta}$,
for all $\delta>0$ small enough we have $\lambda_N N\log\lambda_N^{-1}\leq \lambda_N N(1-5\delta) \log N /80$. 

We have
$$
\left(\frac{2}{p}\right)^{\lambda_N N/2} \leq (2\lambda_N^{-3})^{\lambda_N N/2}=e^{(\log 2)\lambda_N N/2 -1.5 \lambda_N N\log \lambda_N^{-1}}.
$$ 
The $e^{-1.5 \lambda_N N\log \lambda_N^{-1}}$ factor cancels partially with 
$e^{\lambda_N N\log \lambda_N^{-1}}$ (from \eqref{E:tildeE}), giving a remaining term of 
$e^{-.5 \lambda_N N\log \lambda_N^{-1}}$, whose 20th root 
is smaller than all the terms it is multiplied by (above and in \eqref{E:tildeE}).
Thus we conclude 
$$
 \P(\tilde{\CE})  \leq e^{-c_0N/20}+e^{-\lambda_N N\log \lambda_N^{-1}/100}.
$$
Since $x\log x^{-1}$ is decreasing for $x\geq e$ and $\lambda_N\geq N^{-\delta}$,  the lemma follows.
\end{proof}

We remark that in the proof of Lemma \ref{lemma:sparse:2} above the upper bound assumption on $p$ is important. One cannot  expect a similar conclusion for extremely large $p$. 

\subsection{Proof of  Proposition \ref{prop:structure:subexp:sym} and Proposition \ref{prop:structure:subexp:lap}}\label{section:structure:sym}

    Generally speaking, some part of our treatment here is motivated by \cite{FJ} where the authors used the machinery of \cite{FJLS} to give an explicit singularity bound for random symmetric matrices by passing to matrices modulo a very large prime (of order $e^{N^c}$). However, there are significant differences in our results: (1) We study the discrepancy $\rho$, which controls $\sup_a |\P(X=a)-1/p|$ instead of the concentration probability $\sup_a \P(X=a)$. This is the correct concept to measure how a probability distribution in a finite group deviates from the uniform distribution. (2) Our result works as long as $p$ is sufficiently large (instead of focusing only on $p$ of order $e^{N^c}$; in fact if we only focus on the skew-symmetric or symmetric case,  one may be able to take $p\ge 3$, but we will not elaborate on this). As we have seen in Section \ref{section:method}, this wide  range of $p$ is in fact important to our main theorems. (3) Above all, perhaps the most innovative part of this section is to carry out the study for the Laplacian model. Here we have to study the {\it segmentwise} structures of the generalized normal vectors, that is $\rho(Z_{I_0})$, over any fixed set $I_0$ of size $\la n$. With these goals in mind, we will need to develop additional tools and borrow some more ideas from \cite{LMNg}. The most innovative part of this subsection is the proof of Proposition \ref{prop:structure:subexp:lap'} where we exploit structures using the ``propagation method" of Lemma \ref{lemma:sparse:2}.

Let $\Bz=(z_1,\dots,z_N)$ be any vector in $\F_p^N$. We first record an elementary relation (recalling $\rho(.)$ from \eqref{eqn:rho}).
    		\begin{fact}\label{fact:comparison:rho} For any $I \subset [n]$ we have 
		$$\rho(\Bz_I) \ge \rho(\Bz)$$
		and
		$$\rho_L(\Bz_I) \ge \rho_L(\Bz)$$
 		\end{fact}

	We next need the following definitions and results from \cite{FJLS}.

    		\begin{definition} For an $\Bw=(w_1,\dots, w_N) \in \F_p^N$ and $k \in \mathbb{N}$, let $R_k(\Bw)$ be the number of solutions of the form $(s,(i_1,\dots,i_{2k}))\in \{+,-\}^{2k}\times [N]^{2k}$ 
    		to $ \pm w_{i_1} \pm \dots \pm w_{i_{2k}} =0 \pmod{p}$. More generally, for given $\delta \in [0,1]$ we define $R_k^\delta(\Bw)$ to be the number of solutions to 
    			$$
    			\pm w_{i_1} \pm w_{i_2} \cdots \pm w_{i_{2k}} = 0 \, \text{mod} \, p
    			$$
    			that satisfy $|\{i_1, \dots, i_{2k}\}| \geq (1+ \delta)k$.
    		\end{definition}
		It is easy to show that $R_k(\Bw)$ is never much larger than $R_k^{\delta}(\Bw)$. 
    		\begin{lemma}\cite[Lemma 1.6]{FJLS}\label{Lemma:FJLS:delta}
    			For all integers $k,N$ with $k \leq N/2$ and any prime $p$, $\Bw \in \F_p^N$ and $\delta \in [0,1]$, %
    			$$
    			R_k(\Bw) \leq R_k^{\delta}(\Bw) + (40 k^{1-\delta} N^{1+\delta})^k.
    			$$
    		\end{lemma}

The following Hal\'asz-type result connects the above combinatorial structure to $\rho(.)$ (see \cite[Theorem 5.1]{LMNg} \footnote{Although in  \cite[Theorem 5.1]{LMNg} the statement is for $\rho=\sup_{a \in \Z/p\Z} |\P(\mu_1w_1 + \cdots + \mu_Nw_N = a)-\frac{1}{p}|$, the proof works identically for $\rho$ as in \eqref{eqn:rho}. The constant $C$ can be taken to be $\al^{-1}C_0$ where $C_0$ is absolute.}).

    		\begin{theorem}\label{theorem:Halasz} Let $k\ge 1$ be an integer. Let $f:\Z^+ \to \R_{\ge 1}$ be any function such that $f(x) \le x/100$.
    			For any non-zero vector $\Bw= (w_1, \cdots, w_N) \in \F_p^N$ we have 
    			$$\rho(\Bw) \le C \frac{R_k(\Bw)}{ 2^{2k} N^{2k} \sqrt{f(|\supp(\Bw)|)}} + e^{-f(|\supp(\Bw)|)/2}$$ 
    			for $k\le N/f(|\supp(\Bw)|)$, where $C$ is a constant depending on $\al$.
        		\end{theorem}

	 As we will have to deal with subvectors many times in this section, for convenience by $\Bb\subset \Ba$ we mean that $\Bb$ is a truncation of $\Ba$. We use $|\Ba|$ to mean the dimension of the vector $\Ba$, that is 
	 $$|\Ba| := |\suppi(\Ba)|.$$

	\begin{definition}[Choices of parameters]\label{def:parameters} 
	Let $k,s_1,s_2,d \in [n]$ and such that $s_1\le s_2$. The parameters can be chosen in a flexible manner but for instance our arguments will work when 
	\begin{itemize} 
	\item $\delta>0$ is sufficiently small given $\al$; $N$ is sufficiently large given $\alpha,\delta$; 
	\vskip .01in
	\item and also	 \begin{align}\label{eqn:parameter}
	k = \Theta(N^{1/4}), s_1=\Theta(N^{1-4\delta}), s_2=\Theta(N^{1-2\delta}),  d \geq N^{2/3} \text{ and } N^{12\delta} \le p\le \exp(N^{\delta/2}).
	 \end{align}
	 \end{itemize}
	 \end{definition}
	 Note that we choose $p\ge N^{12\delta}$ here to exploit Lemma \ref{lemma:sparse:2} in the setting that $\la_N=N^{-4\delta}$. 
	 In our proof of  Propositions \ref{prop:structure:subexp:sym} and  \ref{prop:structure:subexp:lap} later we will extend this range to cover all sufficiently large $p$.
	  
	 In what follows $C$ is a sufficiently large constant that is allowed to depend on $\al$ and can be different in each statement. For each $t>0$, 
	 given $s_1\le s_2 \le d$ we define
\begin{align*}
\BG_{d,k,s_1,s_2,\ge t}(N):=& \Big\{\Ba \in \F_p^N, |\supp(\Ba)|=d: \forall \Lambda \subset \supp(\Ba) \textrm{ such that } s_1 \le |\Lambda| \le s_2,\\
&\text{ with $\Bb= \Ba|_\Lambda$, we have } \rho(\Bb) \ge 3C(t+1)\sqrt{k}/(p\sqrt{s_1})\Big \}.
\end{align*}
Note that $t$ is a varying integral parameter and we take $t\le p$. In a way, the conditions say that there are lots of structures among the non-zero entries of $\Ba$. 
The next result shows that this set has small size.

\begin{proposition}\label{prop:counting:moderate:sym} For $C$ sufficiently large given $\alpha$, the following holds for any sufficiently small $\delta$.
 Let $p$ 
be a prime and let $t\ge 1$ and $1\le s_1\le s_2 \le d \le N, $ and $1\le k \le s_1/2$ 
such that 
\begin{equation}\label{eqn:pt}
\frac{p}{t} \sqrt{\frac{s_1}{k}} \le \min\{e^{s_1/2k}, (s_1/Ck)^{(1-\delta)k}\}.
\end{equation}
Then we have that (when $d=s_2$)
\begin{equation}\label{eqn:d=s_2}
|\BG_{d,k,s_1,d,\ge t}(N)| \le \binom{N}{d} p^{d} (\delta t)^{-d+s_1} 
\end{equation}
and in general
\begin{equation}\label{eqn:d>s_2}
|\BG_{d,k,s_1,s_2,\ge t}(N)| \le \binom{N}{d} p^{d+s_2} (\delta t)^{-d(1-s_1/s_2)}.
\end{equation}
\end{proposition}
We will apply this result for the choice of parameters from Eq.~\eqref{eqn:parameter} and with $t$ satisfying $1\le t\le p$. 

\begin{proof} As suggested by Theorem \ref{theorem:Halasz}, it is natural to pass to the following analog of $\BG_{d.k,s_1,s_2,\ge t}(N)$,
\begin{align*}{\BG'}_{d,k,s_1,s_2,\ge t}(N):=& \Big\{\Ba \in \F_p^N, |\supp(\Ba)|=d, \forall \Lambda \subset \supp(\Ba) \textrm{ such that } s_1 \le |\Lambda| \le s_2,\\
&\text{ with $\Bb= \Ba|_\Lambda$, we have } R_k^{\delta}(\Bb) \ge t \frac{2^{2k} |\Bb|^{2k}}{p}\Big\}.
\end{align*}
We observe that 
$$\BG_{d,k,s_1,s_2,\ge t}(N) \subset {\BG'}_{d,k,s_1,s_2,\ge t}(N).$$
Indeed, assume that $\Ba \notin {\BG'}_{d,k,s_1,s_2,\ge t}(N)$ and $|\supp(\Ba)|=d$. Then there exists $\Bb\subset \Ba_{\supp(\Ba)} \textrm{ such that } s_1 \le |\Bb| \le s_2$ and $R_k^{\delta}(\Bb) < t \frac{2^{2k} |\Bb|^{2k}}{p}$. Note $|\Bb|=|\supp(\Bb)|$.
By  Lemma \ref{Lemma:FJLS:delta} and Theorem \ref{theorem:Halasz}  (with $f(x)=x/k$) for  $C$ sufficiently large depending on $\alpha$, we have

		\begin{align*}
		\rho(\Bb) &\leq  \frac{C(R_k^\delta(\Bb) + (C k^{1-\delta}|\supp(\Bb)|^{1+\delta})^k }{2^{2k} |\Bb|^{2k} \sqrt{|\supp(\Bb)| /k}} + e^{-|\supp(\Bb)|/2k}\\
		& \leq  \frac{Ct \sqrt{k}}{p \sqrt{|\supp(\Bb)|}} + \frac{(C k^{1-\delta})^k }{ |\supp(\Bb)|^{(1-\delta)k}} + e^{-|\supp(\Bb)|/2k}\\
		& \leq  \frac{Ct\sqrt{k}}{p\sqrt{s_1}} + (\frac{Ck}{s_1})^{(1-\delta)k} + e^{-|\supp(\Bb)|/2k} \\
		& \le \frac{3Ct \sqrt{k}}{p\sqrt{s_1}}.
		\end{align*}   
Hence it suffices to bound the cardinality of $ {\BG'}_{d,k,s_1,s_2,\ge t}(N)$. To this end we just need to apply \cite[Theorem 1.7]{FJLS} to obtain Eq. \eqref{eqn:d=s_2} and \cite[Corollary 3.11]{FJ} to obtain Eq. \eqref{eqn:d>s_2}.
\end{proof}
Here we remark that the arguments of \cite[Theorem 1.7]{FJLS} and \cite[Corollary 3.11]{FJ} to bound $\BG'$ use a nice double counting trick to exploit the largeness of $R_k^\delta(\Bb)$. Very roughly speaking, because $R_k^\delta(\Bb)$ is large for all $\Bb\subset \Ba_{\supp(\Ba)}$, there is a small index set $I$ over which if we fix the values of $w_i$ for $i\in I$, the values of other $w_j$ will be determined. 

While Theorem \ref{prop:counting:moderate:sym} above will suffice to study the symmetric and skew-symmetric case, for the Laplacian case we will need to modify a bit. In what follows  $L$ stands for Laplacian, we define
\begin{align*}
\BG_{L,d,k,s_1,s_2,\ge t}(N):=& \Big\{\Ba =(a_1,\dots, a_N)\in \F_p^N, |\supp(\Ba)|=d, \forall \Bb\subset \Ba_{\supp(\Ba)} \textrm{ such that } s_1 \le |\Bb| \le s_2 \\
& \text{ we have }\rho_L(\Bb) \ge 3C(t+1)\sqrt{k}/(p\sqrt{s_1})\Big\}.
\end{align*}
and 
\begin{align*}{\BG'}_{L,d,k,s_1,s_2,\ge t}(N):=& \Big\{\Ba \in \F_p^N, |\supp(\Ba)|=d, \forall \Lambda \subset \supp(\Ba) \textrm{ such that } s_1 \le |\Lambda| \le s_2,\\
&\text{ with $\Bb= \Ba|_\Lambda$, we have } R_k^{\delta}(\Bb-a\1) \ge t \frac{2^{2k} |\Bb|^{2k}}{p} \text{ for some $a$}\Big\}.
\end{align*}
Arguing as in the non-Laplacian case, by Theorem \ref{theorem:Halasz} we have 
\begin{equation}\label{eqn:GG':L}
\BG_{L,d,k,s_1,s_2,\ge t}(N) \subset {\BG}'_{L,d,k,s_1,s_2,\ge t}(N).
\end{equation}

We have the following

\begin{proposition}\label{prop:counting:moderate:lap} For $C$ sufficiently large given $\alpha$, the following holds for any sufficiently small $\delta$.
 Let $p$ 
be a prime and let $1\le t\le p$ and $1\le s_1\le s_2 \le d \le N, $ and $1\le k \le s_1/2$ 
such that %
$$\frac{p}{t} \sqrt{\frac{s_1}{k}} \le \min\{e^{s_1/2k}, (s_1/Ck)^{(1-\delta)k/2}\}.$$
We have that 
$$|\BG_{L,d, k,s_1,s_2,\ge t}(N)| \le p \binom{N}{d } 2^{d+4} p^{d+s_2} (\delta t)^{-d(1-s_1/s_2)}.$$
\end{proposition}
\begin{proof} 
Consider $\Ba\in \BG_{L,d,k,s_1,s_2,\ge t}(N)$ with $\supp(\Ba) = [d]$.
Without loss of generality we assume that $d$ is even (as the reader will see, if this is not the case we just need to freeze another coordinate of $\Ba$). %
Define $\Ba'=(a_1-a_2,a_3-a_4,\dots,a_{d-1}-a_d)$ and $\Ba''=(a_2-a_3,a_4-a_5,\dots, a_{d}-a_1)$. 

Let $\Bb'$ be any subvector of $\Ba'$ with $s_1/2 \le |\Bb'| \le s_2/2$, that is $\Bb'$ has the form $(a_{i_1}-a_{i_1+1},\dots, a_{i_{|\Bb'|}} - a_{i_{|\Bb'|}+1})$. Then by \eqref{eqn:w-w'}  
\begin{equation}\label{eqn:b':L}
\rho_{\al/2}(\Bb') \ge \rho_{L,\al} (a_{i_1},a_{i_1+1},\dots, a_{i_{|\Bb'|}},a_{i_{|\Bb'|}+1}) \ge 3C(t+1)\sqrt{k}/(p\sqrt{s_1}),
\end{equation}
where in the last estimate we used the fact that $\Ba \in \BG_{L,d, k,s_1,s_2,\ge t}(N)$.

Given $d'$, we will consider the $\Ba'$ with $|\supp(\Ba')|=d'$.   We always have $d' \le d/2$.  First we consider the dominating case when $s_2/2 \le d'$. We apply Theorem \ref{prop:counting:moderate:sym} with the $(s_1,s_2,d,N,k)$ in that theorem replaced by $(s_1/2,s_2/2,d', d/2,\lfloor k/2 \rfloor)$.
By \eqref{eqn:b':L} and by Theorem \ref{prop:counting:moderate:sym}, the number of $\Ba'$ with $d'=|\supp(\Ba')|$ that can come from $\Ba \in \BG_{L,d, k,s_1,s_2,\ge t}(N)$
is  bounded by
$$  \binom{d/2}{d'} p^{d'+s_2/2} (\delta t)^{-d'(1-s_1/s_2)} \le   \binom{d/2}{d'} p^{d/2+s_2/2} (\delta t)^{-(d/2)(1-s_1/s_2)}=:\CN,$$
as long as $C$ is large enough given $\alpha$.

Secondly, if $ s_1/2 \le d' \le s_2/2$, by \eqref{eqn:d=s_2} the number of such  $\Ba'$ is bounded by
$$  \binom{d/2}{d'} p^{d'} (\delta t)^{-d'+s_1/2} \le  \binom{d/2}{d'} t^{s_1/2} (p/t)^{d'} \delta^{-d'+s_1/2} \le    \binom{d/2}{d'} t^{(d/2)s_1/s_2}(p/t)^{d/2} \delta^{-(d/2)(1-s_1/s_2)}\le \CN.$$
Thirdly, if $d' \le s_1/2$ then the number of such vectors is simply bounded by $ \binom{d/2}{d'} p^{s_1/2} \le \CN$ because as $t\le p$, 
$$p^{d/2+s_2/2} (\delta t)^{-(d/2)(1-s_1/s_2)} \ge  p^{d/2+s_2/2} p^{-d/2} \ge  p^{s_2/2}.$$

Hence in total we see that the number of $\Ba'$ with $d'=|\supp(\Ba')|$ that can come from $\Ba \in \BG_{L,d, k,s_1,s_2,\ge t}(N)$ is at most $3\CN$. Summing over $d'$ we see that the number of $\Ba'$ is at most
$$3 \times 2^{d/2} p^{d/2+s_2/2} (\delta t)^{-(d/2)(1-s_1/s_2)}.$$
We bound similarly for the number of $\Ba''$. Therefore the number of $(\Ba',\Ba'')$ is at most
$$2^{d+4} p^{d+s_2} (\delta t)^{-d(1-s_1/s_2)}.$$
Finally, we remark that when $\Ba'$ and $\Ba''$ are given, then any assignment of $a_1$ would then determine the $\Ba$ they came from. 
Thus the number of $\Ba\in \BG_{L,d,k,s_1,s_2,\ge t}(N)$ with $\supp(\Ba) = [d]$ can be bounded by $ p 2^{d+4} p^{d+s_2} (\delta t)^{-d(1-s_1/s_2)}$.  Multiplying by a factor of $\binom{N}{d}$ to include $\Ba$ with any support of size $d$ gives the result.
\end{proof}

In our application, $\Ba$ will be a generalized normal vector (with respect to some given $Y_0$). In our next steps we will be working with overwhelming events defined below.

\begin{remark}\label{rmk:nsp}
Let $N$ be a positive integer, $\delta>0$, $p$ be a prime such that  $8N^{12\delta} \le p\le \exp(N^{\delta/2})$,   $Y_0\in\F_p^N$. 
\begin{itemize}
\item  Let $G_N$ be either $M_N$ or $A_N$. Then let $\CE_{non-sparse}$ (implicitly depending on  our choice of random matrix model and $N,p,Y_0,\delta$) be the event that 
for every nonzero $\Ba\in\F_p^N$ such that $L_N\Ba=Y_0$ in at least $N-N^\delta$ coordinates,
for any $\Lambda \subset \supp(\Ba)$ with $s_1 \le |\Lambda|$ (for $s_1$ as in Definition \ref{def:parameters})) 
we have 
$$ N-s_1/2 \le |\supp(\Ba)|.$$ 
\vskip .1in
\item Let $\CE_{L, non-sparse}$ (implicitly depending on $N,p,Y_0,\delta$) be the event that 
for every nonzero $\Ba\in\F_p^N$ such that $L_N\Bv=Y_0$ in at least $N-N^\delta$ coordinates,
for any $\Lambda \subset \supp(\Ba)$ with $s_1 \le |\Lambda|$ (for $s_1$ as in Definition \ref{def:parameters})) 
we have 
$$|\Lambda| - s_1/2 \le |\supp(\Ba -w\1)_\Lambda|.$$ 
\end{itemize}
By Lemma \ref{lemma:sparse:2} (applied to $\la_N=n^{-4\delta}/2$) we have
$$\P(\CE_{non-sparse}), \P(\CE_{L, non-sparse}) \ge 1 - \exp(-(N^{1-\delta})).$$
\end{remark}

Now for the set of non-sparse structured vectors, we decompose it into disjoint union of sets of type 
$$\BG_{d,k, s_1,s_2, \ge t_i}(N)\bs \BG_{d, k,s_1,s_2, \ge 2t_i}(N) \mbox{ or } \BG_{L,d, k, s_1,s_2, \ge t_i}(N)\bs \BG_{L,d, k,s_1,s_2, \ge 2t_i}(N)$$ 
from which we will take advantage of the counting results from Theorems \ref{prop:counting:moderate:sym} and \ref{prop:counting:moderate:lap} to show that it is unlikely for $G_{N}$ to have structured generalized normal vectors.

\begin{lemma}\label{lemma:l2} Given $\alpha$, for any sufficiently small $\delta>0$ and any sufficiently large $C$, there exists $C'$  depending on $\delta, \al$ and $C$ such that the following holds.
 Let $N,s_1,k,t,d$ be positive integers and $p$ a prime 
  such that $N^{12\delta} \le p\le \exp(N^{\delta/2})$, 
  and  $\frac{p}{t} \sqrt{\frac{s_1}{k}} \le \min\{e^{s_1/2k}, (s_1/Ck)^{(1-\delta)k}\}, $ and $1\le s_1\le d\le N,$ and $C \le k\le s_1 $.   
 We have the following %

\begin{itemize}
\item (symmetric and skew-symmetric) The event $\CE_{non-sparse}$ and the event that there exists $\Ba \in \BG_{d, k,s_1,s_2, \ge t} \bs \BG_{d, k,s_1,s_2,\ge 2t}$ such that $G_N \Ba = Y_0$ in at least $N-N^{\delta/2}$ coordinates have probability at most 
$$(C')^N  p^{N^{\delta/2}} (\frac{1}{p} + \frac{3C(t+1) \sqrt{k}}{p \sqrt{s_1}})^{N-s_2} \binom{N}{d}   p^{d+s_2} (\delta t)^{-d(1-s_1/s_2)}.$$ 

\item (Laplacian)  The event $\CE_{L, non-sparse}$ and the event that there exists $\Ba \in \BG_{L,d, k,s_1,s_2, \ge t} \bs \BG_{L,d, k,s_1,s_2,\ge 2t}$ such that $G_N \Ba = Y_0$ in at least $N-N^{\delta/2}$ coordinates have probability at most 
$$(C')^N  p^{N^{\delta/2}} (\frac{1}{p} + \frac{3C(t+1) \sqrt{k}}{p \sqrt{s_1}})^{N-s_2} \times  p \times\binom{N}{d}  2^{d+4}  p^{d+s_2} (\delta t)^{-d(1-s_1/s_2)}.$$
\end{itemize}
\end{lemma}
Again, we will apply this result for the choice of parameters from Eq.~\eqref{eqn:parameter} and with $t$ satisfying  $\frac{p}{t} \sqrt{\frac{s_1}{k}} \le \min\{e^{s_1/2k}, (s_1/Ck)^{(1-\delta)k}\}$.

\begin{proof}[Proof of Lemma \ref{lemma:l2}] We will only explain the details in the Laplacian case, the cases of symmetric and skew-symmetric matrices are similar. For a fixed $\Ba$ in $\BG_{L,d, k, s_1,s_2, \ge t}(N)\bs \BG_{L,d, k,s_1,s_2, \ge 2t}(N)$ we show that the event that $G_N \Ba = Y_0$ in at least $N-N^{\delta/2}$ coordinates has probability at most 
\begin{equation}\label{eqn:ab}
2^{N} p^{N^{\delta/2}} (\frac{1}{p} + \frac{3C(t+1) \sqrt{k}}{p \sqrt{s_1}})^{N-s_2}.
\end{equation}
Note that on $\CE_{L, non-sparse}$ of Lemma \ref{lemma:sparse:2}, $|\supp(\Ba-w\1)|$ has order $N$ for all $w$. Let $\Bb$ be the a subvector $\Ba_{\Lambda}$ such that $\rho_L(\Bb) \le  \frac{3C(t+1) \sqrt{k}}{p \sqrt{s_1}}$. Recall our treatment from \eqref{eqn:i.i.d.decomposition:nL} and \eqref{eqn:i.i.d.decomposition:L}
 that we can write  
$$G_N =\Big(\begin{array}{cc} G_{([N]\bs \Lambda) \times ([N]\bs \Lambda) }& G_{([N]\bs \Lambda) \times \Lambda} \\ G_{\Lambda \times ([N]\bs \Lambda)} & G_{\Lambda \times \Lambda} \end{array}\Big).$$ 
Assume that we have to work with an event of type $G_N \Ba = Y_0$ in all coordinates (after adding $p^{N^{\delta/2}}$ other possibilities for the missing equalities). We will decompose our matrix as above and then consider a system of $N-|\Lambda|$ equations
$$G_{([N]\bs \Lambda) \times ([N]\bs \Lambda) } \Ba_{[N]\bs \Lambda} +  G_{([N]\bs \Lambda) \times \Lambda} \Ba_\Lambda =(Y_0)_{[N]\bs \Lambda}.$$
If we condition on the entries of $G_{([N]\bs \Lambda) \times ([N]\bs \Lambda) }$, then we can view the above as $ G_{([N]\bs \Lambda) \times I} \Bb  = D$, where $D$ is deterministic (after expanding out the diagonals in $L_N$, $D$ will depend on the other entries not belonging to the above block), and we notice that now the entries of $ G_{([N]\bs \Lambda) \times \Lambda}$ are i.i.d.. So we can estimate the above probability by
\begin{equation}\P(G_{([N]\bs \Lambda) \times \Lambda} \Bb  = D) \le \prod_{i\in [N]\bs \Lambda}(\frac{1}{p}+\rho(\Bb-a_i \1)) \le (\frac{1}{p} +\frac{3C(t+1) \sqrt{k}}{p\sqrt{s_1}})^{N-s_2},
\end{equation}
Summing over the support choice of $\binom{N}{N^{\delta/2}}$ coordinates, and over other $p^{N^{\delta/2}}$ assignments of the unknown entries (toward $G_N \Ba = Y_0$) we complete the proof of \eqref{eqn:ab}.

To complete the proof of Lemma \ref{lemma:l2}, we use Theorem \ref{prop:counting:moderate:lap} and \eqref{eqn:ab} to estimate the probability under consideration by
$$(C')^N  p^{N^{\delta/2}} (\frac{1}{p} + \frac{3C(t+1) \sqrt{k}}{p \sqrt{s_1}})^{n-s_2} \times  p \times\binom{N}{d}  2^{d+4}  p^{d+s_2} (\delta t)^{-d(1-s_1/s_2)}.$$
\end{proof}

Now we prove Proposition \ref{prop:structure:subexp:lap} (Proposition \ref{prop:structure:subexp:sym} can be shown similarly). We will restate in the following form for the reader's convenience.

\begin{proposition}\label{prop:structure:subexp:lap'}
 Let $0<\la \le 1$ be a given constant. There exist 
  positive constants $c_1,c_2,c_3,c_4$ such that the following holds. 
Let $n,N$ be sufficiently large positive integers and let $L_n$ be as in Theorem \ref{theorem:cyclic:lap}, and let $L_N$ be the $[N]\times[N]$ submatrix of $L_n$.  
 Let $I_0\subset [N]$ be a (random) index set of size $\lfloor \la N \rfloor$ that might be depend on the randomness of $x_{kl}, $ for $N+1 \le k \text{ or } N+1 \le l$, but is independent of the randomness of $x_{kl},$ for $1\le k<l \le N$. 
 Let $p$ be a sufficiently large prime given $\lambda$  but $p \le \exp(N^{c_1})$.
Then for any non-zero fixed vector $Y_0 \in \F_p^N$, with probability at least $1 - \exp(-N^{c_2})$  
 for any vector $\Bv$ such that $L_N \Bv=Y_0$  in at least $N -N^{c_3}$ coordinates we have
$$\rho_L(\Bv_{I_0}) \le \exp(-N^{c_4}).$$ 
\end{proposition}
It is clear that Proposition \ref{prop:structure:subexp:lap} would follow by taking $c=\min\{c_1,c_2,c_3,c_4\}$.

\begin{proof}[Proof of Proposition \ref{prop:structure:subexp:lap'}] For the proof, we choose some small $\delta>0$.  In several steps, we will require $\delta$ sufficiently small (in an absolute sense) and we can work with any fixed $\delta$ that satisfies all the requirements.   Throughout the proof, all statements we make are assuming that $N$ is sufficiently large with respect to $\lambda$ and $\delta$ (and since we take $\delta$ an absolute constant in the end, this really only means $N$ is sufficiently large with respect to $\lambda$). It suffices to prove the result assuming 
\begin{equation}\label{eqn:bound:la}
\la \le \delta 
\end{equation}
for some $0<\delta<1$ because the statement for $\delta <\la \le 1$ will follow from the monotonicity of Fact \ref{fact:comparison:rho}. 

The choices for the parameters $c_i$ is somewhat flexible,  but for instance we can take
$$c_1=c_3=c_4=\delta/2, c_2=1-2\delta.$$ 

We let $\CV$ denote the vectors $\Bv$ in $\F_p^N$ for which $\Bv-a \1$ has at least $N-\la N/4$ non-zero coordinates for every $a$, %
and $\CW_{I_0}$ the set of non-zero vectors $\Bv \in \F_p^N$ that such that we have 
$$\rho_{L}(\Bv_{I_0}) > e^{- N^{c_4}}.$$  
Observe that %
\begin{align*}
& \P(\exists \Bv \in \CW_{I_0}, \mbox{$(L_N \Bv)_i = (Y_0)_i$ in at least $N-N^{c_3}$ coordinates })\\
 &= \P(\exists \Bv \in \CW_{I_0} \cap \CV,  \mbox{ $(L_N \Bv)_i = (Y_0)_i$  in at least $N-N^{c_3}$ coordinates }) \\
&+ \P(\exists \Bv \in \CW_{I_0} \cap \CV^c,  \mbox{ $(L_N \Bv)_i = (Y_0)_i$  in at least $N-N^{c_3}$ coordinates }).
\end{align*}
By Lemma \ref{lemma:sparse:2} (with $\lambda_N=\lambda/4$), 
\begin{align}\label{eqn:sparse:rho}
\P\Big(\exists \Bv \in \CV^c, \mbox{ $(G_N \Bv)_i = (Y_0)_i$  in at least $N-N^{c_3}$ coordinates}\Big) \leq \exp(-N^{1-\delta}).
\end{align}

Therefore, it suffices to focus on the event $\CF$  that there exists $\Bv \in \CV \cap \CW_{I_0},$ such that $G_N \Bv = Y_0$ in at least $N-N^{c_3}$ coordinates. We first bound the event $\tilde \CF$ that there exists $\Bv \in \CV \cap \CW_{I_0},$ such that $(G_N \Bv)_i = (Y_0)_i$ for all $i$.

If $p$ is sufficiently large  but $p \le N^{12\delta}$, then as $\Bv\in \CV$, we have that $\Bv_{I_0}$ has highest multiplicity less than $|I_0|/4$ (because in $\CV$, for any $a$ the vector $\Bv_{I_0}-a \1$ has at least $|I_0| -\la N/4$ non-zero coordinates). 
Hence we can apply Lemma \ref{lem:deg:1}, provided that $\delta$ is small,  
to obtain that 
$$\rho_L(\Bv_{I_0}) \le \exp(-N^{1/16}).$$ 

We now assume $p>N^{12\delta}$. In what follows set 
$$\tau:=N^{1/16} \mbox{ and $k :=\lfloor N^{1/4} \rfloor, s_1:=\lfloor N^{1-4\delta} \rfloor , s_2:=\lfloor N^{1-2\delta}/2 \rfloor$}.$$

We let $d\le N$ be fixed with $d\geq s_1$, and let $C$ be a constant which will have to be sufficiently large to an extent that will be specified later.

Note that any non-zero $\Bv$ resides in $\BG^c_{L,d, k,s_1,s_2, >p}$ (because as $R^\delta_k(\Bw)\leq R_k(\Bw) \le 2^{2k} |\Bb|^{2k}$, for $t>p$ we have that $\BG'_{L,d, k,s_1,s_2,\geq t}$ is empty, and hence
$\BG_{L,d, k,s_1,s_2,\geq t}$ is empty by \eqref{eqn:GG':L}). \HC{added this eqn.}

So with $j_0$ being the smallest integer such that $2^{j_0} \tau> p$, we have an increasing nested sequence $(\BG^c_{L, d, k,s_1,s_2, 2^j \tau })_{j=0}^{j_0}$ 
and we can partition $\CV $ into union of disjoint sets 
$$
\CV =(\CV \cap \BG^c_{L, d, k,s_1,s_2, \ge \tau }) \cup \bigcup_{j=1}^{j_0} (\CV \cap \BG^c_{L, d, k,s_1,s_2, \ge 2^j \tau} \cap \BG_{L, d, k,s_1,s_2, \ge 2^{j-1} \tau}). 
$$

Let $X_i$ be the $i$th row of $L_N$. We then have that $\P(\tilde \CF)$ is 
		\begin{align*}
		& \P\Big(\exists \Bv \in \CW_{I_0} \cap \CV,  \Bv \cdot X_i  = (Y_0)_i, 1\le i\le N\Big) = \P\Big(\exists \Bv \in \CW_{I_0}  \cap \CV \cap \BG^c_{L,d,k,s_1,s_2, \tau}, \Bv \cdot X_i = (Y_0)_i, 1\le i\le N \Big) \\
		& + \sum_{j=1}^{j_0} \P\Big(\exists \Bv \in \CW_{I_0} \cap (\CV \cap \BG^c_{L,d,k,s_1,s_2, 2^j \tau} \cap  \BG_{L,d,k,s_1,s_2, 2^{j-1} \tau}),  \Bv \cdot X_i = (Y_0)_i, 1\le i\le n\Big). 
		\end{align*} 
In what follows, without loss of generality we assume that $\la N$ is an integer (otherwise one just replaces $\la N$ by $\lfloor \la N \rfloor$ in all estimates.)

{\bf Case 1.} We begin with vectors in $\CV \cap \CW_{I_0} \cap \BG^c_{L, d, k,s_1,s_2,\tau}$. Let $\tilde \CF_1$ be the subevent of $\tilde \CF$ that there exists a vector $\Bv$ in $\CV \cap \CW_{I_0} \cap \BG^c_{L, d, k,s_1,s_2,\tau}$ such that $L_N \Bv = Y_0$. 
Note that as $\Bv\in \BG^c_{L, d, k,s_1,s_2,\tau}$, we cannot hope for structures of $\Bv$. %
However, because of the lower bound $e^{-N^{c_4}}\le \rho_L(\Bv_{I_0})$, by Lemma \ref{lem:deg:2} there exists a generalized arithmetic progression $P$ of rank one in $\F_p$ and of size $O(p/N^{c_4})$ that contains at least all but $N^{4c_4}$ entries of $\Bv_{I_0}$. 
Note that the number of ways to choose such a $P$ is bounded by $p^{O(1)}$.  For a fixed $P$, the number of vectors $\Bv_{I_0}$ with at least $\la N - N^{4c_4}$ 
components in $P$ is at most
\begin{equation}\label{eqn:a2}
\binom{\la N}{\lfloor N^{4c_4}\rfloor} |P|^{\la N - N^{4c_4}} p^{N^{4c_4}} \leq \binom{\la N}{\lfloor N^{4c_4}\rfloor} O(p/N^{c_4})^{\la N- N^{4c_4}} p^{N^{4c_4}} \leq O(1)^{\la N} (p/N^{c_4})^{ \la N},
\end{equation}
where we used the fact that $p\le \exp(N^{c_1})$ with small $c_1,c_4$. 

Together with $p^{O(1)}$ (which is trivially bounded by $O(1)^N$ as $p\le \exp(N^{c_1})$) number of ways to choose $P$, we thus have the following.
\begin{claim}\label{claim:v_0:subexp} 
The total number of $\Bv_{I_0}$ (where $|I_0|=\la N$) with  $e^{-N^{c_4}}\le \rho_L(\Bv_{I_0})$ is at most $O(1)^{\la N} (p/N^{c_4})^{ \la N}$.
\end{claim}

Our next step is similar to the proof of Lemma \ref{lemma:sparse:2} (where there we started with a subvector with high entry multiplicities) that we will try to saturate $\Bv$ until there is almost no (GAP-type) structure left. 

We will conditioned on the $(x_{kl})_{N+1 \le k \text{ or } N+1 \le l}$ and treat $I_0$ as fixed. Without loss of generality we assume that $I_0$ is the set $[\la N]$ of the first $\la N$ indices. %

Let $k_0 := \lfloor N^{1-\delta} \rfloor$. 
We then divide the index set $[\la N+1,\dots, N]$ into subsequences $J_1,\dots, J_\ell$ of consecutive numbers so that $|J_i| = k_0$ for $i < \ell$ and $k_0 \le |J_\ell| < 2k_0$.

Given $\Bv$, we partition $[N]$ into two subsets, $I_m$ and $I_s$. We let $I_s$ be the union of the $J_i$ such that  $\rho_L(\Bv_{J_i})\ge e^{-N^{c_4}}$, along with $[\la N]$. We let $I_m = [N] \bs I_s$. We write $\Bv_m := \Bv_{I_m}$ and $\Bv_s := \Bv_{I_s}$.

Notice that as the index set of the components of $\Bv_s$ has the form $I_0 \cup_l J_{i_l}$, there are at most $2^\ell \le 2^{N^\delta}$ such subsets (where we used the fact that $N$ sufficiently large given $\delta, \lambda$).

{\it Subcase 1.1.} Let $\CF_1$ be the event that there exists a $\Bv$ as in the definition of $\tilde{\CF}_1$ such that $\Bv_m$ is empty. 
Then by repeatedly applying Claim \ref{claim:v_0:subexp} we obtain that the number of such vectors $\Bv=\Bv_s$ is bounded by
\begin{equation}\label{eqn:sc1}
O(1)^{\la N} (p/N^{c_4})^{ \la N}\left( O(1)^{\lfloor N^{1-\delta} \rfloor} (p/N^{c_4})^{\lfloor N^{1-\delta} \rfloor}\right)^{N^\delta -\la N^\delta} O(1)^{2 \lfloor N^{1-\delta} \rfloor }(p/N^{c_4})^{2\lfloor N^{1-\delta} \rfloor} = O(p^{2N} N^{-c_4 N/8}),
\end{equation}
provided that $N$ is sufficiently large, where the first factor comes from the number of $\Bv_0$, the second factor comes from the number of $\Bv_{J_i}$ for all $i<\ell$, and the third factor comes from the number of $\Bv_{J_\ell}$.

Then
the probability that any 	$\Bv \in \CV \cap \CW_{I_0} \cap \BG^c_{L, d,k,s_1,s_2, \tau}$ of the above type and satisfies $\Bv \cdot X_i  =(Y_0)_i, 1\le i\le N$ is bounded by (via \eqref{eqn:ab}, where we note that in the proof there we exploited the randomness of $(x_{kl})_{k,l\in [N], k\neq l}$ only)
$$
\P(\CF_1) \le p^{2N} N^{-c_4 N/8} \Big(\frac{1}{p} + \frac{3C(N^{1/16}+1) \sqrt{N^{1/4}}}{p \sqrt{N^{1/2}}}\Big)^{N-s_2}   \le N^{-c_4N/9}, 
$$
where we used that $N^{12\delta} < p\le \exp(N^{\delta/2}), s_2=\lfloor N^{1-2\delta} /2\rfloor$ and $\tau=N^{1/16}$ and $N$ is sufficiently large and $\delta$ is sufficiently small.

{\it Subcase 1.2.} Let $\CF_2$ be the event that there exists a $\Bv$ as in the definition of $\tilde{\CF}_1$ such that $\Bv_m$ is not empty. 
We have that $k = \lfloor N^{1-\delta} \rfloor \le |I_m|$ and $|I_s| \ge \la N$. 

Similarly to the proof of  Lemma \ref{lemma:sparse:2}, we now describe a function $f$ from subsets of $[N]$ (that can occur as $I_s$) to subsets of $[N]$.  We describe $f(I_s)$, 
and write $I_m$ for $[N]\setminus I_s$, but note that $f$ does not depend
on $\Bv$.  If $|I_s|>|I_m|$, we let $f(I_s)$ be the first $|I_m|$ elements of $I_s$.
Otherwise, we choose $J_{i_1},\dots, J_{i_l}$ arbitrarily from $I_m$ 
so that $|I_s| -N^{1-\delta} \le |J_{i_1}|+\dots+ |J_{i_l}|<|I_s|$ 
and let $f(I_s):= [N]\bs (J_{i_1}\cup \dots \cup J_{i_l})$. (Hence, to relate to our proof of  Lemma \ref{lemma:sparse:2}, the set $J^\ast$ there plays the role of $J_{i_1}\cup \dots \cup J_{i_l}$, and $f(I_s)$ is the union of $I_s$ and the complement of $J^\ast$ in $I_m$.)

In either case, we have 
$$|f(I_s)| \ge |I_m| \mbox{ and } |[N]\bs f(I_s)|\ge |I_s| - N^{1-\delta}.$$

Recall from the proof of  Lemma \ref{lemma:sparse:2} that $\CE_{drop}$ is the event that  there is a square submatrix $G_{A\times B}$ of $G_N$
of dimension $\geq \lfloor N^{1/2+\delta} \rfloor$
 with rank less than $|A| - \lfloor N^{1/2+\delta} \rfloor$ and we have $\P(\CE_{drop})\leq  e^{-c_{drop} N^{1+2\delta} },
$
for some $c_{drop}>0$ depending on $\alpha$.

We wish to bound the probability of $\mathcal{F}_2\setminus \CE_{drop}$.
Since for $\Bv$ causing $\mathcal{F}_2$, we have  $|I_m|\ge N^{1-\delta}$, 
and hence $|I_m|\geq  \lfloor N^{1/2+\delta} \rfloor$ (by assuming $\delta<1/4$).  
So outside of $\CE_{drop}$,  the matrix $G_{f(I_s) \times I_m}$ has rank $\geq |I_m|-  \lfloor N^{1/2+\delta} \rfloor$, 
and hence it has a square submatrix $G_{{{{{I'_s}}}} \times {{I'_m}}}$ of dimension $|I_m| -  \lfloor N^{1/2+\delta} \rfloor$ which has full rank.

Given subsets $M$, ${{I'_m}}$, and ${{{{I'_s}}}}$ of $[N]$
such that $|{{I'_m}}|=|{{{{I'_s}}}}|$ and ${{I'_m}}\sub M$ and ${{{{I'_s}}}}\sub f([N]\setminus M)$, and such that
$M$ is non-empty and can occur as $I_m$,  let $\mathcal{F}_{2.1}(M,{{I'_m}},{{{{I'_s}}}})$ be the event that there is a $\Bv\in \F_p^N$ such that $\rho_{L}(\Bv_{I_0}) \geq e^{- N^{c_4}}$, %
 $G_N\Bv=Y_0$, and $I_m=M$,  and $G_{{{{{I'_s}}}}\times {{I'_m}}}$ is full rank. 
In particular, above we just saw that $\mathcal{F}_2\setminus \CE_{drop}$
implies 
$\mathcal{F}_{2.1}(M,{{I'_m}},{{{{I'_s}}}})$ for some $M$, ${{I'_m}}$, and ${{{{I'_s}}}}$ with $|M|-|{{I'_m}}|=  \lfloor N^{1/2+\delta} \rfloor$
and $N-|M|\geq \lambda N$.

\begin{claim}
For all $M,{{I'_m}},{{{{I'_s}}}}$ such that $\mathcal{F}_{2.1}(M,{{I'_m}},{{{{I'_s}}}})$ is defined, we have
$$
\P(\mathcal{F}_{2.1}(M,{{I'_m}},{{{{I'_s}}}}))  \le N^{- c_4 \la N/9}.
$$
\end{claim}
\begin{proof} By using Claim \ref{claim:v_0:subexp} over the $J_i$, as $|I_s| = N-|M|$, for $ \Bv$ such that $\rho_L(\Bv_{[\lambda N]})\geq e^{-N^{c_4}}$ and $I_m=M$ there are at most 
$$O(1)^{\la N} (p/N^{c_4})^{ \la N}\left( O(1)^{\lfloor N^{1-\delta} \rfloor} (p/N^{c_4})^{\lfloor N^{1-\delta} \rfloor}\right)^{(N-M-\la N)/ N^{1-\delta}} = O(1)^{N-M} p^{N-|M|} N^{-c_4 (N-|M|)/8}$$ 

choices of $\Bv_{I_s}$, and $\binom{N}{N^{1/2+\delta}}p^{N^{1/2+\delta}}$ choices for $(\Bv)_{I_m\bs I_m'}$. Let $F= f([N]\bs M)$.  We condition on $L_{F \times[N]}$ and compute the conditional probability of $\mathcal{F}_{2.1}(M,{{I'_m}},{{{{I'_s}}}})$.  
If $G_{{{{{I'_s}}}}\times {{I'_m}}}$ is not full rank, then the desired (conditional) probability is $0$.   Otherwise, as in the argument of 
Claim \ref{claim:cond:fix}, we have $\Bv_{I'_m}$ is determined  (and hence the entire $\Bv$ is determined).

We let $J_*$ be one of the $J_i$ that is in $M$ but has no intersection with $F$.
For each $i\in ([N]\setminus F) \setminus J_*$, we let $X_i$ be the $i$th row of $G_N$.  The equation $X_i\Bv=(Y_0)_i$ implies
\begin{equation}\label{E:toget2}
\sum_{j\in J_*} x_{ij}(v_j-v_i)=(Y_0)_i -\sum_{j\in [N]\setminus (J_* \cup \{ i\})} x_{ij}(v_j-v_i)+\sum_{j\in [n+1]\setminus [N]} x_{ij} v_i.
\end{equation}

The $x_{ij}$ for $i\in  ([N]\setminus F) \setminus J_*$ and $j\in J_*$ are all independent.  
We further condition on $x_{ij}$ for $i\in  ([N]\setminus F) \setminus J_*$ and $j\not \in J_*$.
Since $J_*\cap F=\emptyset$, none of the $x_{ij}$ for $i\in  ([N]\setminus F) \setminus J_*$ and $j\in J_*$ have been conditioned on.
Thus after our conditioning, the probability of \eqref{E:toget2} holding is at most $\rho_L(\Bv_{J_*}).$
We have $N-|F|-|J_*|$ independent such equations that are implied by $G_N\Bv=Y_0$.
Thus the probability that, for a given $\Bv_s$ and $\Bv_{I_m\setminus {{I'_m}}},$ we have the event in the claim, is  at most
$$
\rho_L(\Bv_{J_*})^{N-|F|-|J_*|}.
$$
Since $J_*$ is one of the $J_i$ that is a subset of $I_m$, by definition of $I_m$ we have that
$$
\rho_L(\Bv_{J_*})\leq e^{-N^{c_4}}+1/p.
$$
We have $|J_*| \le 2 \lfloor N^{1-\delta} \rfloor$ and $|M| \le |F| \le |M| +N^{1-\delta}$. Thus the probability that, for a given $\Bv_s$ and $\Bv_{I_m\setminus {{I'_m}}},$ we have the event in the claim, is at most
$$
\left( e^{-N^{c_4}}+1/p \right)^{N-|M|- 3 \lfloor N^{1-\delta} \rfloor}.
$$
Summing over the possible values for $\Bv_s$ and $\Bv_{I_m\setminus {{I'_m}}}$ (which are bounded in number above), we obtain 
 we thus obtain a probability bound
$$O(1)^{N-M} p^{N-|M|} N^{-c_4 (N-|M|)/8} \binom{N}{N^{1/2+\delta}}p^{N^{1/2+\delta}}   (e^{-N^{c_4}}+1/p)^{N-|M| - 3 \lfloor N^{1-\delta} \rfloor}  \le N^{-c_4 \la N/9},$$
where we used the fact that $\delta$ is sufficiently small (and $c_1=c_4=\delta/2$) and $p\le e^{N^{\delta/2}}$ and $|M| \le N-  \lambda N$.
\end{proof}

The event $\CF_2$ is the union of $\CE_{drop}$ with $\mathcal{F}_{2.1}(M,{{I'_m}},{{{{I'_s}}}})$ over $M,{{I'_m}},{{{{I'_s}}}}$ with $|M|-|{{I'_m}}|= \lfloor N^{1/2+\delta} \rfloor$
and $|N|-|M|\geq \lambda_N N$ and $M$ a possible value of $I_m$.  There are at most $2^\ell \leq 2^{2N^{\delta}}$ possible values of $I_m$ (and hence $M$).
Given $M$, there are at most $\binom{|M|}{\lfloor N^{1/2+\delta} \rfloor}$ choices of ${{I'_m}}$, and $\binom{|f([N]\setminus M)|}{\lfloor N^{1/2+\delta} \rfloor} \le \binom{|M|+N^{1-\delta}}{\lfloor N^{1/2+\delta} \rfloor}$ choices of ${{{{I'_s}}}}$.  Hence,
we have
\begin{align*}
\P(\CF_2)&\leq   e^{-c_{drop} N^{1+2\delta} } + 2^{2N^{\delta}}  N^{2N^{1/2+\delta}} N^{-c_4 \la N/9}.
\end{align*}
Thus, for $\delta$ sufficiently small and $N$ sufficiently large given $\delta$ and $\lambda$,
$$
\P(\tilde \CF_1)\leq \P(\CF_1)+ \P(\CF_2) \leq N^{-c_4 \la N/10}.
$$

For the next case, we use an argument similar to that in \cite{FJLS}, but with a modification for the Laplacian case from
Theorem \ref{prop:counting:moderate:lap}.

		{\bf Case 2.} We next consider the subevent $\tilde \CF_2$ of $\tilde \CF$ that there exists a vector $\Bv \in  \CW_{I_0} \cap \cup_{j=1}^{j_0}(\CV \cap \BG_{L, d,k,s_1,s_2, 2^j \tau}^c  \wedge \BG_{L, d,k,s_1,s_2, 2^{j-1} \tau})$
	such that $L_N \Bv = Y_0$.		By Lemma \ref{lemma:l2} (where we can check that all of the conditions are met
we have for some sufficiently large $C'>0$ depending on $\delta$
		\begin{align*}
		\P(\tilde \CF_2)\leq &\sum_{j=1}^{j_0} \P\Big(\exists \Bv \in \CW_{I_0} \cap (\CV \cap \BG^c_{L,d,k,s_1,s_2, 2^j \tau}\cap \BG_{L,d,k,s_1,s_2, 2^{j-1} \tau}),  \Bv \cdot X_i  = (Y_0)_i, 1\le i\le n\Big) \\
		&\leq\sum_{j=1}^{j_0} \P\Big(\exists \Bv \in  (\CV \cap \BG^c_{L,d,k,s_1,s_2, 2^j \tau} \cap \BG_{L,d,k,s_1,s_2, 2^{j-1} \tau}),  \Bv \cdot X_i  = (Y_0)_i, 1\le i\le n\Big) \\
		&\leq  \binom{N}{d} \sum_{j=1}^{j_0} (C')^N  p^{N^{\delta/2}} \left(\frac{1}{p} + \frac{3C(2^{j-1} \tau+1) \sqrt{k}}{p \sqrt{s_1}}\right)^{N-s_2}   2^{d+4}   p^{d+s_2+1} (2^{j-1} \tau)^{-d(1-s_1/s_2)} \\
		&\leq \binom{N}{d} 2^{d+4} \sum_{j=1}^{j_0}(C')^N  p^{N^{\delta/2}} 3^{N-s_2-1}\left[ (\frac{1}{p})^{N-s_2} + ( \frac{3C \sqrt{k}}{p \sqrt{s_1}})^{N-s_2}   + ( \frac{3C 2^{j-1} \tau \sqrt{k}}{p \sqrt{s_1}})^{N-s_2}\right]    p^{d+s_2+1} (2^{j-1} \tau)^{-d(1-s_1/s_2)}\\
		& \leq  \sum_{j=1}^{j_0} (12C')^N     p^{N^{\delta/2} +d +s_2 +1} (\frac{1}{p})^{N-s_2}  (2^{j-1} \tau)^{-d(1-s_1/s_2)} \\
		& + \sum_{j=1}^{j_0} (36 C C')^N     p^{N^{\delta/2} +d +s_2 +1}  ( \frac{\sqrt{k}}{p \sqrt{s_1}})^{N-s_2}   (2^{j-1} \tau)^{-d(1-s_1/s_2)} \\
		& + \sum_{j=1}^{j_0} (36 CC')^N   p^{N^{\delta/2} +d +s_2 +1}    ( \frac{2^{j-1} \tau \sqrt{k}}{p \sqrt{s_1}})^{N-s_2}(2^{j-1} \tau)^{-d(1-s_1/s_2)}.
              \end{align*} 
              Recall that $\tau= N^{1/16}, k =\lfloor N^{1/4} \rfloor, s_1=\lfloor N^{1-4\delta}\rfloor, s_2=\lfloor N^{1-2\delta}/2 \rfloor$. It is clear that the first two sums are bounded by $N^{-10 \delta N}$ as long as $\delta$ is sufficiently small and $p \ge N^{12 \delta}$ and $N$ are sufficiently large given $\delta$. 

 For the third sum, for $j=1$ we have 
              $$(36CC')^N ( \frac{\tau \sqrt{k}}{p \sqrt{s_1}})^{N-s_2}   \tau^{-d(1-s_1/s_2)} \le  (1/p)^{N/2}.$$

The sum of other terms can be rewritten as 
\begin{align*}
&\sum_{j=2}^{j_0} (36 CC')^N   p^{N^{\delta/2} +d +s_2 +1}   ( \frac{2^{j-1} \tau \sqrt{k}}{p \sqrt{s_1}})^{N-s_2}(2^{j-1} \tau)^{-d(1-s_1/s_2)}  \\
&=  (36 CC')^N   p^{N^{\delta/2} +d +s_2 +1}  \sum_{j=2}^{j_0}( \frac{\sqrt{k}}{p \sqrt{s_1}})^{N-s_2} \tau^{N-s_2-d(1-s_1/s_2)}  2^{(j-1)(N-s_2-d(1-s_1/s_2))}.
\end{align*}
Note that $N-s_2-d(1-s_1/s_2) \ge N-s_2-N(1-s_1/s_2)=Ns_1/s_2 -s_2>0$ as $N$ is sufficiently large, the above sum is clearly bounded by $(j_0-1)$ times the last summand when $j=j_0$,
and that last term is bounded by (where we recall that $p< 2^{j_0} \tau \le 2p$)
\begin{align*}
 &(36 CC')^N   p^{N^{\delta/2} +d +s_2 +1}    ( \frac{2^{j_0-1} \tau \sqrt{k}}{p \sqrt{s_1}})^{N-s_2}(2^{j_0-1} \tau)^{-d(1-s_1/s_2)} \\
 & \le  (36 CC')^N 2^d  p^{N^{\delta/2} +d +s_2 +1}  (\frac{\sqrt{k}}{\sqrt{s_1}})^{N-s_2} p^{-d(1-s_1/s_2)} \\
 & \le (72 CC')^N p^{N^{\delta/2} + s_2+ d s_1/s_2}  (\frac{\sqrt{k}}{\sqrt{s_1}})^{N-s_2} \le N^{-N/4}
 \end{align*}
because $d\le N$ and $p\le \exp(N^{c_1})$ where  $c_1=\delta/2$ and we take $\delta$ sufficiently small.

Combining the above estimates, by summing over all $d \le N$, we can bound $\P(\tilde \CF_2)$ of Case 2 by 
$$\P(\tilde \CF_2) \leq N^{-10 \delta N} + N^{-N/4}.$$
Putting the two cases together we conclude that
		$$
		\P(\tilde \CF) = \P\Big(\exists \Bv \in \CV\cap \CW_{I_0}, \mbox{$(G_N \Bv)_i = (Y_0)_i$ for all $i$}\Big) \le  N^{-c_4 \la N/10}  + N^{-10\delta N}+ N^{-N/4} \le 3 N^{-c_4 \la N/10}
		$$ 
as $c_4= \delta/2$ and $\delta$ is sufficiently small.		

		To bound the probability of $\CF$, we need to account also for the analog of $\tilde \CF$ where $Y_0$ is replaced by any vector that shares at least $N-N^{c_3}$ 
		coordinates with it.  
This gives
$$
\binom{N}{N^{c_3}} p^{N^{c_3}}
$$
total events with the same probability bound as we showed above for $\tilde \CF$, whose union is $\CF$.
 We then have
\begin{align}
 \P(\CF) \notag
\leq &\left (\binom{N}{N^{c_3}} p^{N^{c_3}}
 \right) 3  N^{-c_4 \la N/10}  \le N^{-c_4 \la N/11},
\end{align}

provided that $\delta$ is chosen sufficiently small. This, together with \eqref{eqn:sparse:rho},
 complete the proof of Proposition \ref{prop:structure:subexp:lap'}.
					\end{proof}

\section{Treatment for moderate primes: proof of Propositions \ref{prop:moderate:lap} and \ref{prop:moderate:alt}}\label{section:rankevolving}

We will be focusing mostly on the Laplacian case, the skew-symmetric and symmetric cases will be discussed later (by a simpler argument). The plan for the Laplacian matrix works as follows.

\begin{enumerate}
\item We first sample $L_n$ using Phase 1 from Definition \ref{def:lap}. For each $N_0= \lfloor cn \rfloor \le N \le n$ (where $c$ is a sufficiently small positive constant), we record the index set $I_{N+1}\subset [N]$ of $1\le j\le N$ where $v_j$ has exactly one neighbor in $\{v_{N+1}, v_{N+2}\}$. Recall \eqref{eqn:concentration}, by Chernoff's bound with probability at least $1-\exp(-\Theta(n))$ (with respect to $L_n$) we have for all $N_0 \le N \le n$
\begin{equation}\label{eqn:I_N+1}
(1/2-c)N \le  |I_{N+1}| \le (1/2+c)N.
\end{equation}
\vskip .1in
\item We notice that the random set $I_{N+1}$ is independent of the entries $x_{ij}, 1\le i <j\le N$ of $L_N$. Hence for each $N_0 \le N \le n$ we can apply Proposition \ref{prop:structure:subexp:lap} to $L_N$ with $I_0=I_N$. To start with, when $N=N_0$ let us call the event under consideration $\CE_{N_0}^\ast$. This event belongs to the $\sigma$-algebra generated by the random entries of $L_n$, and 
\begin{equation}\label{eqn:N_0}
\P(\CE_{N_0}^\ast ) \ge 1 -O(\exp(-N^c)) = 1 -O(\exp(-\Theta(n^c))).
\end{equation}
\vskip .1in
\item We next reshuffle the neighbors of $v_{N_0+1}$ and $v_{N_0+2}$ to create new randomness in the column corresponding to $v_{N_0+1}$. After adding this column to $L_{N_0}$ we obtain $L_{N_0+1}^\ast$, where we use the asterisk to distinguish with $L_{N_0+1}$ of $L_n$. 
\vskip .05in
\begin{itemize} 
\item On the one hand, conditioning on $\CE_{N_0}^\ast$, we can use the randomness of the column corresponding to $v_{N_0+1}$ to study the rank relation of $L_{N_0+1}^\ast$ with that of $L_{N_0}$. 
\vskip .05in
\item By Fact \ref{fact:EoM}, $L_{N_0+1}^\ast$ has the same distribution as $L_{N_0+1}$ (of $L_n$), hence we can apply Proposition \ref{prop:structure:subexp:lap} to $L_{N_0+1}^\ast$ with respect to $I_0=I_{N_0+1}$. Let $\CE_{N_0+1}^\ast$ denote the overwhelming event obtained by this result (with probability similarly to \eqref{eqn:N_0}), which now belongs to the sigma-algebra generated by $L_n$ and the new random variables of the column corresponding to $v_{N_0+1}$. \end{itemize}
We will repeat the process until $v_n$.
\end{enumerate}

\begin{figure}
	\centering
	\begin{tikzpicture}

		\draw (1,1) -- (4,1) -- (4,4) -- (1,4) -- (1,1);
		\node at (2.5,2.5) {$L_N^\ast$};
		\draw (1,0.5)  -- (4,0.5);
		\node at (4.5,0.5)[circle,fill,inner sep=1pt]{};
		\draw (3.5,-0.2) node[right=0.1in]{$x_{N+1}$}; 
		\draw (4.5,1) --(4.5,4);
                  \draw [line width=2pt] (4.5,2)--(4.5,3);
                  \node at (5,2.5) {$I_{N+1}$}  ;

		\draw (5,4) node[above=0.1in]{$X_{N+1}$}; 
		\draw [dashed] (1,-0.5)--(5.5,-0.5);
                \draw [dashed] (5.5,-0.5)--(5.5,4);

               \draw [dashed] (1,-1)--(6,-1);
               \draw [dashed] (6,-1)--(6,4);

	\end{tikzpicture}
	\caption{Row and column exposure process during reshuffling.}
	\label{CornerEP}
\end{figure}
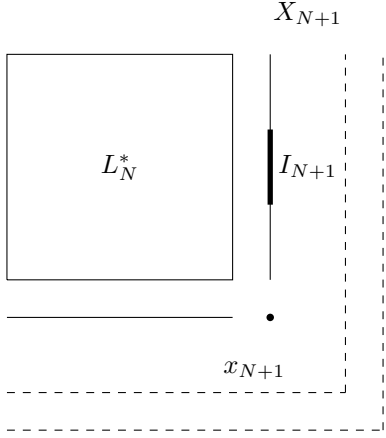

With the outline in mind, we will study the ranking evolution when moving from $L_N^\ast$ to $L_{N+1}^\ast$.

\begin{definition}\label{def:X} Given $I_0\subset [N]$. A random vector $X=(x_1,\dots, x_n)$ is called $I_0$-adapted if the $x_i$ are i.i.d. copies of $\xi$ satisfying Eq.~\eqref{eqn:alpha} (for a given prime $p$) for all $ i\in I_0$, while $x_i, i\notin I_0$ are deterministic (i.e. held fixed). 
\end{definition}
We first prove an elementary decoupling result to motivate our later treatment (see also \cite{CTV,Msym,KNg}). 
\begin{lemma}\label{lemma:decoupling1} Assume that $a_{ij} \in \Z/p\Z$ and $a_{ij}=a_{ji}$ and $b_i \in \Z/p\Z$. Assume $X$ is as in \ref{def:X} for a given $I_0$. Then for any $I \subset [n]$ we have
$$\sup_{r }|\P(\sum_{ij} a_{ij} x_i x_j + \sum_i b_i x_i = r) - 1/p|^4 \le |\P(\sum_{i\in I, j \in I^c} a_{ij} y_i y_j=0) -1/p|,$$
where $y_i, i\in I_0$ are i.i.d. copies of $\xi-\xi'$.
\end{lemma}
\begin{proof}[Proof of Lemma \ref{lemma:decoupling1}] For short we write $f(X)=\sum_{ij} a_{ij} x_i x_j + \sum_i b_i x_i$. We write
$$|\P(f(X)= r) -\frac{1}{p}| \le \frac{1}{p}  \sum_{t\neq 0} | \E e_p(-tf(X))|$$
where $e_p(x) = e^{2\pi i x/p}$.
 
We then use Cauchy-Schwarz to complete squares, 
\begin{align*}
LHS^2 \le  \frac{p-1}{p^2} \sum_{t\neq 0} |\E  e_p(-tf(X))|^2 &\le \frac{p-1}{p^2} \sum_{t\neq 0} \E_{X_I} |\E_{X_{I^c}}  e_p(-tf(X_I,X_{I^c})) |^2 \\
&= \frac{p-1}{p^2} \sum_{t\neq 0} \E_{X_I} \E_{X_{I^c}, X'_{I^c}}  e_p(-t[f(X_I,X_{I^c})-f(X_I,X_{I^c}')])\\
&= \frac{p-1}{p^2} \sum_{t\neq 0}  \E_{X_{I^c}, X'_{I^c}} \E_{X_I} e_p(-t[f(X_I,X_{I^c})-f(X_I,X_{I^c}')]).
\end{align*}
Using Cauchy-Schwarz once more,
\begin{align*}
LHS^4 & \le (\frac{p-1}{p^2})^2 (p-1)\sum_{t\neq 0} \E_{X_{I^c}, X'_{I^c}} \E_{X_I, X_I'} e_p(-t[f(X_I,X_{I^c})-f(X_I,X_{I^c}') -f(X_I',X_{I^c})+f(X_I',X_{I^c}')  ]) \\
&= (\frac{p-1}{p^2})^2 p(p-1) \frac{1}{p} \sum_{t \neq 0} \E_{Y_I,Y_{I^c}} e_p(-t \sum_{i \in I, j\in I^c} a_{ij} y_i y_j)= (\frac{p-1}{p})^3 (\P(\sum_{i\in I, j \in I^c} a_{ij} y_i y_j =0) -1/p).
\end{align*}

\end{proof}
 By Lemma \ref{lemma:quadratic}, with probability at least $1-\exp(-\Theta(n))$ we can assume that $L_N^\ast$ has rank at least $N -cN/2$ for each $N+0 \le N \le n$. 
 Let us consider the event that $L_N^\ast$ has rank exactly $N-k$ (where $k\le c N/2 $).
\begin{claim}\label{claim:principleminor} Assume that $L_N^\ast $ has rank $N-k$, then there is a set $I \subset [N], |I| =N-k$ such that the principle minor matrix $L_{I \times I}^\ast$ has full rank $N-k$.
\end{claim}
\begin{proof}[Proof of Claim \ref{claim:principleminor}] Assume without loss of generality that $\row_1(L_N^\ast), \dots, \row_{n-k}(L_N^\ast)$ span the row vectors of $L_N^\ast$, then in particularly $\row_i|_{[N-k]}, i\ge N-k+1$ belong to the span of  $\row_i|_{[N-k]}, 1\le i\le n-k$. This implies that the matrix spanned by the first $N-k$ columns $X_1,\dots, X_{N-k}$ has rank at most $N-k$. On the other hand, as the matrix is symmetric, this column matrix has the same rank as that of the matrix generated by $\row_1(L_N^\ast), \dots, \row_{N-k}(L_N^\ast)$, which is $N-k$. So the matrix $L_{[N-k] \times [N-k]}$ generated by $\row_i|_{[N-k]}, 1\le i\le N-k$ has rank $N-k$. 
\end{proof}
In what follows, assume again without loss of generality that $\row_1(L_N^\ast), \dots, \row_{N-k}(L_N^\ast)$ span the row vectors of $L_N^\ast$. When we add a new column $X_{N+1} = (x_{1(N+1)},\dots, x_{(N+1) (N+1)}):=(x_1,\dots, x_{N+1})$ and its transpose to create $L_{N+1}^\ast$, if $\rk(L_{N+1}^\ast)< \rk(L_N^\ast)+2$ then the extended row vector $\row_1(L_{N+1}^\ast),\dots, \row_{N-k}(L_{N+1}^\ast)$ still generate the space of the vectors $\row_1(L_{N+1}^\ast),\dots, \row_N(L_{N+1}^\ast)$. In particularly, this implies that 
\begin{equation}\label{eqn:x,a}
x_{i} = \sum_{j=1}^{N-k}  a_{ij} x_{j}, \text{for $N-k+1\le i\le N$,}
\end{equation}
where $a_{ij}$ are determined from $L_{N}^\ast$ via $\row_i(L_{N}^\ast) = \sum_{j=1}^{N-k} a_{ij} \row_j(L_{N}^\ast)$. In other words, Eq.~\eqref{eqn:x,a} says that the vector $(x_{1}, \dots, x_{N})$ is orthogonal to the vectors $(a_{i1},\dots, a_{i(N-k)}, 0,\dots,-1,0,\dots,0)$, or equivalently it belongs to the hyperplane $H_N$ generated by the column vectors of $L_{N}^\ast$. 

Now, by using our result on the normal vectors in the previous section, we will obtain a bound $p^{-k} + O(e^{-n^{c}})$ for this event.

\begin{lemma}\label{lemma:linear:lap} Assume that $1\le k\le cN$.  There exists a positive constant $c$ so that for any sufficiently large prime $p$ at most $\exp(n^{c})$ there exists an event $\CE_{N}^\ast$ with $\P(\CE_{N}^\ast) \ge 1- \exp(-n^{c})$ such that the following holds
$$\P\big(X_{N+1}|_{[N]} \in H_N \big | \rank(L_{N}^\ast)=N-k \wedge \CE_{N}^\ast \big) = \frac{1}{p^k} + O(\exp(-n^{c})).$$
\end{lemma}
We also refer the reader to \cite[Theorem 1.1]{M1}, \cite[Theorem A.1, A.4]{NgP}, and \cite[Theorem 8.2]{NgW} for similar results for random i.i.d. matrices as well as its i.i.d. Laplacian.

To prove Lemma \ref{lemma:linear:lap}, we first show the following analog of \cite[Lemma 7.1]{LMNg} and \cite[Lemma 2.11]{KNg} for adapted vectors.

\begin{lemma}\label{lemma:highdim} Assume that $H$ is a subspace in $\F_p^N$ of codimension $k$, and such that for any $\Bw \in H^\perp$ we have 
$$\rho(\Bw_{I_0}) \le \delta_0.$$ Then with $X$ being adapted to $I_0$ as in \ref{def:X} we have 
		$$|\P(X \in H) - 1/p^k|\le 2 \delta_0.$$
	\end{lemma}
		\begin{proof}[Proof of Lemma \ref{lemma:highdim}] 			
			We have the following identity,
	\begin{align*}
		1_{a_1 = 0 \wedge \dots \wedge a_k = 0} = \frac{1}{p^k}\sum_{t_1,\dots,t_k\in\F_p} e_p(a_1t_1 + \dots + a_kt_k).
	\end{align*}
	Therefore, letting $\Bv_1,\dots, \Bv_k$ be a basis for $H^\perp$,
	\begin{align*}
		\P(X\in H) &= \P(X\cdot \Bv_1 \wedge \dots \wedge X\cdot \Bv_k) \\
		&= \E\Big[ \frac{1}{p^k}\sum_{t_1,\dots,t_k\in\F_p} e_p(t_1[X\cdot \Bv_1] + \dots + t_k[X\cdot \Bv_d])\Big] \\
		&= \E\Big[ \frac{1}{p^k} +  \frac{1}{p^k}\sum_{\substack{t_1,\dots,t_k\in\F_p\\ \text{ not all $t_i$ are zero}}} e_p(t_1[X\cdot \Bv_1] + \dots + t_k[X\cdot \Bv_k])\Big]\\
		&=  \frac{1}{p^k} +  \frac{1}{p^k}\sum_{\substack{t_1,\dots,t_k\in\F_p\\ \text{ not all $t_i$ are zero}}} \E e_p(t_1[X\cdot \Bv_1] + \dots + t_k[X\cdot \Bv_k]).
	\end{align*}
	We now split the sum into projective equivalence classes. Let $\sim$ be the equivalence relation given by $(t_1,\dots,t_k)\sim(t_1',\dots,t_k')$ if there exists $t\neq 0$ such that $(t_1,\dots,t_k) = (t\cdot t_1',\dots,t\cdot t_k')$. Not worrying about our choice of representative on account of our inner sum,
	\begin{align*}
		&= \frac{1}{p^k} + \frac{1}{p^k}\sum_{(t_1,\dots,t_k)\in (\F_p^N)^\times/\sim}\bigg[\sum_{t\in\F_p^\times} e_p(t (t_1[X\cdot \Bv_1] + \dots + t_k[X\cdot \Bv_k]))\bigg] \\
		&= \frac{1}{p^k} + \frac{1}{p^k}\sum_{(t_1,\dots,t_k)\in (\F_p^N)^\times/\sim}p\bigg[\frac{1}{p}\sum_{t\in\F_p^\times} e_p(t(X\cdot \sum_{i=1}^n t_i \Bv_i))\bigg].
	\end{align*}
	Now observe,
	\begin{align*}
		\frac{1}{p}\sum_{t\in\F_p^\times} e_p(t(X\cdot \sum_{i=1}^k t_i \Bv_i))= \P(X\cdot \sum_{i=1}^k t_i \Bv_i = 0) - 1/p.
	\end{align*}
	The $\Bv_i$ are a basis and the $t_i$ are not all zero so $\sum_{i=1}^k t_i\Bv_i$ is a nonzero element of $H^\perp$. By assumption then this is bounded by $\delta_0$. There are $\frac{p^k-1}{p-1}$ elements in $(t_1,\dots,t_k)\in (\F_p^n)^\times/\sim$. We have $\frac{p(p^k-1)}{p^k(p-1)} \leq 2$. So by the triangle inequality
	\begin{align*}
		|\P(X\in H) - 1/p^k| \leq 2 \delta_0.
	\end{align*}

\end{proof}
			
		\begin{proof}[Proof of Lemma \ref{lemma:linear:lap}] Let $\CE_{N}^\ast$ be the event from Proposition \ref{prop:structure:subexp:lap} applied to $L_N^\ast$ for $I_0=I_{N+1}$.	Under $\CE_{N}^\ast$ we have learned that for any $\Bw \in H_N$ then $\rho_L(\Bw_{I_{N+1}}) = O(\exp(-n^c))$, where $I_{N+1}$ is the given index set where $X_{N+1}$ is adapted. 
		\end{proof}
For symmetric matrices $\rank(L_{N+1}^\ast) \le \rk(L_{N}^\ast)+1$ does not automatically implies that $\rank(L_{N+1}^\ast) =\rk(L_{N}^\ast)$, so we have to consider the events $L_{N+1}^\ast) =\rk(L_{N}^\ast)$ and $\rank(L_{N+1}^\ast) =\rk(L_{N}^\ast)+1$ {\it separately}. Let's consider the first event that 
$$\rank(L_{N+1}^\ast) =\rk(L_{N}^\ast).$$ 
Here beside the event $X_{N+1}|_{[N]} \in H_N$ considered in Lemma \ref{lemma:linear:lap}, $X_{N+1}=\row_{N+1}(L_{N+1}^\ast)$ also belongs to the subspace generated by $\row_1(L_{N+1}^\ast),\dots, \row_N(L_{N+1}^\ast)$. This latter condition can be interpreted in quadratic forms as follows. Conditioning on $\rank(L_{N}^\ast)=N-k \wedge \CE_{N}^\ast$, by Claim \ref{claim:principleminor} we can assume that $L_{I \times I}^\ast$ is a submatrix of full rank $N-k$ in $L_{N}^\ast$, for some $I\subset [N]$ and $|I|=N-k$. Let $B=(b_{ij})$ be the inverse of $L_{I \times I}^\ast$ in $\F_p$. The exposure of $(x_1,\dots, x_{N-k}, x_{N+1})$ would then increase the rank of $L_{N+1}^\ast$ except when
\begin{equation}\label{eqn:b_{ij}}
\sum_{ij} b_{ij} x_i x_j +x_{N+1} = 0.
\end{equation}
This leads us to consider the following for adapted vectors.

\begin{lemma}\label{lemma:quad:lap} Let $k\ge 0$. Let $H_N$  and $\CE_N^\ast$ be as in the proof of Lemma \ref{lemma:linear:lap}. Let $X$ be as in \ref{def:X} for a given index set $I_0$ satisfying \eqref{eqn:I_N+1}. For any sufficiently large prime $p$ at most $\exp(n^{c})$, we have 
$$\Big|\P\big(X \in H_N \wedge f(X) =0 | \rank(L_{N}^\ast)=N-k \wedge \CE_{N}^\ast \big) - p^{-k-1}\Big| = O(\exp(-\Theta(n^{c}))),$$
where $f(X) = \sum_{1\le i,j\le N-k} b_{ij} x_i x_j + x_{N+1}$. 
\end{lemma}

\begin{proof}[Proof of Lemma \ref{lemma:quad:lap}] We write 
\begin{align*}
\P(X \in H_N \wedge f(X) =0) &= p^{-k-1} \sum_{\xi \in H_N^\perp, t \in \Z/p\Z} \E e_p(X \cdot \xi + f(X)t)\\
&=p^{-k-1} + p^{-k-1}\sum_{\xi \in H_N^\perp, t \in \Z/p\Z, t\neq 0 }  \E e_p(X \cdot \xi + f(X)t) +  p^{-k-1} \sum_{\xi \neq 0, \xi \in H_N^\perp} \E e_p(X \cdot \xi).
\end{align*}
Note that the third sum is 
$$p^{-k-1} \sum_{\xi \neq 0, \xi \in H_N^\perp} \E e_p(X \cdot \xi) = p^{-k-1} (\sum_{\xi \in H_N^\perp} \E e_p(X \cdot \xi) - 1) =\frac{1}{p}\P(X \in H_N) - p^{-k-1}.$$
Hence the third sum can be bounded by $O(\exp(-n^c))$ in absolute value using the condition on $G_n$ from Lemma \ref{lemma:linear:lap}. 

For the second sum, we have 
$$|p^{-k-1} \sum_{\xi \in H_N^\perp, t \in \F_p, t\neq 0} \E e_p(X \cdot \xi + f(X)t| \le \Big(p^{-k-1} \sum_{\xi \in H_N^\perp, t \in \Z/p\Z, t\neq 0} |\E e_p(X \cdot \xi + f(X)t|^4 \Big)^{1/4}.$$
Recall that $|I_0|$ satisfies \eqref{eqn:I_N+1}. For now we let $I_2\subset I_0$ of size $\lfloor cN \rfloor$, and let $I_1=I\backslash I_2$. 
By using Cauchy-Schwarz as in the proof of Lemma \ref{lemma:decoupling1}, we then bound this by
\begin{align*}
& \  \  \ p^{-k-1} \sum_{\xi \in H_N^\perp, t \in \Z/p\Z, t\neq 0} \big|\E e_p(X \cdot \xi + f(X)t\big|^4\\
&\le p^{-k-1} \sum_{\xi \in H_N^\perp, t \in \Z/p\Z, t\neq 0} \Big|\E_{X_{I_1}} \E_{X_{I_2}, X_{I_2}'}e_p\big((X_{I_2} -X_{I_2}') \cdot \xi + (f(X_{I_1},X_{I_2}) -f(X_{I_1}, X_{I_2}'))t\big)\Big|^2\\
&\le p^{-k-1} \sum_{\xi \in H_N^\perp, t \in \Z/p\Z, t\neq 0}  \E_{X_{I_2}, X_{I_2}'} \E_{X_{I_1}, X_{I_2}}e_p\Big((f(X_{I_1},X_{I_2}) -f(X_{I_1}, X_{I_2}') - f(X_{I_1}', X_{I_2})+f(X_{I_1}', X_{I_2}'))t\Big)\\
&= p^{-k-1} \sum_{\xi \in H_N^\perp, t \in \Z/p\Z, t\neq 0}  \E_{Y}e_p(\sum_{i\in I_1, j\in I_2} b_{ij}y_i y_j t)=\frac{1}{p} \sum_{ t \in\Z/p\Z, t\neq 0}  \E_{Y}e_p(\sum_{i\in I_1, j\in I_2} b_{ij}y_i y_j t)\\
&=\P(Y_{I_1}' \cdot B Y_{I_2}' = 0) - \frac{1}{p},
\end{align*}
where $y_i, i\in I_0$ are i.i.d. copies of the the random variable $\xi-\xi'$, and where $Y_{I_1}'$ and $Y_{I_2}'$ are the vectors in $\F_p^{I}$ obtained from $Y_{I_1}$ (i.e. $Y_I|_{I_1}$) and $Y_{I_2}$ (i.e. $Y_I|_{I_2}$) resp. by appending zero entries. 

Let $J = [N]\bs I$, so $|J|=k$ and that $[N] = I_1 \cup I_2 \cup J$. We let $B'$ be the symmetric matrix of size $n$ obtained from $B$ by simply adding zero entries. Similarly, let $Y_{I_1}''$ and $Y_{I_2}''$ be the vectors in $\F_p^N$ obtained from $Y_{I_1}'$ and $Y_{I_2}'$ by appending zero components. Then $Y_{I_1}' \cdot B Y_{I_2}' = 0$ is equivalent to
$$Y_{I_1}'' \cdot B' Y_{I_2}'' =0.$$
Furthermore, notice that the vector $Y_{I_2}''$ is non-zero (in $\F_p^N$) with probability at least $1-\exp(-\Theta(n))$ (because $Y_{I_2}''$ has $\lfloor c N\rfloor$ i.i.d. entries) and that
$$L_{N}^\ast(B' Y_{I_2}'') = (L_{N}^\ast B') Y_{I_2}''= Y_{I_2}'',$$
where we used the fact that $B=(b_{ij})$ is the inverse of $L_{I \times I}^\ast$ in $\F_p$.

It follows by Proposition \ref{prop:structure:subexp:lap} that on $\CE_{N}^\ast$
$$\rho_L(B' Y_{I_2}''|_{I_0 \cap I_1}) \le \exp(-\Theta(n^c)).$$
Conditioned on such $Y_{I_2}''$, by \eqref{eqn:Fourier}  
$$|\P_{Y_{I_1}'}(Y_{I_1}' \cdot B Y_{I_2}' = 0) - \frac{1}{p}| = |\P_{Y_{I_1}'}(Y_{I_1}'' \cdot B' Y_{I_2}'' = 0) - \frac{1}{p}|\le \exp(-\Theta(n^{c})).$$
 \end{proof}

Having obtained the necessary concentration bounds, we will next put things together to obtain the rank statistics for the three matrix models, our method is similar to that of \cite{Msym,KNg}.

\begin{prop}[Rank relations for Laplacian matrices]\label{prop:rankevolution:lap} There exists a positive constant $c$ such that the following holds. Assume that $p$ is a sufficiently large prime and $p \le \exp(n^c)$. For $N_0 \le N\le n$ there exists an event $\CE_{N}^\ast$ on the $\sigma$-algebra generated by $L_n$ and by the neighbor reshuffling process up to step $N$ such that $\P(\CE_{N}^\ast) \ge 1 - \exp(-n^c)$ and for $k\le cN$ 
\begin{itemize}
\item assume that $k\ge 1$ then
$$\Big|\P\big(\rk(L_{N+1}^\ast) \le \rk(L_{N}^\ast)+1 | \CE_{N}^\ast \wedge \rk(L_{N}^\ast)=N-k)\big) -\frac{1}{p^{k}}\Big| \le  \exp(-\Theta(n^{c}));$$
\item assume that $k\ge 0$ then
$$\Big|\P\big(\rk(L_{N+1}^\ast) = \rk(L_{N}^\ast) |\CE_{N}^\ast \wedge \rk(L_{N}^\ast)=N-k)\big) -\frac{1}{p^{k+1}}\Big| \le \exp(-\Theta(n^{c})).$$
\end{itemize}
\end{prop}

By a similar much much simpler method (where we don't have to do the reshuffling process, and where $I_0= [N]$ at each step, but we still need variants of Lemma \ref{lemma:linear:lap} and Lemma \ref{lemma:quad:lap} with obvious modifications) we also obtain the following for symmetric matrices.

\begin{prop}[Rank relations for symmetric matrices]\label{prop:rankevolution:sym} There exists a positive constant $c$ such that the following holds. Assume that $p$ is a sufficiently large prime and $p \le \exp(n^c)$.  For $N_0 \le N\le n$ there exists an event $\CE_{N}$ on the $\sigma$-algebra generated by the first $N$ rows (and columns) of the matrix $M_N= M_{N\times N}$ (as a principle minor of $M_n$ from Theorem \ref{theorem:cyclic:sym}) such that $\P(\CE_N) \ge 1 - \exp(-n^c)$ and for $k\le cN$ 
\begin{itemize}
\item assume that $k\ge 1$ then
$$\Big|\P\big(\rk(M_{N+1}) \le \rk(M_N)+1 | \CE_N \wedge \rk(M_N)=N-k)\big) -\frac{1}{p^{k}}\Big| \le  \exp(-\Theta(n^{c}));$$
\item assume that $k\ge 0$ then
$$\Big|\P\big(\rk(M_{N+1}) = \rk(M_N) | \CE_N \wedge \rk(M_N)=N-k)\big) -\frac{1}{p^{k+1}}\Big| \le \exp(-\Theta(n^{c})).$$
\end{itemize}
\end{prop}

Next we comment on random skew-symmetric matrices. Here the rank evolution is slightly simpler and different. We note that, similarly to random symmetric matricers, Lemma \ref{lemma:linear:lap} also works for random skew-symmetric matrices. Let us call $\CE_N$ the event under consideration. The key difference here is that the event that for skew-symmetric matrices, $\rank(A_{N+1}) \le \rank(A_N)+1$ implies that $\rank(A_{N+1}) =\rk(A_N)$. This is because if $X_{N+1}|_{[N]} = A_N Y_0$ for some $Y_0$, then $Y_0^T \cdot X_{N+1}|_{[N]} =0$, and so $Y_0^T A_{[N] \times [N+1]} = X_{N+1}^T$, that is the $(N+1)$-th row also belongs the the linear space spanned by the first $n$ rows. Thus Lemma \ref{lemma:linear:lap} applied to skew-symmetric matrices imply

\begin{prop}[Rank relations for skew-symmetric matrices]\label{prop:rankevolution:alt}  There exists a positive constant $c$ such that the following holds. 
Assume that $p$ is a sufficiently large prime and $p \le \exp(n^c)$.  For $N_0 \le N\le n$ there exists an event $\CE_{N}$ on the $\sigma$-algebra generated by the first $N$ rows (and columns) of the matrix $A_N= A_{N\times N}$ (as a principle minor of $A_N$ from Theorem \ref{theorem:cyclic:alt}) such that $\P(\CE_N) \ge 1 - \exp(-n^c)$ and for $k\le cN$ \footnote{Strictly speaking, Lemma \ref{lemma:linear:lap} just gave $k\ge 1$, but for $k=0$ the bound automatically holds with probability one.} then
$$\Big|\P\big(\rk(A_{N+1}) = \rk(A_N)  | \CE_N \wedge \rk(A_N)=N-k\big) -\frac{1}{p^{k}}\Big| =  O(\exp(-n^{c})).$$
\end{prop}
One can then deduce the asymptotic probability of the event $\rk(A_{N+1}) = \rk(A_N)+2$ under $\CE_N \wedge \rk(A_N)=N-k$. We note that the transition probability $\P(\rk(A_{N+1}) = \rk(A_N))$ of the skew-symmetric case above is quite similar to that of the i.i.d. models considered  in \cite{M1, NgW, NgP} (where there we expose the first $N$ column vectors of the matrix of size $n$). However the difference here is that $\rk(A_{N+1})$ jumps by 2 (rather than 1) in the complement event.

To complete the subsection, we remark that Propositions~\ref{prop:rankevolution:sym} and~\ref{prop:rankevolution:alt} automatically extend to the uniform models where $M_N, A_N$ are uniformly chosen from the set of all symmetric matrices and skew-symmetric matrices in $\F_p$ because the uniformly chosen entries are clearly $\al$-balanced.

\subsection{The rank evolving process: completing the proof of Proposition \ref{prop:moderate:lap} and Proposition \ref{prop:moderate:alt}} Let $G_N$ be either $A_N, M_N$ or $L_N^\ast$. We will start from $N=N_0 = \lfloor cn \rfloor$, which is assumed to possess $\CE_{N_0}$ and have rank at least $N_0 -N_0^{1/2+c}$, where we have learned that the latter event has probability at least $1- \exp(\Theta(n))$ by Lemma \ref{lemma:quadratic}. We will apply Propositions \ref{prop:rankevolution:alt} and \ref{prop:rankevolution:lap}  for $G_{N_0+1}, \dots, G_n$. To compare with the rank evolution of the uniform model, for convenience we will use the following result from \cite[Theorem 5.3]{NgW}.
\begin{theorem}\label{theorem:comparison}
Let $x_{N_0},\dots,x_n,g_{N_0-1},\dots,g_{n-1}$ be a sequence of random variables.
  Let $y_{N_0},\dots,y_n$ be a sequence of  random variables where $y_{N_0}=x_{N_0}$.   We assume each $x_N,y_N$ takes on at most countably many values, and $g_N\in \{0,1\}$.
  Suppose that  for $0\leq N \leq n-1$, 
  \begin{align*}
&\P(y_{N+1}=s|y_N=r)=\P(x_{N+1}=s |x_N=r\textrm{ and }g_N=1) +\delta(N,r,s) \\
&\textrm{ for all $r$ and $s$
s.t. $\P(y_N=r)\P(x_N=r\textrm{ and }g_N=1)\ne0$}. 
\end{align*}
Then for any set $A$ of values taken by $x_n$ and $y_n$, we have
\begin{align*}
&|\P(x_n\in A)-\P(y_n\in A)|\\
&\leq \frac{1}{2}\sum_{N=N_0}^{n-1} \sum_{r} \sum_s |\delta(N,r,s)| \P(x_N =r)  +
\sum_{N=N_0}^{n-1} \Pr(g_N\ne 1),
\end{align*}
where $r$ is summed over $\{ r\ |\ \P(x_{N}=r)\ne 0 \textrm{ and } \P(y_{N}=r)\ne 0)\}$ and
$s$ is summed over $\{ s\ |\ \P(x_{N+1}=s)\ne 0 \textrm{ or } \P(y_{N+1}=s)\ne 0)\}.$
\end{theorem}
For the detailed statistics of skew-symmetric matrices, let $\mu_{alt,n}$ be the rank distribution of the uniform skew-symmetric model in $\F_p$ (i.e. each non-diagonal entry is independent uniform in $\Z/p\Z$). By \cite{Mac} we have 
$$\mu_{alt,n}(k)= \P(\rank(A_{n, uniform}=n-k)= \frac{N_0(n,n-k)}{p^{\binom{n}{2}}}$$
where 
$$N_0(n,2h) = \prod_{i=1}^{h} \frac{p^{2i-2}}{p^{2i}-1} \prod_{i=0}^{2h-1} (p^{n-i} -1) \mbox{ and } N_0(2h+1)=0.$$
However we cannot apply Theorem \ref{theorem:comparison} directly to $x_N = \rank(A_N)$ and $y_N=\rank(A_{N, uniform})$ and $g_N=1_{\CE_N}$ because the above uniform statistics is grown from a zero-dimension matrix. To amend this, we can start from any realization of $A_{N_0}'= A_{N_0}$ and add rows and columns according to the uniform model until $A_n'$. We call this model uniform with initial matrix $A_{N_0}$. Then \cite[Proposition 3.13]{KNg} applied to this $A_n'$ shows that its rank statistics is extremely close to the uniform model, that 
 $$d_{TV}(n-\rank(A_n'/p),\mu_{alt, n}) \le p^{-\Theta(n)}.$$
Now we can apply  Theorem \ref{theorem:comparison} and use the triangle inequality to conclude that 
$$d_{TV}(n-\rank(A_n/p),\mu_{alt, n}) \le e^{-n^c}.$$
On the other hand, by \cite[Section 5]{FG} 
$$d_{TV}(\mu_{alt, n}, \mu_{alt,e}) = O(\frac{1}{p^{n+1}}) \mbox{ if $n$ is even}$$
and
$$d_{TV}(\mu_{alt, n}, \mu_{alt,o}) = O(\frac{1}{p^{n+1}}) \mbox{ if $n$ is odd}$$
where the limiting distributions are given by 
\[
	\P(\mu_{alt,e} = k):= \begin{cases}
\prod_{i=0}^\infty (1 -p^{-2i-1}) \frac{p^k}{\prod_{i=1}^k (p^i-1)}, & \mbox{$k$ is even}\\
 0, \mbox{ $k$ is odd}
\end{cases}
\]
and
\[
	\P(\mu_{alt,o}=k):= \begin{cases}
\prod_{i=0}^\infty (1 -p^{-2i-1}) \frac{p^k}{\prod_{i=1}^k (p^i-1)}, & \mbox{$k$ is odd}\\
 0, \mbox{ $k$ is even}
\end{cases}
\]
Putting together 
\begin{theorem}[Ranks statistics of random skew-symmetric matrices]\label{theorem:rankstatistics:alt} Assume that $p$ is a sufficiently large prime and $p \le \exp(n^c)$. Assume that $A_n$ is as in Theorem \ref{theorem:cyclic:alt}, then 
$$d_{TV}(n-\rank(A_n/p),\mu_{alt, e}) =O(e^{-n^c}), \text{ $n$ is even}$$
and
$$d_{TV}(n-\rank(A_n/p),\mu_{alt, o})=O(e^{-n^c}), \text{ $n$ is odd.}$$
\end{theorem}
The claim of Proposition \ref{prop:moderate:alt} then follows easily because 
$$\sum_{k\ge 3} \P(\mu_{alt,e}=k)=\sum_{k\ge 4} \P(\mu_{alt,e}=k) = O(\frac{1}{p^6})$$
and
$$\sum_{k\ge 3} \P(\mu_{alt,o}=k)=O(\frac{1}{p^3}).$$

Now for symmetric and Laplacian matrices, let with $\mu_{sym,n}$ be the rank distribution of the random uniform symmetric matrix of size $n$ where each $x_{ij}, 1\le i\le j\le n$ are independent uniform over $\Z/p\Z$. By \cite{C,Mac} we have 
$$\mu_{sym,n}(k)=\P(\rank(M_{uniform}=n-k)) = \frac{N(n,n-k)}{p^{\binom{n+1}{2}}}$$
where 
$$N(n,2h) = \prod_{i=1}^h \frac{p^{2i}}{p^{2i}-1} \prod_{i=0}^{2h-1} (p^{n-i} -1), 2h\le n$$
and
$$N(n,2h+1) = \prod_{i=1}^h \frac{p^{2i}}{p^{2i}-1} \prod_{i=0}^{2h} (p^{n-i} -1), 2h+1\le n.$$
Arguing as in the skew-symmetric case, we can grow the uniform model from  $G_{N_0}'= G_{N_0}$, and add rows and columns according to the uniform model to form $G_{N_0+1}'$ and so on until $G_n'$. Then \cite[Proposition 3.13]{KNg} applied to this uniform model $G_n'$ with initial $G_{N_0}$ shows that its rank statistics is extremely close to the uniform model, and hence by  Theorem \ref{theorem:comparison} and by the triangle inequality
$$d_{TV}(n-\rank(G_n/p),\mu_{sym, n}) \le e^{-n^c}.$$
On the other hand, by \cite[Section 4]{FG} we have
$$d_{TV}(\mu_{sym, n}, \mu_{sym}) = O(\frac{1}{p^{n+1}})$$
where $$\P(\mu_{sym}=k)= \frac{\prod_{i=0}^\infty (1-p^{-2i-1})}{\prod_{i=1}^k (p^i-1)}.$$
Putting together we have thus obtained
\begin{theorem}[Ranks statistics of random symmetric matrices]\label{theorem:rankstatistics:sym} Assume that $p$ is a sufficiently large prime and $p \le \exp(n^c)$. Assume that $M_n$ is as in Theorem \ref{theorem:cyclic:sym}, then $$d_{TV}(n-\rank(M_n/p),\mu_{sym}) =O(e^{-n^c}).$$
\end{theorem}
\begin{theorem}[Ranks statistics of random Laplacian matrices]\label{theorem:rankstatistics:lap} Assume that $p$ is a sufficiently large prime and $p \le \exp(n^c)$. Assume that $L_n$ is as in Theorem \ref{theorem:cyclic:lap}, then 
$$d_{TV}(n-\rank(L_n/p),\mu_{sym}) =O(e^{-n^c}).$$
\end{theorem}
Proposition \ref{prop:moderate:lap} then follows because
$$\sum_{k\ge 2} \P(\mu_{sym}=k)=O(\frac{1}{p^3}).$$

To conclude the section, as $G_n=A_n,M_n,L_n$ are integral, if $\det(G_n)=0$ then for any prime $p$ we have $G_n/p$ is singular, i.e. $\rank(G_n/p)\le n-1$. Using our result of corank comparison, by choosing $p$ sub-exponentially large we obtain the following bound for singularity. 

\begin{corollary}[Singularity of skew-symmetric, symmetric and Laplacian matrices]\label{cor:singularity} There exists a positive constant $c>0$ so that
$$\P(\det(A_{2n}) = 0)= O(\exp(-n^c)); \P(\det(M_n)=0)= O(\exp(-n^c));$$
and 
$$\P(\det(L_n)=0) = O(\exp(-n^c)).$$
\end{corollary}
For the random symmetric model we also refer the reader to \cite{CMMM, FJ, Ver} for analogous sub-exponential bounds (with explicit constants), and to a more recent work \cite{CJMS} for exponential bounds. Our results for the skew-symmetric and Laplacian cases are new.
\vskip .2in
\begin{remark} While the current paper was under preparation, a recent paper by Ferber et. al. \cite{FJSS} has obtained a similar result to Theorem \ref{theorem:rankstatistics:sym} for Bernoulli matrices with explicit $c$. 
Here in the symmetric case we established this result for general random matrices of balanced entries. We also obtained similar results for the skew-symmetric model, which has different and interesting rank statistics. However, as the reader can see, our main difficulty for the moderate primes lies in Laplacian matrices, where we had to look at the local structures of the generalized normal vectors of $L_N$ over a random index set $I_{N+1}$. In fact, prior to this work, not much has been known about this complicated model from the viewpoint of rank statistics and singularity. The only results that we are aware of for the Laplacian are from \cite{Wood2017}, where it justifies Theorem \ref{theorem:rankstatistics:lap} for relatively small $p$, and that $\P(\det(L_n)=0)=1-o(1)$ for a somewhat implicit rate of convergence. 
\end{remark}

\section{Treatment for large primes: control of blowing up and quadratic inverse theorems}\label{section:bilinear} 
While in the previous sections we are able to control all primes up to $\exp(n^c)$, our treatment for larger primes cannot follow the same way because taking union bound over all large primes is extremely costly. We cannot take union bound even with the conjectural forms of Theorems \ref{theorem:rankstatistics:alt}, \ref{theorem:rankstatistics:sym} and \ref{theorem:rankstatistics:lap} that the error bounds are of form $\exp(-cn)$ in place of $\exp(-n^c)$ because there are $n^{cn}$ primes to handle.
To avoid this obstacle, we will find a common structure that is passable to {\it all} large primes simultaneously. In general, this idea was also applied in \cite{NgW} to treat with random non-symmetric matrices, but the extension to symmetric and Laplacian matrices does require significantly new ideas.

Our main results of this section are Theorem \ref{theorem:ILO:bilinear} and Theorem \ref{theorem:ILO:quadratic}. To start with, we first recall from \cite{TVbook} a notion of additive structures in abelian groups. Let $G$ be an (additive) abelian group. 

\begin{definition}\label{def:gap}
A set $Q$ is a \emph{generalized arithmetic progression} (GAP) of
rank $r$ if it can be expressed as in the form
$$Q= \{a_0+ x_1a_1 + \dots +x_r a_r| M_i \le x_i \le M_i' \hbox{ and $x_i\in\Z$ for all } 1 \leq i \leq r\}$$
for some elements $a_0,\ldots,a_r$ of $G$, and for some integers $M_1,\ldots,M_r$ and $M'_1,\ldots,M'_r$.

It is convenient to think of $Q$ as
the image of an integer box $B:= \{(x_1, \dots, x_r) \in \Z^r| M_i \le x_i
\le M_i' \} $ under the linear map
$$\Phi: (x_1,\dots, x_r) \mapsto a_0+ x_1a_1 + \dots + x_r a_r. $$
Given $Q$ with a representation as above
\begin{itemize}
\item the numbers $a_i$ are  \emph{generators} of $Q$, the numbers $M_i$ and $M_i'$ are  \emph{dimensions} of $Q$, and $\Vol(Q) := |B|$ is the \emph{volume} of $Q$ associated to this presentation (i.e. this choice of $a_i,M_i,M_i'$);
\vskip .05in
\item we say that $Q$ is \emph{proper} for this presentation if the above linear map is one to one, or equivalently if $|Q| =|B|$;
\vskip .05in
\item If $-M_i=M_i'$ for all $i\ge 1$ and $a_0=0$, we say that $Q$ is {\it symmetric} for this presentation.
\end{itemize}
\end{definition}

We note again that unlike in other applications of additive structures in random matrix theory that researchers used structures to enumerate vectors satisfying certain properties, here we use the above structure as it is to pass to all large primes; so the structure itself is important.

The following inverse-type result established by the current authors from \cite{NgW} (which was in turn motivated by the inverse-type idea from \cite{TVinverse}) will allow us to prove bounds sharper than Theorem \ref{theorem:LO}. 

\begin{theorem}\label{theorem:ILO} Let $\eps<1$ and $C$ be positive constants. Assume that $p \ge C' N^C$ is a prime where $C'$ is sufficiently large.  Let $\xi$ be a random variable taking values in $\Z/p\Z$ which is $\alpha$-balanced.
 Assume $\Bw=(w_1,\dots, w_N)\in(\Z/p\Z)^N$ such that  
 $$\rho_l (\Bw) := \sup_{a\in \Z/p\Z} \P(\xi_1 w_1 +\dots + \xi_N w_N=a) \ge  N^{-C},$$
 where $\xi_1,\dots, \xi_N$ are i.i.d. copies of  $\xi$.  Then for any $N^{\ep/2} \a^{-1} \le N' \le N$ there exists a proper symmetric GAP $Q$ of rank $r=O_{C,\ep}(1)$ which contains all but $N'$ elements of $w$ (counting multiplicity), where 
$$|Q|\le \max\left \{1, O_{C,\ep}(\rho_l^{-1}/(\al N')^{r/2})\right \} .$$ 
\end{theorem}

We remark that $\rho_l(.)$ (where the subscript stands for linear) is slightly different from $\rho(.)$ defined in \eqref{eqn:rho}. We removed $1/p$ in $\rho_l(.)$ as it has little effect when $p$ is large. Note that this result continues to hold for $\a$ could be as small as $n^{-1+o(1)}$, but we just assume $\a$ to be a constant here as usual. We also refer the reader to \cite{KNgP} for various versions in general Abelian groups where GAP is replaced by coset-progressions.

We now introduce an inverse result for bilinear  forms, which is another contribution of the current paper.

\begin{theorem}\label{theorem:ILO:bilinear}
Let $C>0$ and $\eps>0$. There exists a constant $A$ such that the following holds for any prime $p\ge N^A$. Assume that $B=(b_{ij})_{1\le i,j\le N}$ is an array of elements of $\Z/p\Z$ so that 
$$\rho_b:=\sup_{a \in \Z/p\Z} \P_{X,Y}(\sum_{i,j\le n}b_{ij}x_{i}y_{j} =a) \ge N^{-C},$$
where $X=(x_{1},\dots,x_{N}), Y=(y_{1},\dots,y_{N})$, and $x_{i}$ and $y_{i}$ are i.i.d. copies of $\xi$ as in Theorem \ref{theorem:ILO}.
Then, there exist an integer $k\neq 0, |k|=N^{O_{C,\eps}(1)}$, a set of $r=O(1)$ rows $\row_{i_1}, \dots, \row_{i_r}$ of $B=(b_{ij})_{1\le i, j\le n}$, and set $I$ of size at least $N-N^\ep$ such that for each $i\in I$, there exist integers $k_{ii_1},\dots, k_{ii_r}$, all bounded by $N^{O_{C,\ep}(1)}$, with the following property
\begin{equation}\label{eqn:special}
\P_Z\Big( Z \cdot (k\row_i(B) + \sum_{j=1}^r k_{ii_j} \row_{i_j}(B)) = 0 \mod p \Big)\ge N^{-O_{C,\ep}(1)},
\end{equation}
where $Z=(z_1,\dots,z_N)$ and $z_i$ are i.i.d. copies of $\xi$.
\end{theorem}

In connection to our previous sections, the parameter $N$ will be chosen to satisfy $cn\le N \le n$ throughout this section. In our later application we just need to consider $p\ge \exp(n^c)$, so the above assumption on the range of $p$ is natural. On the other hand, it is of independent interest to extend the results of Theorem~\ref{theorem:ILO:bilinear} (and Theorem~\ref{theorem:ILO:quadratic} below) to small $p$, we hope to be able to address this issue elsewhere.

So in a way our result says that the rows of $B$ have low rank modulo some GAP noise. (Indeed it follows from \eqref{eqn:special} and from Theorem \ref{theorem:ILO} that for each $i\in I$, most of the entries of $k\row_i(B) + \sum_{j} k_{ii_j} \row_{i_j}(B)$ belong to a symmetric GAP $P_i$ of rank $O(1)$ and size $N^{O(1)}$. One can in fact unify these structures into one but we will not do it here.) We will also introduce a quadratic version later (Theorem \ref{theorem:ILO:quadratic}) and prove it in the appendix. For the rest of this section we give a proof of Theorem \ref{theorem:ILO:bilinear}, adapting the method of \cite{Ng} toward $\Z/p\Z$. 
 
We first rely on the following simple fact about generalized arithmetic progressions of small rank and large characteristic.
\begin{fact}\label{fact:GAP} Let $C>$ be given. Assume that $q_1,\dots,q_{r+1}$ are elements of a GAP of rank $r$ and of cardinality $N^C$ in $\Z/p\Z$, where $p\ge N^A$ for sufficiently large $A$ depending on $r$ and $C$, then there exist integer coefficients $\alpha_1,\dots,\alpha_r$ with $|\alpha_i|\le N^{rC}$, not all zero, such that in $\Z/p\Z$
$$\sum_{i=1}^{r+1} \alpha_i q_i =0.$$
\end{fact}
To prove Theorem \ref{theorem:ILO:bilinear}, we begin by applying Theorem \ref{theorem:ILO}.
\begin{lemma}\label{lemma:roworthogonal} 
Let $\ep<1$ and $C$ be positive constants. With the assumption as in Theorem \ref{theorem:ILO}, assume that $\rho_b \ge N^{-C}$. Then the following holds with probability at least $3\rho_b/4$ with respect to $Y$. There exist a proper symmetric GAP $Q_{Y}$ in $\Z/p\Z$ of rank $O_{C,\ep}(1)$ and size $O_{C,\ep}(1/\rho_b)$ and a set $I_{Y}$ of $N-N^\ep$ indices such that for each $i\in I_{Y}$ we have $\sum_{j}b_{ij}y_j \in Q_{Y}$.
\end{lemma}

\begin{proof}
For short we write $\sum_{i,j} b_{ij}x_{i}y_j = \sum_{i=1}^N x_{i} B_i(Y)$, where 
$$B_i(Y):=\sum_{j}b_{ij}y_j.$$
We call a vector $Y$ {\it good} if $\P_{\Bx}(\sum_{i=1}^n x_{i}B_i(Y)=a) \ge \rho_b/4$. We call $Y$ {\it bad} otherwise. Let $G$ be the collection of good vectors. By averaging, one can show that   the probability of random vector $Y$ being good is at least, say $3\rho_b/4$.

Next, we consider good vectors $Y \in G$. By definition, $\P_{\Bx} (\sum_{i=1}^N x_{i} B_i(Y)=a) \ge \rho_b/4$. A direct application of Theorem \ref{theorem:ILO} to the sequence $B_i(Y)$, $i=1,\dots,N$ yields the desired result.
\end{proof}

By a useful property of GAP containment (see for instance \cite[Section 8]{TVinverse} and \cite[Theorem 6.1]{Ng}), we may assume that the $q_i(Y)$ span $Q_{Y}$  in $\Z/p\Z$. From now on we fix such a $Q_{Y}$ for each $Y$. Recall that $G$ is the collection of good vectors and we have  
\begin{equation}
\P_{Y}(Y\in G)\ge 3\rho_b/4.
\end{equation}

Now we state a main lemma for the proof of Theorem \ref{theorem:ILO:bilinear}.

\begin{lemma}\label{lemma:reduction} There exits an index set $I$ of size at least $N-2N^\ep$, an index set $I_0$ of size $O_{C,\ep}(1)$, and an integer $k\neq 0$ with $|k|\le N^{O_{C,\ep}(1)}$ such that for any index $i$ from $I$, there are numbers $k_{ii_0} \in \Z, i_0\in I_0$, all bounded by $n^{O_{C,\ep}(1)}$, such that 
$$\P_{Y}\Big(k B_i(Y)+ \sum_{i_0\in I_0} k_{ii_0} B_{i_0}(Y)=0 \mod p \Big)=\rho_b/N^{O_{C,\ep}(1)}.$$
\end{lemma}
Assume this result for a moment.
\begin{proof}[Proof of Theorem \ref{theorem:ILO:bilinear}] From Lemma \ref{lemma:reduction}, for any fixed $i\in I$, we note that
$$k B_i(Y)+ \sum_{i_0\in I_0} k_{ii_0} B_{i_0}(Y) =\sum_j (k b_{ij} + \sum_{i_0} k_{ii_0} b_{i_0j }) y_j.$$
By the conclusion of Lemma \ref{lemma:reduction}, we have 
$$\sup_{a  \in \Z/p\Z}\P_{Y}(\sum_j (k b_{ij} + \sum_{i_0 \in I_0} k_{ii_0} b_{i_0j }) y_j = a \mod p)\ge \rho_b/N^{O_{C,\ep}(1)}.$$ 
\end{proof}
We now give a proof for Lemma \ref{lemma:reduction}.
\begin{proof}[Proof of Lemma \ref{lemma:reduction}] 
For each $Y\in G$, we choose from $I_{Y}$ $s$ indices $i_{(1,Y)},\dots,i_{{(s,Y)}}$ such that $q_{i_{{(j,Y)}}}(Y), 1\le j\le s$, span $Q_{Y}$  in $\Z/p\Z$, where $s$ is the rank of $Q_{Y}$. We note that $s=O_{C,\ep}(1)$ for all $Y\in G$. 

Consider the tuples $(i_{{(1,Y}},\dots,i_{{(s,Y}})$ for all $Y\in G$. Because there are $\sum_{s} O_{C,\ep}(N^s) = N^{O_{C,\ep}(1)}$ possibilities these tuples can take, by the pigeon-hole principle there exists a tuple, say $(1,\dots,r)$ (by rearranging the rows of $B=(b_{ij})$ if needed), such that $(i_{{(1,Y}},\dots,i_{{(s,Y}})=(1,\dots,r)$ for all $Y\in G'$, where $G'$ is a subset of $G$ satisfying 
\begin{equation}
\P_{Y}(Y\in G')\ge \P_{Y}(Y\in G)/N^{O_{C,\ep}(1)} =\rho_b/N^{O_{C,\ep}(1)}.
\end{equation}
For each $1\le i\le r$, we express $q_i(Y)$ in terms of the generators of $Q_{Y}$ for each $Y\in G'$, 
$$q_{i}(Y) = c_{i1}(Y)g_{1}(Y)+\dots + c_{ir}(Y)g_{r}(Y) \mod p,$$ 
where $c_{i1}(Y),\dots c_{ir}(Y)$ are integers bounded by $n^{O_{C,\ep}(1)}$, and $g_{i}(Y)$ are the generators of $Q_{Y}$. Furthermore, as the $q_i(Y)$ span $Q_{Y}$, the vectors $(c_{11},\dots,c_{1r}),\dots, (c_{r1},\dots c_{rr})$ span $\Z^{\rank(Q_{Y})}$.

We show that there are many $Y$ that correspond to the same coefficients $c_{i_1i_2}$. 

\begin{claim}\label{claim:bilinear:common} There exists a (``dense'') subset  $G''\subset G'$ such that the following holds
\begin{itemize}
\item  $\P_{Y}(Y\in G'')\ge \P_{Y}(Y\in G')/N^{O_{C,\ep}(1)} \ge \rho_b/N^{O_{C,\ep}(1)};$
\vskip .1in
\item there exist $r$ tuples  $(c_{11},\dots,c_{1r}),\dots, (c_{r1},\dots c_{rr})$, whose components are integers bounded by $n^{O_{C,\ep}(1)}$ and $(c_{11},\dots,c_{1r}),\dots, (c_{r1},\dots c_{rr})$ span $\Z^{\rank(Q_{Y})}$ such that for all $Y\in G''$
$$q_{i}(Y) = c_{i1}g_{1}(Y)+\dots + c_{ir}g_{r}(Y) \mbox{ for $i=1,\dots,r$.}$$
\end{itemize}
\end{claim}

\begin{proof}[Proof of Claim \ref{claim:bilinear:common}]
Consider the collection $\mathcal{C}$ of the coefficient-tuples 
$$\mathcal{C}:=\Big\{\Big(\big(c_{11}(Y),\dots,c_{1r}(Y)\big);\dots; \big(c_{r1}(Y),\dots c_{rr}(Y)\big)\Big), Y\in G'\Big\} .$$ 
Because the number of possibilities these tuples can take is at most $(N^{O_{C,\ep}(1)})^{r^2} =N^{O_{C,\ep}(1)}$, again by the pigeon-hole principle there exists a coefficient-tuple, say  $\Big((c_{11},\dots,c_{1r}),\dots, (c_{r1},\dots c_{rr})\Big)\in \mathcal{C}$, such that  
\begin{align*}
&\quad \Big(\big(c_{11}(Y),\dots,c_{1r}(Y)\big);\dots; \big(c_{r1}(Y),\dots c_{rr}(Y)\big) \Big)\\ 
&=\Big((c_{11},\dots,c_{1r}),\dots, (c_{r1},\dots c_{rr})\Big)
\end{align*}
for all $Y$ from a subset $G''$ of $G'$ which satisfies 
\begin{equation}
  \P_{Y}(Y\in G'')\ge \P_{Y}(Y\in G')/N^{O_{C,\ep}(1)} \ge \rho_b/N^{O_{C,\ep}(1)}.
\end{equation}
\end{proof}

Now we focus on the elements of $G''$. Because $|I_{Y}|\ge N-N^\ep$ for each $Y\in G''$, by an averaging argument we can obtain the following.

\begin{claim}\label{claim:bilinear:I}
There is a set $I$ of size $N-3N^\ep$ such that $I \cap \{1,\dots,r\} =\emptyset$ and for each $i\in I$ we have 
\begin{equation}\label{eqn:optional}
\P_{Y}(i\in I_{Y}, Y\in G'') \ge \P_{Y}(Y\in G'')/2.
\end{equation}  
\end{claim}

Now we conclude the proof of Lemma \ref{lemma:reduction}. Fix an arbitrary index $i$ from $I$. We concentrate on those $Y\in G''$ where the index $i$ belongs to $I_{Y}$. Because $q_{i}(Y) \in Q_{Y}$, we can write 
$$q_{i}(Y)= c_{1}(Y)g_{1}(Y)+\dots c_{r}(Y)g_{r}(Y) \mod p,$$ 
where $c_{1}(Y),\dots,c_r(Y)$ are integers bounded by $N^{O_{C,\ep}(1)}$.

For short, we denote by $v_{i,Y}$ the vector $(c_{1}(Y),\dots c_{r}(Y))$, we also use the shorthand $v_j$ for the vectors $(c_{j1},\dots,c_{jr})$ obtained from Claim \ref{claim:bilinear:common}. 

Because $Q_{Y}$ is spanned by $q_{1}(Y),\dots, q_{r}(Y)$, we must have $k:=\det({v}_1,\dots {v}_r)\neq 0  \mod p$ (and hence $k\neq 0$ because $p\ge n^A$ with sufficiently large $A$) and that in $\Z/p\Z$
\begin{align*}
&k q_i(Y) + \det({v}_{i,Y}, {v}_2,\dots, {v}_r)q_{1}(Y)+\dots \\
&+ \det({v}_{i,Y}, {v}_1,\dots, {v}_{r-1})q_{r}(Y)=0. 
\end{align*}

Furthermore, because each coefficient of the identity above is bounded by $n^{O_{C,\ep,\mu}(1)}$, there exists a subset $G_{i}''$ of $G''$ such that all $Y \in G_{i}''$ correspond to the same identity, and
\begin{align*}
\P_{Y}(Y \in G_{i}'') &\ge (\P_{Y}(Y \in G'')/2)/(N^{O_{C,\ep}(1)})^r\\
&\ge \rho_b/N^{O_{C,\ep}(1)}.
\end{align*}
In other words, there exist integers $k_1,\dots,k_r$, all bounded by $N^{O_{C,\ep}(1)}$, such that 
$$k q_i(Y) + k_1 q_{1}(Y)+ \dots + k_r q_{r}(Y)=0$$
for all $Y \in G_{i}''$. 

Note that $k$ is independent of the choice of $i$ and $Y$. Finally recall that $q_i = \sum_j b_{ij} y_j=B_i(Y)$, we thus complete the proof of Lemma \ref{lemma:reduction}.
\end{proof}

After proving our inverse result for bilinear forms, by using a decoupling method (similarly to Lemma \ref{lemma:decoupling1}) we can also obtain the following analog of Theorem \ref{theorem:ILO:bilinear} for quadratic forms, which is another highlight of the section.

\begin{theorem}\label{theorem:ILO:quadratic}
Let $C>0$ and $\eps>0$.  Then there exists a constant $A$ such that the following holds for any prime $p\ge n^A$. Assume that $B=(b_{ij})_{1\le i,j\le N}$ is a symmetric array of elements of $\Z/p\Z$ so that 
$$\rho_q:=\sup_{a \in \Z/p\Z} \P_{X}(\sum_{i,j\le n}b_{ij}x_{i}x_{j} =a)\ge N^{-C},$$
where $X=(x_{1},\dots,x_{n})$, and $x_{i}$ are i.i.d. copies of $\xi$.
Then, there exist an integer $k\neq 0, |k|=N^{O_{C,\eps}(1)}$, a set of $r=O(1)$ rows $\row_{i_1}, \dots, \row_{i_r}$ of $B$, and set $I$ of size at least $N-2N^\ep$ such that for each $i\in I$, there exist integers $k_{ii_1},\dots, k_{ii_r}$, all bounded by $N^{O_{C,\ep}(1)}$, such that the following holds.
\begin{equation}\label{eqn:special:q}
\P_{Z}\Big( Z \cdot (k\row_i(B) + \sum_{j=1}^r k_{ii_j} \row_{i_j}(B)) = 0 \mod p \Big)\ge N^{-O_{C,\ep}(1)},
\end{equation}
where $Z=(z_1,\dots,z_N)$ and $z_i$ are i.i.d. copies of $(u-u')(\xi-\xi')$, where $u,u',\xi,\xi$ are independent and $u,u'$ are Bernoulli random variables of parameter $1/2$.
\end{theorem}
A proof of this result is given in Appendix \ref{appendix:quadratic}. 

To conclude the section, when working with Laplacian matrices our random vectors will be adapted along an index set $I_0$ as in Definition~\ref{def:X}. However the above result automatically applies if we restrict to the coefficients indexed from $I_0$. In other words we have
\begin{cor}\label{cor:ILO:quadratic:lap}
Let $C>0$ and $\eps>0$.  Then there exists a constant $A$ such that the following holds for any prime $p\ge n^A$. Assume that  $X$ is adapted to $I_0$ as in Definiotion~\ref{def:X} where $|I_0|\ge \eps N$ and
$$\sup_{a \in \Z/p\Z} \P_{X}(\sum_{i,j\in I_0}b_{ij}x_{i}x_{j} =a)\ge N^{-C}.$$
Then, there exist an integer $k\neq 0, |k|=N^{O_{C,\eps}(1)}$, a set of $r=O(1)$ rows $\row_{i_1}, \dots, \row_{i_r}$ of $B$, and set $I\subset I_0$ of size at least $|I_0|-2N^\ep$ such that for each $i\in I$, there exist integers $k_{ii_1},\dots, k_{ii_r}$, all bounded by $N^{O_{C,\ep}(1)}$, such that the following holds.
\begin{equation}\label{eqn:special:q2}
\P_{Z}\Big( Z_{I_0} \cdot (k\row_i(B) + \sum_{j=1}^r k_{ii_j} \row_{i_j}(B))_{I_0} = 0 \mod p \Big)\ge N^{-O_{C,\ep}(1)},
\end{equation}
where $Z=(z_1,\dots,z_N)$ and $z_i$ are as in Theorem \ref{theorem:ILO:quadratic}. 
\end{cor}
Theorem~\ref{theorem:ILO} will be exploited many times in our next section. Also, Corollary \ref{cor:ILO:quadratic:lap} will play a key role in the proof of Lemma~\ref{lemma:non-zero}.

\section{Treatment for large primes: proof of Propositions~\ref{prop:large:sym} and \ref{prop:large:alt}}\label{section:largeprimes} 

For a sufficiently positive small constant $c$ (to be fixed throughout the section) we let 
 $$\CP_n:=\Big\{p \textrm{ prime},   e^{n^{c}}\le p\le (C_\xi \sqrt{n})^n \Big\}.$$ 
\subsection{Outline for the symmetric and Laplacian models} As usual, $G_N$ stands for the principle minor $M_N$ for symmetric matrices, or the principle minor $L_N^\ast$ (after the neighbor reshuffling process) for laplacian matrices. 
 
Let us start from $G_{N_0}$, where 
$$N_0 = \lfloor c n \rfloor.$$

Consider the event $\CE_{N, non-sing}$ that the matrix $G_{N}$ is non-singular in $\R$ (or $\Z$); we are going to work on the intersection of all these events.

\begin{definition} Let $\CE_{non-sing}$ be the intersection of all $\CE_{N, non-sing}$, where $N_0\le N \le n$. Note that this event is independent of $p$. By Corollary \ref{cor:singularity},  
$$\P(\CE_{non-sing}) =1-\exp(-\Theta(n^{c})).$$ 
\end{definition}

Let $\CB_{N_0}$ be the list of primes $p \in \CP_n$ such that $p| \det(G_{N_0})$, this is the list of bad primes. Notice that as by Hadamard's bound, $|\det(G_N)|\le (C_\xi N)^{N/2}$, and so, counting multiplicities we have
$$|\CB_{N_0}| \le \frac{(N_0/2) \log (C_\xi N)}{n^{c}} \le n.$$

Notice that for $p\in \CB_{N_0}$ and $p\in \CP_n$, the  $\rank(G_{N_0}/p)$ could be as small as $N_0- \frac{(N_0/2) \log (C_\xi N)}{n^{c}}  \ge N_0- n^{1-c/2}$, but we will show that throughout the column exposure process below, this rank will be improved fast and achieves the value at least $n-1$ in the last step. 

For $N_0\le N \le n-1$, we consider the process of adding the $N+1$-th column and $N+1$-th row to form $G_{N+1}$, note that for the Laplacian case, this is Phase 2, obtained via the neighbors reshuffling. 

After each step, let $\CB_{N+1}$ be the collection of primes $p \in \CB_N$ that $G_{N+1}/p$ does not have full rank and the new primes $p$ that did not belong to $\CB_N$ but $\rank(G_{N+1}/p) \le N$. Notice that for these new primes, as $p\notin \CB_N$, we have $\rank(G_N/p) = N$, and so $\rank(G_{N+1}/p)$ must be $N$ in this case. As of now, we will want to make sure that for the newly arising primes the rank are not very small, especially in the very last steps, because otherwise adding the last few rows and columns will not increase the ranks to at least $n-1$ as desired. Another remark here is that, as $0<|\det(G_{N+1})|\le (C_\xi N)^{N/2}$, the number of newly arising primes is bounded by $(N/2)\log (C_\xi N)/n^{c} \le n$. So we trivially have 
$$|\CB_{N+1}| \le |\CB_N| + n \le (N+1)n.$$

Our key result of this section is the following proposition.

\begin{proposition}\label{prop:fullgrowth} Let $C>0$ be a given large constant. There is an event  $\CE=\CE_N$ in characteristic zero with probability at least $1-n^{-C}$ such that under this event, for any $p\in \CB_N$, and for all $N_0 \le N\le n-1$ we have
$$\rank(G_{N+1}/p) =\min \{ \rank(G_N/p)+2, N+1\}.$$
\end{proposition}

It is clear that this proposition would then imply Proposition~\ref{prop:large:sym} because with probability at least $1-O(n^{-C+3})$ (on the intersection of $\cap_{N_0\le N\le n} \CE_{N}$), for all $p\in \CP_n$ we have $\rank(G_{n-1}/p)\ge n-1$ as $\rank(G_n/p)$  becomes full after exposing the last $n$-th row and column for all $p\in \CB_{n-1}$, and for newly arising primes $p$ we have $\rank(G_n/p)= \rank(G_{n-1}/p)=n-1$ as $p \notin \CB_{n-1}$.

Before moving to discuss the technical details, we pause to compare the current method with that of \cite[Section 6]{NgW} for the i.i.d. model. Both methods rely on the ``watch list" argument. However, unlike in the i.i.d. case, here at the starting point $N=N_0$ some prime $p$ in the list $\CB_N$ might have very high multiplicities. Fortunately, each exposure step usually improves the rank by 2 (rather than by 1 as in the i.i.d. case) thanks to Proposition \ref{prop:fullgrowth}. So the rank over $\F_p$ will become almost full very fast. When the rank becomes full we remove $p$ from the watch list $\CB_N$, but this $p$ might reappear later in the process, the difference now is that the corank over $\F_p$ will be at most one as we have seen above. %

In the remaining part  we prove Proposition \ref{prop:fullgrowth}, which is an innovative part of the treatment. We will mainly focus on the Laplacian model (and hence $G_N = L_N^\ast$) as the symmetric case will follow almost automatically. Our proof consists of three steps outlined below
\begin{itemize}
\item Step 1. If 
$$\P(\rank(G_{N+1}/p) =\min \{ \rank(G_N/p)+2, N+1\}) \le 1- n^{-C},$$ 
where the randomness is on the $N+1$-th column and $N+1$-th row (of the reshuffling process), then by Theorem \ref{theorem:ILO:quadratic} and Corollary \ref{cor:ILO:quadratic:lap} there is a ``local" $O(1)$-normal vector of $G_N$ which has {\it partially rich} structure in $\F_p$. Here $C$ is large enough to compensate with the loss of the watching list argument after taking union bound over all primes from $\CB_N, N\le n-1$. We refer the reader to Subsection \ref{sub:partial} for precise statements.
\vskip .1in
\item Step 2. We then use Definition \ref{def:lap'} to pass to the symmetric model, showing that the event that a local normal vector is {\it partially structured} but not {\it fully structured} has probability $n^{-An}$ with large $A$, for which we can take union bound over all $p \in \CP_n$. This step is carried out in Subsection \ref{sub:partial-full}.
\vskip .1in
\item Step 3. It remains to estimate the event that there is a fully structured normal vector for each $p$ in the watch list. We then use the low rank and rich structure to show that this event for $p\ge e^{n^{c}}$ can be passed to characteristic zero, and henceforth to $\Z/p\Z$ for some $ e^{n^{c/2}} \le p\le e^{n^{c}}$, while for the latter setting we have shown in the previous section that this hold with probability $\exp(-\Theta(n^c))$. We will complete this final step in Subsection \ref{sub:full}.
\end{itemize}

From now on, if not specified otherwise, $p$ is a prime from $\CP_n$.

\subsection{Step 1: normal vectors with adapted structures.}\label{sub:partial}  In what follows, for the Laplacian case we recall that the randomness is from Phase 2 of Definition \ref{def:lap}. That is $I_0=I_{N+1}$ is the set of indices from $v_1,\dots, v_{N}$ that is connected to exactly one vertex from the pair $\{v_{N+1},v_{N+2}\}$. Recall \eqref{eqn:concentration}, by Chernoff's bound with probability at least $1- \exp(-\Theta(n^{1-c}))$, for all $N_0 \le N \le n-1$ we have
\begin{equation}\label{eqn:I_0}
 |I_0| \in [N/2 - N^{1-c/2},N/2 +N^{1-c/2}].
 \end{equation}

\begin{lemma}\label{lemma:non-zero} Let $0<\eps<1,C>0$ be given constants. Under $\CE_{non-sing}$, assume that $N\ge N_0$, and 
$$\P\Big(\rank(G_{N+1}/p) =\min \{ \rank(G_N/p)+2, N+1\}\Big) \le 1- n^{-C},$$ 
where the randomness is on the random reshuffling process with respect to $v_{N+1}$ and $v_{N+2}$. Then there exist a constant $C_\ast$ depending on $C,c,\eps$ and there exist a non-zero vector $\Bv$ and a subset $J_0 \subset I_0$ such that 
\begin{itemize}
\item $|J_0| \ge |I_0|-2N^{1-c/2}$;
\vskip .05in
\item all of the entries of $\Bv_{J_0}$ belong to a GAP of size at most $n^{C_\ast}$ and rank at most $C_\ast$,
\vskip .05in
\item $\Bv$ is orthogonal to all but at most $C_\ast$ rows of $G_N$.
\end{itemize}
\end{lemma}

\begin{proof}[Proof of Lemma \ref{lemma:non-zero}] As we are working with primes $p > \exp(n^c)$, if we are on the event $\CE_{non-sing}$ (in particularly $\det(G_N)\neq 0$) then as $|\det(G_N)| \le (C_\xi N)^{N/2}$, we have that 
$$k=\rank(G_N/p) \ge N - N^{1-c/2}.$$
{\bf Case 1.} Assume that $k\le N-1$. Let $G_{J \times J}$ be a submatrix in $G_N$ of full rank, $|J| =k$, where $J=\{i_1,\dots, i_k\}$. We add one more column corresponding to the vertex $v_{N+1}$ and consider the matrix $G_{J \times (J \cup \{v_{N+1}\})}$. Let $H$ be the subspace generated by the rows of this matrix, and let $\Bv'=(v_{i_1},\dots, v_{i_{k+1}})$ be a normal vector of $H$. When we expose the $N+1$-th vector $X_{N+1}$, as we have seen in Section \ref{section:rankevolving},  if $(x_{i_1},\dots, x_{i_k}, x_{N+1}) \notin H$ then $\rank(G_{N+1}/p) =\rank(G_N/p)+2$. So we must have $
(x_{i_1},\dots, x_{i_k}, x_{N+1}) \in H$, and hence 
\begin{equation}\label{eqn:orth}
\Bv' \cdot (x_{i_1},\dots, x_{i_k}, x_{N+1})=0.
\end{equation}
We will restrict the above event to the randomness over $x_{i_j}$ where $i_j\in I_0\cap J$. Recall that as $|J| \ge N - N^{1-c/2}$ and $I_0 \subset [N]$ satisfying \eqref{eqn:I_0}, we have 
$$| I_0\cap J | \ge N/2 -2 N^{1-c/2}.$$ 
By Theorem \ref{theorem:ILO}, the event of \eqref{eqn:orth} has probability smaller than $N^{-C}$ except  all but $N'=N^{1-c/2}$ of the entries of the $v_{i_j}, i_j\in  I_0\cap J$ belong to a GAP structure of size $O(\rho_l(\Bv')^{-1}/\sqrt{N^{1-c/2}})$. To complete the proof, by definition $\Bv'$ is orthogonal to all columns of $G_{J\times [N]}$ (because $G_{J \times J}$ has the same rank as $G_N$), and hence the vector $\Bv$ obtained from $\Bv'$ by appending $N-k-1$ components of value zero is orthogonal to all columns of $G_N$.
\vskip .05in
{\bf Case 2.}  Assume that $k= N$. In this case $G_{J \times J}=G_N$. Let $B=(b_{ij})$ be the inverse matrix $G_N^{-1}$. As we have seen in \eqref{eqn:b_{ij}}, by adding a new row and column (associated to $v_{N+1}$) we see that the rank does increase unless $\sum_{i,j\in I_0}b_{ij}x_{i}x_{j}=0$. Assume that this holds with probability at least $n^{-C}$, then we can apply Corollary \ref{cor:ILO:quadratic:lap}. For each $i\in I$ obtained by this corollary (where we recall that $I$ is of subset size at least $|I_0|-2n^\ep$ of $I_0$) let 
$$\row_i' = k\row_i(B) + \sum_{j=1}^r k_{ii_j} \row_{i_j}(B).$$
Because $B$ has rank $m$, the rows $\row_i(B)$ are linearly independent, and so $\row_i'$ is non-zero for $i\notin \{i_1,\dots, i_r\}$. We fix one such non-zero vector $\Bv'=\row_{i_0}'$, and in this case set $\Bv:=\Bv'$. By definition $\Bv$ is orthogonal to all but only $r+1$ columns  of indices from ${i_0, i_1,\dots,i_r}$. Furthermore by \eqref{eqn:special} we have that 
$$\P_Z\Big( Z \cdot \Bv = 0 \mod p \Big)\ge n^{-O_{C,\ep}(1)},$$
where $Z$ is adapted to $I_0$. As a direct application of Theorem \ref{theorem:ILO}, we then infer that all but $n^\eps$ of the entries of $\Bv_{I_0}$ belongs to a GAP structure of  size $O(n^{O(1)})$ and rank $O(1)$ as claimed.
 \end{proof}

We note from the proof above that $\Bv$ is defined via $B$, and hence $G_{J \times J}$, a submatrix of full rank in $G_N$. Hence it is natural to call $\Bv$ a {\it local} (w.r.t. $J$) almost normal vector of $G_N$. In what follows, for given $C$ and $c,\eps$, the parameter $C_\ast=C_\ast(C,c,\eps)$ is always chosen as in the conclusion of Lemma \ref{lemma:non-zero}. Motivated by this result, it is natural to define the following notion of partially structured vectors.
 
\begin{definition}\label{def:full} Let $J_0\subset [N]$ be given. Let $C_\ast,C_\ast', \eps>0$ be constants (where $C_\ast=C_\ast(C,c,\eps)$). 
\begin{itemize}
\item (partially structured) We say $\Bv \in \F_p^N$ is a $J_0$-structured almost normal vector with respect to $G_N$ (and with respect to the parameter $C_\ast$) if all components of $\Bv_{J_0}$ come from a GAP (over $\Z/p\Z$) with size at most $n^{C_\ast}$ and rank at most $C_\ast$ and $\Bv$ is orthogonal to all but at most $C_\ast$ columns of $G_N$. 
\vskip .1in
\item (fully structured) We say that $\Bv$ is a fully structured almost normal vector (with respect to the parameters $C_\ast',\eps$) if all but $n^{\eps}$ components of $\Bv$ come from a GAP (over $\Z/p\Z$) with size at most $n^{C_\ast'}$ and rank at most $C_\ast'$ and $\Bv$ is orthogonal to all but $C_\ast'$ columns of $G_N$. 
\end{itemize}
\end{definition}

By Lemma~\ref{lemma:non-zero}, if $G_N$ has an almost normal vector that is not locally structured then we would be done with the proof of Proposition~\ref{prop:fullgrowth}. Hence we need to work with the event that $G_N$ has a non-trivial almost normal vector that is $J_0$-structured for some $J_0$. On the other hand, for the step of passing to all primes in $\CP_n$ at once it is desirable to have fully structured vectors instead, which motivated us to introduce Definition \ref{def:full} and add another twist into the plan to treat with locally but not fully structured below. We notice that this problem only occurs in the Laplacian case; for the random symmetric model $I_0$ is already the whole set $[N]$ and in this case we can completely skip these extra treatments.

\subsection{Step 2: partially but not fully structured almost normal vectors of Laplacian matrices}\label{sub:partial-full} 
 Note that in the Laplacian model, after Phase 1 of Definition \ref{def:lap'} the degree sequence $\Bd=(d_1,\dots, d_{n+1})$ is fixed, and also by \eqref{eqn:degreesq} with probability $1- \exp(-\Theta(n))$ 
\begin{equation}\label{eqn:degreesq'}
d_1,\dots, d_{n+1}  \in [(1/2-\eps)n, (1/2+\eps)n].
\end{equation} 
For convenience we define
\begin{definition} Let $\CE_{mixing}$ be the intersection of the events from \eqref{eqn:degreesq'} and from \eqref{eqn:I_0} and  for all $N_0 \le N \le n-1$. Note that this event is on the graph and is independent of $p$ and
$$\P(\CE_{mixing}) \ge 1 - \exp(-\Theta(n^{1-c})).$$
\end{definition}
 We will show the following

\begin{lemma}[structure propagation in perturbed symmetric matrices]\label{lemma:strucnotstruc} Let $0<\eps<1$ and  $A,C>0$ be given constants. Then there exists a constant $C_\ast'$ such that the following holds for all $p\in \CP_n$, and on $\CE_{mixing}$. The probability of the intersection of the events that all principle minors $G_N, N_0 \le N \le n$ have non-zero determinant and that there exist $N$ in the above range and a local almost normal vector of $G_N/p$ that is $J_0$-structured (with respect to $C_\ast$ as in Lemma \ref{lemma:non-zero} for given $C,\eps,c$) for some $J_0 \subset [N]$ of size at least $N/2-2N^{1-c/2}$ but not fully structured with respect to $C_\ast',\eps$ is bounded by  $O(n^{-An})$, where the implied constants are allowed to depend on the given constants. 
\end{lemma}

Note again that it would be more natural to have the probability bound of the form $O(N^{-AN})$ for each $N$, but by changing $A$ we can replace the bounds by $O(n^{-An})$ for convenience. We will prove Lemma \ref{lemma:strucnotstruc} by passing to random symmetric matrices of given diagonal entries, via Definition \ref{def:lap'} of the Laplacians. We can do this thanks to the key lemma below which roughly says that the rare event is very rare. 

\begin{lemma}[partial structure in perturbed random symmetric matrices]\label{lemma:strucnotstruc:sym}  Let $0<\eps<1$ and  $A,C>0$ be given constants.  Then there exists a constant $C_\ast'$ such that the following holds. Let $p\in \CP_n$. Let $\Bd=(d_1,\dots, d_{n+1})$ be an degree sequence so that 
$$\P\big(G(n+1,1/2) \in G_{\Bd}\big)\ge n^{-4n}.$$ 
Assume that $M_n$ is the random symmetric matrix as in Theorem \ref{theorem:cyclic:sym}. Then on the intersection of the event $\CE_{non-sing}$ in $\Z$ that all principle minors $M_N-D_N,  N_0\le N\le n$ have non-zero determinant (where $D_N=\diag(d_1,\dots, d_N)$ is the diagonal matrix with entries $d_i$), the event that there exists $N$ in the above range and a local almost normal vector $\Bv$ of $(M_N-D_N)/p$ that is $J_0$-structured  (with respect to $C_\ast$) for some $J_0 \subset [N]$ satisfying $|J_0| \ge N/2-2N^{1-c/2}$ but $\Bv$ is not fully structured with respect to  $C_\ast' ,\eps$  has probability bounded by  $O(n^{-An})$.
\end{lemma}

We then deduce the following version for random matrices of given degree sequence.

\begin{corollary}\label{cor:strucnotstruc}  Let $0<\eps<1$ and $C>0$ be given. There there exists a constant $C_\ast'$ such that the following holds. Let $\Bd=(d_1,\dots, d_{n+1})$ be an degree sequence with $(1/2-c)n \le d_i \le (1/2+c)n$  so that 
$$\P(G(n+1,1/2) \in G_{\Bd})\ge n^{-4n}.$$ 
Let $G_\Bd$ be a random graph in $\CG_{\Bd}$ defined in Phase 2 of Definition \ref{def:lap'},
$$\P(G_{\Bd}=G) := \frac{\P(G(n+1,1/2) = G)}{\P(G(n+1,1/2) \in \CG_d)}, \mbox{ for each $G$ of degree sequence $\Bd$}.$$ 
Then the conclusion of Lemma \ref{lemma:strucnotstruc:sym} holds for the matrix $L_{n, \Bd}$ obtained from $M_{G_\Bd}-\diag(\Bd)$ by removing the last row and column.
\end{corollary}

\begin{proof}[Proof of Corollary \ref{cor:strucnotstruc}] Let $M_{n+1}$ be the adjacency matrix of $G(n+1,1/2)$ and let $\CC_{\Bd}$ be the event that $G(n+1,1/2) \in \CG_\Bd$ (that is $\{M_{n+1} \in \CC_\Bd \}= \{G(n+1,1/2) \in \CG_\Bd\}$). Then by definition $\P(M_{n+1} \in \CC_\Bd) \ge n^{-4n}$ and for each $G\in \CG_\Bd$
$$\P(M_{G_\Bd}=M_G) =  \frac{\P(M_{n+1} = M_G)}{\P(M_{n+1} \in \CC_\Bd)}.$$
Let $\CF$ be the set of square matrices (of size $n$) where every $J_0$-structured normal vectors is fully structured, and let $\CE$ be the non-singularity event $\CE_{non-sing}$ considered in Lemma~\ref{lemma:strucnotstruc:sym}. We have (where $L_n(M_{n+1})$ is the principle minor of $M_{n+1} - \diag(\Bd)$)
\begin{align*}
\P(L_{n, \Bd} \in \CE \cap \bar{\CF}) &= \frac{\P(L_n(M_{n+1}) \in \CE \cap \bar{\CF} \wedge M_{n+1} \in \CC_\Bd)}{\P(M_{n+1}\in \CC_\Bd)}\le \frac{\P(L_n(M_{n+1}) \in \CE \cap \bar{\CF} \wedge M_{n+1} \in \CC_\Bd)}{n^{-4n}}\\  
&= O(n^{-(A-4)n}),
\end{align*}
where in the last estimate we applied Lemma \ref{lemma:strucnotstruc:sym}.
\end{proof}

\begin{proof}[Proof of Lemma \ref{lemma:strucnotstruc}]  On  $\CE_{mixing}$ we have (recalling that $G_n=L_n$)
\begin{align*}\P(L_n \in \CE \cap \bar{\CF} \cap \CE_{mixing}) & \le \sum_{\Bd; \P(G(n+1,1/2) \in \CG_{\Bd})\ge n^{-4n}}\P(M_{n+1} \in \CE \cap \bar{\CF} | M_{n+1} \in \CC_{\Bd})\P(M_{n+1} \in \CC_{\Bd})\\
& =\sum_{\Bd; \P(G(n+1,1/2) \in \CG_{\Bd})\ge n^{-4n}} \P(L_{n,\Bd}\in \CE \cap \bar{\CF})  \P(M_{n+1} \in \CC_{\Bd}) \\
&=O( n^{-(A-4)n}).
\end{align*}
\end{proof}
What remains is to justify the symmetric matrix model.

\begin{proof}[Proof of Lemma \ref{lemma:strucnotstruc:sym}] Our proof is somewhat similar to those of Proposition \ref{prop:structure:subexp:lap} and Lemma \ref{lemma:sparse:2}. It suffices to consider for a fixed $N$ in the range $N_0 = \lfloor c n \rfloor \le N \le n$. Assume that there is a local almost orthogonal vector $\Bv'$ (from both cases of the proof of Lemma \ref{lemma:non-zero}, not including the $N -k$ appended zero components) that is $J_0$-structured for some $J_0$ but not fully structured. We will show that the probability of this event is as small as expected via two stages. 
In what follows we will use the notation from the proof of Lemma  \ref{lemma:non-zero}.

{\bf Stage 1.} Let $\Bv'=(\Bv_0,\Bv_1)$ be the decomposition into the $J_0$-structured part $\Bv_0$ and the remaining part $\Bv_1$, where $J_0=\suppi(\Bv_0)$. 
We will  fix $J_0$ (there are crudely at most $2^{N}$ such $J_0$), where $|J_0| \ge N/2 -2N^{1-c/2}$ on $\CE_{mixing}$, and $|J_1|= |\suppi(\Bv_1)| =|J\bs J_0| +1 \le N/2+2N^{1-c/2}$ (see the left picture from Figure \ref{figure:local}.)  %

We will argue that $\Bv_1$ can be determined via $\Bv_0$ and some part of the matrix $G_N$. We will be focusing on Case 1 of the proof of  Lemma \ref{lemma:non-zero}, the other case is similar.  Using notation from that proof, we have $G_{J \times (J\cup \{N+1\})}\Bv'=0$, and so $M_0\Bv_0 + M_1\Bv_1 =0$ where $M_0 = G_{J \times J_0}$ and $M_1=G_{J \times J_1}$. As the matrix $G_{J \times (J\cup \{N+1\})}$ has full rank, the matrix $G_{J \times J_1}$ has rank at least $|J_1|-1$. Without loss of generality we assume that $G_{J \times J_1}$ has rank $|J_1|$, and hence it has a square submatrix $M_1'$ of size $|J_1|$ which is non-singular. 
\begin{claim}\label{claim:AJ_1} With probability at least $1-O(\exp(-n^{1+\eps}))$, $M_1'$ can take the form $M_{B \times J_1}$ where 
$$|B \cap J_1|\le 5N^{1-c/2}.$$
\end{claim} 
\begin{proof} We first choose $B'=J_0 \cup B''$ for any $B''\subset J_1$ so that $|B'|=J_1$. Then $|B'\cap J_1| \le 4N^{1-c/2}$. By the proof of Lemma \ref{lemma:quadratic}, we see that the matrix $M_{B \times J_1}$ already has rank at least $|J_1| - n^{1/2+\eps}$ with probability at least $O(\exp(-n^{1+2\eps}))$. Hence by interchanging with at most $ n^{1/2+\eps}$ row vectors of $M_{(J\bs B')\times J_1}$, we can obtain $B$ for which $M_{B \times J_1}$ has full rank, and that $|B \cap J_1| \le  |B'\cap J_1| + n^{1/2+\eps} \le 5 N^{1-c/2}$. 
\end{proof}
Using the matrix $M_1'=M_{B\times J_1}$ above, we will extract from $M_{ J \times (J\cup \{v\}) }\Bv'=0$ the portion $M_0' \Bv_0 + M_1' \Bv_1=0$, where $M_0' =M_{B \times J_0}$. After fixing $M_0'$ and $M_1'$ we have
 $$\Bv_1 = (M_1')^{-1} M_0 \Bv_0.$$ 
 In summary, by conditioning on $M_0'$ and $M_1'$, and on a realization of $\Bv_0$, the vector $\Bv_1$ is determined.  Note that by fixing $M_0'$ and $M_1'$, we have fixed a submatrix of size $J_1 \times [N]$ in $G_N$. 
 
 Set $A_0:=C_\ast$. There are $N^{C_\ast N/2}p^{O(1)}=N^{A_0 N/2}e^{O(n^c)}$ ways to choose the structured vector $\Bv_0$, and $O(2^N \times 2^N)=O(4^N)$ ways to choose $J_0$ and the row index for $M_1'$. Thus in total we have a collection $\CC_0$ of at most $N^{A_0 N}$ ways to choose the vectors $\Bv_0,\Bv_1$. 

\begin{figure}%

\centering
\begin{tikzpicture}
\draw (0,0) -- (4,0) -- (4,4) -- (0,4) -- (0,0);
\draw (2,0) -- (2,4);
\draw (0,3) -- (4,3);
\draw (0,1) -- (2,1)  node[above  left = .2in]{$M_0'$};
\draw (2,1) -- (4,1)  node[above  left = .2in]{$M_1'$};
\draw (0,4) -- (2,4)  node[above  left =.2in]{$\Bv_0$};
\draw (2,4) -- (4,4)  node[above  left =.2in]{$\Bv_1$};
\draw [line width=2pt] (2,4) -- (3,4)  node[above  left =.1in]{$\Bv_1'$};
\draw [dashed] (0,4) -- (4,0);
\draw (3,0) -- (3,4);
\end{tikzpicture}
\hfil
\begin{tikzpicture}
\draw (0,0) -- (4,0) -- (4,4) -- (0,4) -- (0,0);

\draw (2,0) -- (2,4);
\draw (0,4) -- (2,4)  node[above  left =.2in]{$\Bv_0$};
\draw (2,4) -- (2.5,4)  node[above=.2in]{$\Bv_1'$};
\draw (2,4) -- (2.4,4)  node[above=.05in]{$\Bv_1''$};
\draw [line width=2pt] (2,4) -- (2.5,4);

\draw  [dashed] (0,4) -- (4,0);
\draw (3,0) -- (3,4);
\draw (2.5,0)--(2.5,4);
\draw (0,1.5) -- (4,1.5);
\draw (0,1.5) -- (1,1.5)  node[above=.05in]{$M_0''$};
\draw (2,1.5) -- (2.5,1.5)  node[above=.05in]{$M_1''$};
\draw (3,1.5) -- (3.5,1.5)  node[above=.05in]{$M_0''$};
\draw (0,2.5) -- (4,2.5);
\end{tikzpicture}
\caption{Finding non-structured subvectors.}
\label{figure:local}
\end{figure}
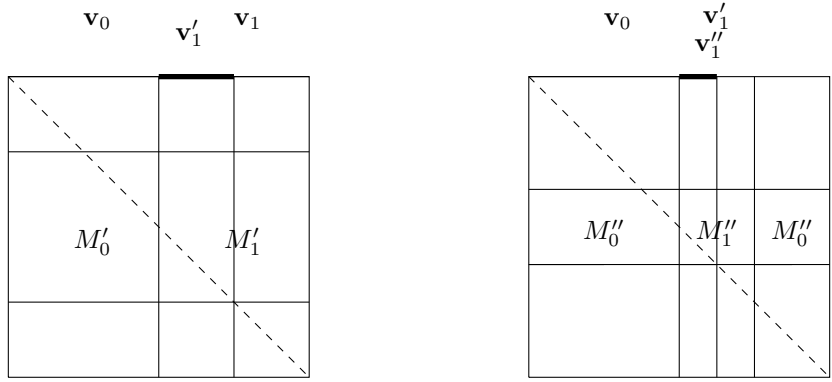

Next we show that for a fixed $\Bv_1$ in the way above, by exploiting the remaining randomness and the remaining equations from $M_0\Bv_0 + M_1\Bv_1 =0$, the most part of $\Bv_1$ must be structured with extremely high probability. To do this, let $(J_1\bs B)^+$ and $(J_1\bs B)^-$ be any two parts of $J_1\bs B$ of size almost equal. If the part $\Bv_1$ restricted to $J_1\bs B$ has the property that 
\begin{equation}\label{eqn:J1A}
\rho_l((\Bv_1)_{(J_1\bs B)^+}) \le n^{-K_0} \mbox{ or } \rho_l((\Bv_1)_{(J_1\bs B)^-}) \le n^{-K_0}
\end{equation}
for some large constant $K_0$ chosen depending on $A_0$ (i.e. on $C_\ast$) and $A$. Without loss of generality assume that $\rho_l((\Bv_1)_{(J_1\bs B)^+}) \le n^{-K_0}$. Then by using the i.i.d. decomposition as in \eqref{eqn:i.i.d.decomposition:nL} and  \eqref{eqn:i.i.d.decomposition:L} we can bound the probability that $M_{(J\cup \{N+1\}) \times J}\Bv'=0$ by
\begin{align*}
\P(\exists \Bv' \in \CC_0, M_{J \times (J\cup \{v\})}\Bv'=0) & \le \sup_{\Ba} \P(\exists \Bv'\in \CC_0, M_{(J\bs B) \times (J_1 \bs B)}\Bv'=\Ba)\\
& \le (1/n^{K_0})^{|(J_1\bs B)^-|} N^{A_0 N} 4^n   \le 1/n^{NK_0/4} \le 1/n^{A n}.
\end{align*}

Hence it remains to assume that $\rho_l((\Bv_1)_{(J_1\bs B)^+})  \ge n^{-K_0}$ and $\rho_l((\Bv_1)_{(J_1\bs B)^-})  \ge n^{-K_0}$. Under these assumptions, however, Theorem \ref{theorem:ILO} implies that $(\Bv_1)_{J_1\bs B}$ is structured in the sense that all but $n^\eps$ element of $(\Bv_1)_{J_1\bs B}$ belongs to a GAP of size bounded by $n^{K_0}$. Hence we have just shown that all but at most $n^\eps$ exceptional entries of the vector $\Bv'|_{ J_0 \cup (J_1\bs B)}$ belong to a GAP of size $n^{O(1)}$ and rank $O(1)$, where the implied constants depend on $\eps, K_0$. So crudely the collection $\CC_1$ of such structured vectors is now bounded by $n^{A_1 n}$ for some sufficiently constant depending on $A_0,A, \eps$.
  
{\bf Stage 2.} Because the contribution of $n^\eps$ exceptional entries can be easily absorbed into the total number of unstructured entries, let us focus only on the vector $\Bv_1|_{B \cap J_1}$, where we will show that with extremely high probability all but $n^{O(\eps)}$ entries of this vector are structured. We have learned that $N':=|B \cap J_1| \le 4 N^{1-c/2}$. If $N' \le n^\eps$ then we are done. Assume otherwise, we will then apply the following variant of Lemma \ref{lemma:quadratic} and Claim \ref{claim:AJ_1} above.
\begin{claim}\label{claim:iteration} Assume that $I$ is an index set with $|I|=N'$ where $n^\eps \le N'\le 4 N^{1-c/2}$. Then with probability at least $1-\exp(- \Theta(N \times {N'}^{\eps}))$ the matrix $G_{([N]\bs I) \times I}$ has a square submatrix of size $|I|$ of rank at least $|I|-{N'}^{1/2+\eps}$.
\end{claim}
\begin{proof}
We decompose the matrix $G_{([N]\bs I) \times I}$ into $k=\lfloor N/N' \rfloor$ disjoint square matrices of type $M_{B_1 \times I}, \dots, M_{B_k \times I}$ where $B_1,\dots, B_k$ are arbitrary but disjoint, each has size $N'$ in $[N]\bs I$. For each such matrix, the probability it has rank at most $|I|-{N'}^{1/2+\eps}$ is bounded by $\exp(-\Theta({N'}^{1+2\eps}))$ by the argument of Lemma \ref{lemma:quadratic} (with $N'$ in place of $n$). Because these matrices are independent, the probability that at least one of these matrices has rank at least $|I|-{N'}^{1/2+\eps}$ is bounded below by 
$$1- \exp(-\Theta( k {N'}^{1+2\eps})) \ge 1-\exp(- \Theta(N \times {N'}^{\eps})).$$
\end{proof}
In our next step we work with $\Bv_1|_{B \cap J_1}$ as with $\Bv_1$ from Stage 1. By using the argument in Stage 1 (relying on $M_0\Bv_0 + M_1\Bv_1 =0$, but now with a fixed structured vector $\Bv'|_{ J_0 \cup (J_1\bs B)}$ from $\CC_1$) to passing to a potential unstructured vector $\Bv_1''$ (see the right-side illustration in Figure \ref{figure:local}) whose support is a subset of size at most ${N'}^{1/2+\eps}$ in $J_1$, where $\rho_l(\Bv_1'') \le n^{-K_2}$ for some $K_2$ chosen sufficiently large compared to $A_1,A,\eps$. 

One then iterate Claim \ref{claim:iteration} until we get a subvector supported on an index set of size $O(n^\eps)$. It is clear that this process terminates after $O(1)$ steps because the support size of the potential unstructured vectors decreases from $N'$ to ${N'}^{1/2+\eps}$ each time. 

By gathering the structures together, we obtain a unified GAP (over $\Z/p\Z$), which might have large size and rank, but still of order $O(n^{O(1)})$ and $O(1)$ respectively. Let $C_\ast'$ be the maximum of these constants (which depend on the choices of $K_i$, and hence of $A$ and $\eps$), we are done with the proof.
\end{proof}

To conclude this subsection, by combining Lemma \ref{lemma:non-zero} Lemma \ref{lemma:strucnotstruc} we obtain the following

\begin{lemma}\label{lemma:fullstruct} Let $C, \eps>0$ be given, where $C$ can be large and $\eps$ can be small. There exists a constant $C_\ast'$, and there exists an event $\CE_{partial-full}$ of probability $1-n^{-An}$ such that on the event $\CE_{mixing}\wedge \CE_{non-sing}\wedge \CE_{partial-full} $, assume that $N\ge c n$, and $\P(\rank(G_{N+1}/p) =\min \{ \rank(G_N/p)+2, N+1\}) \le 1- n^{-C}$ for some $p \in \CP_n$ (where the randomness is on the $N+1$-th column and $N+1$-th row of the reshuffling process if in the Laplacian case.) Then there exists a non-zero vector $\Bv$ such that all but $n^{\eps}$ of its entries belong to a GAP in $\Z/p\Z$ of size at most $n^{C_\ast'}$ and rank at most $C_\ast'$ such that $\Bv$ is orthogonal to all but at mot $C_\ast'$ rows of $G_N$.
\end{lemma}

We end this section with a result for some restricting range of $p$, that will be useful in Step 3. This result is an analog of Proposition \ref{prop:structure:subexp:sym} and \ref{prop:structure:subexp:lap} for GAP structures.

\begin{proposition}\label{prop:GAP:lap} Let $\eps, C_\ast'$ be given. Assume that $N\ge c n$ and $e^{n^{c/2}} \le p$.
The probability that there exists a non-sparse vector $\Bv$ whose all but $n^\eps$ entries belong to a GAP of size at most $n^{C_\ast'}$ and rank at most $C_\ast'$ and  $\Bv$ is orthogonal to all but at mot $C_\ast'$ rows of $G_N$ is bounded by $n^{-\Theta(n)}$.
\end{proposition}
We refer the reader to Appendix \ref{section:GAP:lap} for a proof of this result. We also refer the reader to \cite[Lemma 7.6]{NgW} for a similar statement for non-symmetric matrices.

\subsection{Step 3: passing to characteristic zero and back}\label{sub:full} In what follows we let $W_{m}/p$ be the subspace generated by the columns (rows) of $G_m$ over $p$. We first need the following

\begin{lemma}[Lifting sparse normal vectors from $\Z/p\Z$ to $\R$] \cite[Lemma 6.5]{NgW}\label{lemma:passing:0} Let $T>0$ be fixed \footnote{In our application, as the matrix entries in all models under consideration are bounded by $O(1)$, we can let $T=1$, say, and assume $n$ to be sufficiently large.}. Let $k,l,n$ be positive integers with $l\le k$, and $M$ a $l\times k$ matrix with integer entries $|M_{ij}|\leq n^T$.
If $p$ is a prime larger than $e^{(k \log k)/2+k T \log n}$, then the rank of $M$ over $\Q$ is equal to the rank of $M/p$ over $\Z/p\Z$.  This has the following corollaries. 
\begin{enumerate}
\item \label{i:lindep} If $Z_1,\dots, Z_k \in \Z^{l}$ are vectors with entries $|Z_{ij}|\leq n^T$, and $Z_1/p,\dots, Z_k/p$ are linearly dependent in $(\Z/p\Z)^{l}$, then $Z_1,\dots, Z_k$ are also linearly dependent in $\Z^{l}$. 
\item \label{i:sparse} Let %
$Z_1,\dots, Z_{i}\in \Z^m$ be vectors with entries $|Z_{ij}|\leq n^T$ . If there is a non-zero vector $\Bw \in (\Z/p\Z)^m$ with at most $k$ non-zero entries that is normal to $Z_1/p,\dots, Z_{l}/p$, then there is a non-zero vector $\Bw' \in \Z^m$ with at most $k$ non-zero entries and normal to $Z_1,\dots, Z_{l}$.
\item \label{i:kernel} The kernel of the map $M: \Z^k \ra \Z^l$ surjects onto the kernel of the map
$M: (\Z/p\Z)^k \ra (\Z/p\Z)^l$.

\end{enumerate}
\end{lemma}

 We say a submodule of $\Z^n$ is \emph{admissible} if it is generated by vectors with coordinates at most $n^{T}$ in absolute value. Recall that a vector is structured (or fully structured) if it satisfies the hypothesis and conclusion of Lemma \ref{lemma:fullstruct}.

\begin{lemma}[Lifting and reducing structured vectors]\label{lemma:passing} Let $\eps$ be given sufficiently small (given $c,C,A, T$). Let $M_0$ be an admissible submodule of $\Z^n$, and $p$ be a prime $\geq e^{n^{c}}$.
Then $M_0$ has a structured almost normal vector in $\Z$ (that is all but $n'=n^\eps$ entries of it belongs to a symmetric GAP with integral generators which has size at most $n^{C_\ast'}$ and rank at most $C_\ast'$, where $\C_\ast'$ is allowed to depend on $\eps$ as in the conclusion of Lemma \ref{lemma:fullstruct}) if and only if $M_0/p$ has a structured almost normal vector (of the same parameters).
\end{lemma}
We note that this result is similar to \cite[Lemma 7.7]{NgW}, where almost normal vectors were replaced by normal vectors. A proof of this result is included in Appendix \ref{section:passing} for convenience.

We now complete our main result of the section.

\begin{proof}[Proof of Proposition \ref{prop:fullgrowth} for the Laplacian model (and hence also for the symmetric model)] Here we assume $N\ge c n$. Let $W_N$ be the submodule in $\Z$ generated by the columns of $G_N$. We let $\mathcal{S}'$ be the set of  submodules $H_N$ of $\Z^N$ such that for all primes $p>e^{n^c}$, the vector space $H_N/p$ has no structured almost normal vector $\Bw$. First, we will bound $\P(W_{N}\not\in \mathcal{S}')$.  By Lemma~\ref{lemma:passing}, for $p\geq e^{n^{c}}$, if $W_{N}/p$ has a structured almost normal vector, then $W_{N}$ has a structured almost normal vector, and then
$W_{N}/p'$ has a structured almost normal vector for every prime $p'$ with $e^{n^{c/2}} \leq p' < e^{n^c}$. So it suffices to bound the condition that $W_{N}/p$ has a structured almost normal vector for $p$ is a prime  $e^{n^{c/2}} \le p< e^{n^{c}}.$ 

We will include in our upper bound the probability that $W_{N}/p$ has a non-zero almost normal vector $\Bw$ with $|\supp(\Bw)|\leq c N$ for some prime $p<e^{n^{c}}$, which is at most 
$e^{-\Theta(n)}$ by Lemmas~\ref{lemma:sparse:1}.  Then, otherwise, by Proposition \ref{prop:GAP:lap}, it is of probability at most 
$n^{-\Theta(n)}$ that, for some prime $e^{n^{c/2}} \le p \le e^{n^{c}}$, the space $W_{N}/p$ has a non-sparse structured almost normal vector $\Bw$.  We conclude that, unconditional on all overwhelming events such as  $\CE_{mixing}, \CE_{non-sing}, \CE_{partial-full}$ we have 
$$\P(W_{N}\in \mathcal{S}')\geq 1-e^{-\Theta(n^c)}.$$
Let $H_N\in \mathcal{S}'$, then $H_N/p$ has no structured almost normal vectors for any $p\ge e^{n^c}$. Then by Lemma \ref{lemma:fullstruct} we must have
$$\P(\rank(G_{N+1}/p) =\min \{ \rank(G_N/p)+2, N+1\}) \ge 1- n^{-C},$$ 
completing the proof.
\end{proof}

\subsection{Outline for the skew-symmetric model} As this case is fairly simple compared to the symmetric and especially the Laplacian case, we just sketch the proof. Our initial setting is similar to the case of symmetric and Laplacian matrices, where we rely on Corollary \ref{cor:singularity}  for $A_{2n}$ and the following analog of Proposition \ref{prop:fullgrowth}.

\begin{proposition}\label{prop:fullgrowth:alt} There is an event  $\CE_{alt}$ in characteristic zero with probability $1-e^{-O(n^c)}$ such that under this event, for any primes $p\ge e^{n^c}$ in the watch list $\CB_{2N}$, and for all $c n \le 2N\le 2n-2$ we have
$$\rank(A_{2N+1}/p) =\rank(A_{2N}/p)+2$$
and
$$\rank(A_{2N+2}/p) =\rank(A_{2N+1}/p)+2.$$
\end{proposition}
Using this we can conclude Proposition \ref{prop:large:alt} as follows.

\begin{proof}[Proof of Proposition \ref{prop:large:alt}] We will condition on the event $\CE_{alt}$ from Proposition \ref{prop:fullgrowth:alt}. We first show that 
$$\rank(A_{2n}/p) \ge 2n-2, \forall p \ge e^{n^c}.$$ 
Indeed, if $\rank(A_{2n}/p) \le 2n-4$, then $p \in \CB_{2n-2}$. Let $N_{\min} \le n-1$ be a smallest index where for all $N\in [N_{\min}, n-1]$ we have $p \in \CB_{2N}$. Because of Proposition \ref{prop:fullgrowth:alt}, we certainly have $n/2\le N_{\min}$. Now if $N_{\min} < n-1$, then we have learned that $p\notin \CB_{2 N_{\min}-1}$, and hence $A_{2 N_{\min}-1}/p$ has full rank but $A_{2 N_{\min}}/p$ has rank $2 N_{\min}-2$. In this case,  by Proposition \ref{prop:fullgrowth:alt} $A_{2 N_{\min}+2}$ has rank $2 N_{\min}+2$, so $p \notin \CB_{2(N_{\min}+1)}$, a contradiction. So we must have $N_{\min}=n-1$. This implies that $A_{2(n-2)}/p$ has full rank and $A_{2(n-1)}/p$ has rank $2n-4$. Another application of Proposition \ref{prop:fullgrowth:alt} then yields that $A_{2n}/p$ has rank $2n-2$, another contradiction.

Now we show that 
$$\rank(A_{2n+1}/p) \ge 2n, \forall p \ge e^{n^c}.$$
If $A_{2n}/p$ has rank $2n$ then there is nothing to prove. If $A_{2n}/p$ has rank $2n-2$, then $p \in \CB_{2n}$, and hence by Proposition \ref{prop:fullgrowth:alt} we have $A_{2n+1}/p$ has rank $2n$.
\end{proof}

In what follows we sketch the idea to prove Proposition \ref{prop:fullgrowth:alt}. Our method follows the arguments from the previous section (especially for the symmetric matrices case) without the need of Subsection \ref{sub:partial-full}. Indeed, assume that  
$$\P\Big(\rank(A_{N+1}/p) =\min\big\{\rank(A_N/p)+2, N+( N\mod 2)\big\}\Big) \le 1- n^{-C},$$ 
where $C$ is chosen sufficiently large (so after taking union bound over $O(n^2)$ primes from the watch lists we still have the above event with overwhelming probability.) Then by Lemma \ref{lemma:linear:lap} (applied to the skew-symmetric matrix case) and by Theorem \ref{theorem:ILO} there is an almost normal vector of $A_N$ (the vectors of $a_{ij}$ from \eqref{eqn:x,a}) which is fully structured. (We remark that here there is no need  to use Theorem \ref{theorem:ILO:quadratic} as in the Laplacian case because it suffices to consider $X_{N+1} \notin W_N$ (where $X_{N+1}\in \F_p^N$ is the last column of $A_{N+1}/p$ without the last entry, and for this Theorem \ref{theorem:ILO} suffices). Finally we then use the methods from Subsection \ref{sub:full} to show that these events over different $p$ can be simultaneously treated by working only on one $p$ in the range $e^{n^{c/2}} \le p \le e^{n^c}$, and to this end we use Proposition \ref{prop:GAP:lap}.

\section{Further remarks and  directions}\label{section:further}

Our treatment for Laplacian graphs can be extended to random Erd\H{o}s-R\'enyi graphs of other parameters. 

\begin{theorem}\label{theorem:lap:gen} Let $q>0$ be fixed. The results of Theorem \ref{theorem:cyclic:L_G} and Theorem \ref{theorem:prodcyc:lap} also hold for the Laplacian of $G(n,q)$. 
\end{theorem}
We sketch the main ideas below, omitting the details. 
\begin{itemize}
\item First, the neighbor reshuffling process remains the same where we reshuffle the pairs of edge and non-edge in Definition \ref{def:lap} by a fair coin flip (hence the randomness created by reshuffling are i.i.d. Bernoulli taking values $0,1$ with probability 1/2). Here by Chernoff we just need to modify \eqref{eqn:degreesq} to $\P(\wedge_{i=1}^{n+1} qn - t\sqrt{n} \le d_i \le qn + t\sqrt{n}) \ge 1-n \exp(-\Theta_q(t^2))$ and \eqref{eqn:concentration} to $\P\big(\wedge_{k \ge c n} |I_k| \in (2q(1-q)k - t\sqrt{n}, 2q(1-q)k+ t\sqrt{n}\big) \ge 1- n\exp(-\Theta_q(t^2))$. We also need to modify \eqref{eqn:I_0} and \eqref{eqn:degreesq'} accordingly.
\vskip .1in
\item Next, for our treatment of small primes, by \cite[Theorem 1.1]{Wood2017}, Theorem~\ref{theorem:W:lap} extends to the Laplacian model of $G(n,q)$. 
\vskip .1in
\item For moderate primes, to establish Proposition~\ref{prop:moderate:lap} we will need Theorem~\ref{theorem:rankstatistics:lap}. For this, the main results of Sections \ref{section:lemmas} and \ref{section:normalvector}, Lemma~\ref{lemma:quadratic}, Proposition~\ref{prop:structure:subexp:lap}, Lemma~\ref{lemma:sparse:2}, and Proposition~\ref{prop:structure:subexp:lap'} are all valid in the more general setting of $G(n,q)$ with the modified reshuffling process. Here we note that Theorem~\ref{theorem:LO} and Theorem~\ref{theorem:Halasz} work under fairly general assumption of randomness (such as \eqref{eqn:alpha}).
\vskip .1in
\item Lastly, for large primes, to establish Proposition~\ref{prop:large:lap} we use the ``watch list" method of Section \ref{section:largeprimes}. Here, again we rely on the inverse results of Theorem~\ref{theorem:ILO} and \ref{theorem:ILO:quadratic} of Section \ref{section:bilinear}, as well as Lemma~\ref{lemma:strucnotstruc} by passing to the matrix model of Definition \ref{def:lap'}.  
\end{itemize} 
It is an interesting problem to extend our results to random sparse graphs and matrices. This is not impossible, especially when $q \ge n^{-1+\eps}$, but the proofs are expected to be extremely technical. It is also interesting to generalize our results to Laplacian of general random matrices. (In this paper we relied on the neighbor reshuffling process, which is rather specific to $\{0,1\}$ matrices.) Furthermore, it would be interesting to extend our results to adjacency matrices of random graphs of given degrees $d_i$ (say all of order $n$, for instance the model $\CG_\Bd$ from Def. \ref{def:lap'}). Some of our results, such as Corollary \ref{cor:strucnotstruc}, are applicable to this setting, but other ingredients such as Theorem~\ref{theorem:W:lap} \footnote{Heuristically one can hope to use the method of \cite{Me} for this, but the implementation seems highly non-trivial.} and Theorem~\ref{theorem:rankstatistics:lap} are largely missing.

Finally, there are other interesting global statistics that we do not access in this paper, such as the probability that the group has square-free order (or cube-free order in the even dimensional skew-symmetric case), it would be interesting to address these issues.

\section*{Acknowledgements} The first author is partially supported by NSF CAREER grant DMS-1752345. The second author is partially supported by a Packard Fellowship for Science and Engineering, NSF CAREER grant DMS-2052036, and NSF Waterman award DMS-2140043.

\appendix

\section{Proof of Theorem \ref{theorem:ILO:quadratic}}\label{appendix:quadratic}
Let $f(X)=\sum_{1\le i\le j\le N} b_{ij} x_i x_j$. We write 
\begin{align*}
|\P( f(X) = a) -\frac{1}{p}|&=\frac{1}{p} |\sum_{t \neq 0, \in \F_p} \E e_p(f(X)t) e_p(-at)| \le \frac{1}{p}\sum_{t \neq 0, \in \F_p} |\E e_p(f(X)t)|.
\end{align*}
For now we let $U \subset [N]$ and let $U^c=[N]\backslash U$. This set $U$ will be randomized at the end, but now it is deterministic. By using Cauchy Schwarz as in the proof of Lemma \ref{lemma:decoupling1}
\begin{align*}
 (\frac{1}{p}\sum_{t \neq 0, \in \F_p} |\E e_p(f(X)t)|)^4 &\le \frac{1}{p} \sum_{t \in \F_p, t\neq 0} |\E e_p(f(X)t)|^4\le \frac{1}{p} \sum_{t \in \F_p, t\neq 0} |\E_{X_{U}} \E_{X_{U^c}, X_{U^c}'}e_p( (f(X_{U},X_{U^c}) -f(X_{U}, X_{U^c}'))t)|^2\\
&= \frac{1}{p} \sum_{ t \in \F_p, t\neq 0}  \E_{X_{U^c}, X_{U^c}'} \E_{X_{U}, X_{U}'}e_p((f(X_{U},X_{U^c}) -f(X_{U}, X_{U^c}') - f(X_{U}', X_{U^c})+f(X_{U}', X_{U^c}'))t)\\
&= \frac{1}{p} \sum_{t \in \F_p, t\neq 0}  \E_{Y}e_p(\sum_{i\in U, j\in U^c} b_{ij}y_i y_j t)=\P(Y_{U} B Y_{U^c} = 0) - \frac{1}{p}=\P(Y_1 B_U Y_2 = 0) - \frac{1}{p},
\end{align*}
where $B_U$ is the matrix $B_U(ij) = b_{ij}$ iff $i\in U, j\in U^c$ and $B_U(ij)=0$ otherwise, and $Y_1,Y_2$ are independent with entries as i.i.d. copies of $\xi-\xi'$. 

If $\P( f(X) = a) \ge n^{-C}$ and if $p\ge n^{A}$ with sufficiently large $A$ (given $C$) then $|\P( f(X) = a) -1/p|\ge N^{-C}/2$, and hence $\P(Y_1 B_U Y_2 = 0)  \ge n^{-4C}/16 +1/p \ge N^{-4C}/16$. We can then apply Lemma \ref{theorem:ILO:bilinear} to this bilinear form to obtain

\begin{cor}\label{cor:quadratic:row} There exist a set $I_0(U)$ of size $O_{C,\ep}(1)$ and a set $I(U)$ of size at least $n-n^\ep$ such that for any $i\in I$, there are integers $0\neq k(U)$ and $k_{ii_0}(U), i_0\in I_0(U)$, all bounded by $n^{O_{C,\ep}(1)}$, such that 
$$\P_Y\big(\langle k(U)\row_{B_U}(i),Y \rangle + \langle \sum_{i_0\in I_0} k_{ii_0}(U) \row_{B_U}(i_0),Y \rangle = 0\big)= N^{-O_{C,\ep}(1)},$$
where $Y=(y_1,\dots,y_N)$ and $y_i$ are i.i.d. copies of $\xi-\xi'$.
\end{cor}

Note that Corollary \ref{cor:quadratic:row} holds for all $U$. In what follows we gather the information together to obtain structures for the entire matrix $B$ (rather than for $B_U$). 

 As $I_0(U)\subset [N]^{O_{C,\ep}(1)}$ and $k(U)\le n$, there are only $N^{O_{C,\ep}(1)}$ possibilities that $(I_0(U),k(U))$ can take. Thus there exists a tuple $(I_0,k)$ such that 
$I_0(U)=I_0$ and $k(U)=k$ for $2^N/N^{O_{C,\ep}(1)}$ different $U$. Let us denote this set of $U$ by $\mathcal{U}$. Thus 
$$|\mathcal{U}|\ge 2^N/N^{O_{C,\ep}(1)}.$$
Next, let $I$ be the collection of $i$ which belong to at least $|\mathcal{U}|/2$ index sets $I_U$. Then we have
 \begin{align*}                         
|I||\mathcal{U}| + (N-|I|)|\mathcal{U}|/2 & \ge (N-N^\ep )|\mathcal{U}|\\
|I| &\ge  N-2N^\ep.
\end{align*}
Fix an $i\in I$. Consider the tuples $(k_{ii_0}(U), i_0\in I_0)$ where $i\in I_U$. Because there are only $N^{O_{C,\ep}(1)}$ possibilities such tuples can take, there must be a tuple, say $(k_{ii_0}, i_0\in I_0)$, such that $(k_{ii_0}(U), i_0\in I_0)=(k_{ii_0}, i_0\in I_0)$ for at least $|\mathcal{U}|/2N^{O_{C,\ep}(1)}=2^N/N^{O_{C,\ep}(1)}$ sets $U$. 

Because $|I_0|=O_{C,\ep}(1)$, there is a way to partition $I_0$ into $I_0' \cup I_0''$ such that there are $2^N/N^{O_{C,\ep}(1)}$ sets $U$ above satisfying that $I_0''\subset U$ and  $U\cap I_0'=\emptyset$. Let $\mathcal{U}_{I_0',I_0''}$ denote the collection of these $U$.

By passing to consider a subset of  $\mathcal{U}_{I_0',I_0''}$ if needed, we may assume that either $i\notin U$ or $i\in U$ for all $U\in  \mathcal{U}_{I_0',I_0''}$. Without loss of generality, we assume the first case that $i\notin U$. (The other case can be treated similarly).

Let $U\in \mathcal{U}_{I_0',I_0''}$ and $\Bu=(u_1,\dots,u_N)$ be its characteristic vector, that is $u_j=1$ if $j\in U$, and $u_j=0$ otherwise. Then, by the definition of $B_U$, and because $I_0''\subset U$ and $I_0'\cap U=\emptyset$, for $i_0'\in I_0'$ and $i_0''\in I_0''$ we can respectively write 
$$\langle \row_{i_0'}(B_U),Y \rangle = \sum_{j=1}^N a_{i_0'j}u_jy_j, \mbox{ and } \langle \row_{i_0''}(A_U),Y \rangle = \sum_{j=1}^N a_{i_0''j}(1-u_j)y_j.$$
Also, because $i\notin U$, we have $\langle \row_{i}(B_U),Y \rangle = \sum_{j=1}^N a_{ij}u_jy_j$. Thus, 
\begin{align*}
\langle k\row_i(B_U),Y \rangle + \sum_{i_0\in I_0} \langle k_{ii_0} \row_{i_0}(B_U),Y \rangle & = \langle k\row_i(B_U),Y \rangle +  \langle \sum_{i_0'\in I_0'} k_{ii_0'} \row_{i_0'}(B_U),Y \rangle + \langle \sum_{i_0''\in I_0''} k_{ii_0''} \row_{i_0''}(B_U),Y \rangle\\ 
&= \sum_{j=1}^N kb_{ij} u_jy_j + \sum_{j=1}^N \sum_{i_0'\in I_0'} k_{ii_0'} b_{i_0'j} u_jy_j +  \sum_{j=1}^N \sum_{i_0''\in I_0''} k_{ii_0''} b_{i_0''j} (1-u_j)y_j\\
&= \sum_{j=1}^n (kb_{ij} + \sum_{i_0'\in I_0'} k_{ii_0'} b_{i_0'j}- \sum_{i_0''\in I_0''} k_{ii_0''} b_{i_0''j} ) u_jy_j +  \sum_{j=1}^N \sum_{i_0''\in I_0''} k_{ii_0''} b_{i_0''j} y_j
\end{align*}
Next, by Corollary \ref{cor:quadratic:row}, for each $U\in \mathcal{U}_{I_0',I_0''}$ we have  $\P_Y\big (\langle k \row_i(B_U),Y \rangle + \sum_{i_0\in I_0} \langle k_{ii_0} \row_{i_0}(B_U),Y \rangle = 0\big)=N^{-O_{C,\ep}(1)}$. Also, note that $|\mathcal{U}_{I_0',I_0''}|= 2^N/N^{O_{C,\ep}(1)}$. Hence, 
$$\E_Y\E_U \big(k\langle \row_i(B_U),Y \rangle + \sum_{i_0\in I_0} \langle k_{ii_0} \row_{i_0}(B_U),Y \rangle =0\big) \ge n^{-O_{C,\ep}(1)}.$$
Finally, when $U$ runs through the subsets of $[N]$, we can view $\Bu$ as a random vector with i.i.d. $\{0,1\}$-Bernoulli with parameter 1/2.  By applying the Cauchy-Schwarz inequality, we obtain 
\begin{align*}
N^{-O_{C,\ep}(1)}&\le \left(\E_Y \E_U(k\langle \row_i(B_U),Y \rangle + \sum_{i_0\in I_0} \langle k_{ii_0} \row_{i_0}(B_U),Y \rangle =0)\right)^2 \\
&\le  \E_Y \left(\E_U(k\langle \row_i(B_U),Y \rangle + \sum_{i_0\in I_0} \langle k_{ii_0} \row_{i_0}(B_U),Y \rangle =0) \right)^2 \\
&= \E_Y \left(\E_{\Bu}(\sum_{j=1}^n (kb_{ij}+ \sum_{i_0'\in I_0'} k_{ii_0'} b_{i_0'j}-\sum_{i_0''\in I_0''} k_{ii_0''} b_{i_0''j}) u_jy_j+  \sum_{j=1}^n \sum_{i_0''\in I_0''} k_{ii_0''} b_{i_0''j} y_j= 0)\right)^2\\
&\le \E_Y \E_{\Bu,\Bu'}\big(\sum_{j=1}^n (k b_{ij}+\sum_{i_0'\in I_0'} k_{ii_0'} b_{i_0'j}-\sum_{i_0''\in I_0''} k_{ii_0''} b_{i_0''j}) (u_j-u_j')y_j= 0\big)\\
&= \E_Z \big(\sum_{j=1}^n (kb_{ij}+\sum_{i_0'\in I_0'} k_{ii_0'} b_{i_0'j}-\sum_{i_0''\in I_0''} k_{ii_0''} b_{i_0''j})z_j =0\big) 
\end{align*}
where $Z=(z_1,\dots, z_n)$ and $z_j:=(u_j-u_j')y_j$, and in the last inequality we used the simple observation that $\E_{u,u'}(f(u)=0,f(u')=0) \le \E_{u,u'}(f(u)-f(u')=0)$.

\section{Proof of Proposition \ref{prop:GAP:lap}} \label{section:GAP:lap}

First of all, by definition the number of structured vectors $\Bv$ is bounded by $n^{C_\ast' n} p^{n^\eps} \le n^{2C_\ast' n}$. Let $\la>0$ be a constant to be chosen sufficiently small later (for instance $\la \le \eps c /24C_\ast'$ would work). We divide $[N]$ into $\lfloor \la^{-1} \rfloor$ index intervals $I_i$ of length approximately $\la N$ each. Let $\rho_i = \rho_l(\Bv_{I_i})$, and let
$$\rho^\ast(\Bv) = \min_i \rho_i.$$
First assume that $\Bv$ is such that 
\begin{equation}\label{eqn:ast:large}
\rho^\ast(\Bv)\le n^{-3C_\ast' c^{-1}}.
\end{equation}
Then, supposing that $\rho^\ast$ is attained at $I_0$, by the decomposition from \eqref{eqn:i.i.d.decomposition:nL} we thus obtain that
$$\P(G_N \Bv = 0 \mbox{ in $n-O(1)$ coordinates}) \le n^{O(1)} (1/p + \rho^\ast)^{(1-\la)n} \le n^{-2.5 C_\ast' c^{-1} N} \le n^{-2.5 C_\ast' n}.$$
Here we used the assumption that $p$ is larger than $e^{n^{c/2}}$.
 
So the contribution of $\P(G_N \Bv = 0)$ over $\Bv$ satisfying \eqref{eqn:ast:large} is bounded by  $n^{2C_\ast' n}n^{-2.5 C_\ast' n}\le n^{-C_\ast' n/2}$.

Now we assume that 
$$n^{-3C_\ast' c^{-1}} \le \rho^\ast =O( n^{-1/2}),$$ 
where we note that $n^{-1/2}$ is the upper bound because our vectors are non-sparse. We divide this range into intervals $L_i$ of forms $[n^{-(i+1)\delta }, n^{-i\delta}]$ where $\delta$ is sufficiently small (such as $\delta\le \eps/4$) and consider the class $\CC_i$ where $\rho^\ast \in L_i$.

We claim that this class $\CC_i$ has at most $(n^{(i+1)\delta}/n^{\eps/2})^{n} p^{n^{\eps/2} \la^{-1}}$ vectors. Indeed this is because over each interval $I$ of length $\la n$ we have $\rho(\Bv_I) \ge \rho^\ast$. So by Theorem \ref{theorem:ILO}, we then have all but $n^{\eps}$ entries of $\Bv_I$ belong to a GAP of size $\rho^\ast/\sqrt{n^{\eps}}$. If we glue all the subvectors together we then obtain the bound as claimed, where $p^{\la^{-1}n^{\eps} }$ is the number of ways to choose the unstructured entries altogether. 

Now for each $\Bv \in \CC_i$, by using the decomposition from \eqref{eqn:i.i.d.decomposition:L} we have
$$\P(G_N \Bv = 0) \le (1/p + \rho^\ast)^{(1-\la)n} \le (2  n^{-i\delta})^{(1-\la)n},$$
where we again used the assumption that $p\ge e^{n^{c/2}}$. Taking union bound over $\Bv\in \CC_i$ we have
\begin{align*}
\P(\exists \Bv\in \CC_i, G_N \Bv = 0  \mbox{ in $n-O(1)$ coordinates})   & \le (n^{(i+1)\delta}/n^{\eps/2})^{n} p^{n^{\eps} \la^{-1}} (2  n^{-i\delta})^{(1-\la)n} \\
& \le 2^n p^{n^{\eps} \la^{-1}} n^{-(\eps/2 - \delta - i \delta \la)n }\\
&\le n^{-(\eps/8)n},
\end{align*}
provided that $\delta \le \eps/4$ and $\la \le \eps c/8C_\ast'$ (and so $i \delta \la \le 3C_\ast' c^{-1} \eps c/24C_\ast' =\eps/8$). Summing over all $i$ we thus obtain the claim.

\section{Proof of Lemma \ref{lemma:passing}}\label{section:passing} With room to spare, we use $O(1)$ to replace any quantity that might depend on $C,A,c,\eps, C_\ast, C_\ast'$. By restricting to the $n-O(1)$ columns that are orthogonal to $\Bw$, it suffices to consider the case that all column vectors generating $M_0$ are orthogonal to $\Bw$. In what follows $n'=n^\eps$.

We will first prove the ``if'' direction.
Assume that the first $m=n-n'$ entries of the normal vector $\Bw=(w_1,\dots, w_n)$ belong to a symmetric well-bounded GAP $Q$ with $r$ generators $a_1,\dots, a_r$ in $\Z/p\Z$, and $w_{j}  = \sum_{l=1}^r  x_{j l} a_l$ for $ 1\le j\le m$. 
Let $M$ be the matrix with entries at most $n^T$ in absolute value whose columns generate $M_0$.   Let $R_1,\dots, R_n$  be the rows of $M$. We have the equality modulo $p$
$$0 = \sum_{j=1}^m w_{j} R_{j}+   \sum_{j=m+1}^n w_{j} R_{j}  = \sum_{l=1}^r a_l (\sum_{j=1}^m x_{j l} R_{j}) +  \sum_{j=m+1}^n w_{j} R_{j} .$$
Now for $1\leq l\leq r$, let  $Z_l:= \sum_{j=1}^m x_{j l} R_{j}$.
We have $|x_{j l}|\leq |Q|\leq n^{O(1)}$.
  The entries of $Z_l$ are then bounded by $O(n^{T+O(1)})$, while the entries of $R_{{m+1}},\dots, R_{n}$ are bounded by $n^T$. 
Let $M'$ be the matrix whose columns are $Z_1,\dots,Z_r,R_{m+1},\dots R_n$.   
  The above identity then implies that $(a_1,\dots,a_r,w_{m+1},\dots w_n)^T$ is in the kernel of $M'$.
  Lemma \ref{lemma:passing:0} \eqref{i:kernel} applied to $M'$, with $k=r+n'$  implies that as long as $p\ge  e^{(k \log k)/2 + k(T+O(1)) \log n}$ (which is satisfied because $p\ge e^{n^{c}}$, and $n'=n^{\eps}$, and $n$ is sufficiently large), then
  there exist integers $a_l', w_j'$, reducing mod $p$ to $a_l, w_j$, for $1\leq l\leq r$ and $m+1\leq j\leq n$,
   such that
$$\sum_{k=1}^r a_l' Z_l + \sum_{j=m+1}^n w_j' R_{j}=0.$$
Let $\Bw'=(w_1',\dots,w_n')$ where $w_{j}'= \sum_{l=1}^r  x_{j l} a_l'$ for $1\le j\le m$. By definition the $w_{j}'$ for $1\le j\le m$ belong to the symmetric  GAP with generators $a'_l$ and with the same rank and dimensions as $Q$, and   $\Bw'$ is  normal to $M_0$.  Further $\Bw'$ is non-zero since it reduces to $w$ mod $p$.

The ``only if'' direction appears easier at first---if we start with a structured normal vector, we can reduce the generators of the GAP and the normal vector mod $p$ for any prime $p$.  However, the difficulty is that for general primes $p$ it is possible for the generators $a_l$ of the GAP to be not all $0$ mod $p$, but yet the resulting normal vector  $\Bw$ to be $0$ mod $p$.  
Given $M_0$, we choose $\Bw$ minimal (e.g. with $\sum_i |w_i|$ minimal) so that the first $m=n-n'$ entries (without loss of generality) of the normal vector $\Bw=(w_1,\dots, w_n)$ to $M_0$ belong to a symmetric well-bounded GAP $Q$ with $r$ generators $a_1,\dots, a_r$ in $\Z$, and $w_{j}  = \sum_{l=1}^r  x_{j l} a_l $ for  $1\le j\le m$ and $w$ is non-zero. 
 Let $M_\Bx$ be the $n\times (r+n')$ matrix with entries $x_{jl}$ in the first $m$ rows and $r$ columns, the $n'\times n'$ identity matrix in the last $n'$ rows and columns, and zeroes elsewhere.  
So for $\Ba:=(a_1,\dots,a_r,w_{m+1},\dots w_n)^T$, we have $M_\Bx \Ba=\Bw^T.$ 

Certainly by minimality of $\Bw$ at least some coordinate of $\Bw$  is not divisible by $p$ (else we could divide the $a_l$ and $w_j$ all by $p$ and produce a smaller structured normal $\Bw$).  
Suppose, for the sake of contradiction that all of the coordinates of $\Bw$ are divisible by $p$.
 The entries of $M_\Bx$ are bounded by $n^{O(1)}$,
so, as above, for $p\geq e^{n^{c}}$, by Lemma \ref{lemma:passing:0} \eqref{i:kernel} we have that $\ker M_\Bx|_{\Z^{r+n'}}$ surjects onto $\ker M_\Bx/p.$   So $\Ba/p$ is in the kernel of $M_\Bx/p$, and choose some lift $\Ba':=(a'_1,\dots,a'_r,w'_{m+1},\dots w'_n)^T\in\Z^n$ of $\Ba/p$ in the kernel of $M_\Bx$.  Then $\Ba-\Ba'\in p \Z^n$, and $M_\Bx(\frac{1}{p}(\Ba-\Ba'))=\frac{1}{p}\Bw$.  Note that $\frac{1}{p}\Bw$ is non-zero integral normal vector to $A$, and the equality $M_\Bx(\frac{1}{p}(\Ba-\Ba'))=\frac{1}{p}\Bw$ shows that all but $n'$ of the coordinates of 
$\frac{1}{p}\Bw$ belong to a symmetric well-bounded GAP with integral generators and the same rank and volume as $Q$, contradicting the minimality of $\Bw$.  Thus we conclude that $\Bw/p$ is non-zero and thus a structured normal vector of $M_0/p$ for the GAP $Q/p$.

\bibliography{MyLibrary.bib,Laplib.bib}

\newcommand{\etalchar}[1]{$^{#1}$}
\begin{thebibliography}{CMMM21}

\bibitem[AV12]{CC}
Carlos~A. Alfaro and Carlos~E. Valencia.
\newblock On the sandpile group of the cone of a graph.
\newblock {\em Linear Algebra Appl.}, 436(5):1154--1176, 2012.

\bibitem[BdlHN97]{Bacher}
Roland Bacher, Pierre de~la Harpe, and Tatiana Nagnibeda.
\newblock The lattice of integral flows and the lattice of integral cuts on a
  finite graph.
\newblock {\em Bull. Soc. Math. France}, 125(2):167--198, 1997.

\bibitem[Big99]{Big}
N.~L. Biggs.
\newblock Chip-firing and the critical group of a graph.
\newblock {\em J. Algebraic Combin.}, 9(1):25--45, 1999.

\bibitem[BKL{\etalchar{+}}15]{Bhargava2015b}
Manjul Bhargava, Daniel~M. Kane, Hendrik~W. Lenstra, Jr., Bjorn Poonen, and
  Eric Rains.
\newblock Modeling the distribution of ranks, {{Selmer}} groups, and
  {{Shafarevich}}-{{Tate}} groups of elliptic curves.
\newblock {\em Cambridge Journal of Mathematics}, 3(3):275--321, 2015.

\bibitem[BN09]{BN1}
Matthew Baker and Serguei Norine.
\newblock Harmonic morphisms and hyperelliptic graphs.
\newblock {\em Int. Math. Res. Not. IMRN}, (15):2914--2955, 2009.

\bibitem[BS10]{BSbook}
Zhidong Bai and Jack~W. Silverstein.
\newblock {\em Spectral analysis of large dimensional random matrices}.
\newblock Springer Series in Statistics. Springer, New York, second edition,
  2010.

\bibitem[Car54]{C}
L.~Carlitz.
\newblock Representations by quadratic forms in a finite field.
\newblock {\em Duke Math. J.}, 21:123--137, 1954.

\bibitem[CJMS21]{CJMS}
Marcelo Campos, Matthew Jenssen, Marcus Michelen, and Julian Sahasrabudhe.
\newblock The singularity probability of a random symmetric matrix is
  exponentially small, 2021.

\bibitem[CKL{\etalchar{+}}15]{Clancy2015}
Julien Clancy, Nathan Kaplan, Timothy Leake, Sam Payne, and Melanie~Matchett
  Wood.
\newblock On a {{Cohen}}\textendash{{Lenstra}} heuristic for {{Jacobians}} of
  random graphs.
\newblock {\em Journal of Algebraic Combinatorics}, pages 1--23, May 2015.

\bibitem[CLP15]{CLP}
Julien Clancy, Timothy Leake, and Sam Payne.
\newblock A note on {J}acobians, {T}utte polynomials, and two-variable zeta
  functions of graphs.
\newblock {\em Exp. Math.}, 24(1):1--7, 2015.

\bibitem[CMMM21]{CMMM}
Marcelo Campos, Let\'{\i}cia Mattos, Robert Morris, and Natasha Morrison.
\newblock On the singularity of random symmetric matrices.
\newblock {\em Duke Math. J.}, 170(5):881--907, 2021.

\bibitem[Coo17]{Cook}
Nicholas~A. Cook.
\newblock On the singularity of adjacency matrices for random regular digraphs.
\newblock {\em Probab. Theory Related Fields}, 167(1-2):143--200, 2017.

\bibitem[CTV06]{CTV}
Kevin~P. Costello, Terence Tao, and Van Vu.
\newblock Random symmetric matrices are almost surely nonsingular.
\newblock {\em Duke Math. J.}, 135(2):395--413, 2006.

\bibitem[Del01]{Delaunay2001}
Christophe Delaunay.
\newblock Heuristics on {{Tate}}-{{Shafarevitch Groups}} of {{Elliptic Curves
  Defined}} over {{$\mathbb{Q}$}}.
\newblock {\em Experimental Mathematics}, 10(2):191--196, 2001.

\bibitem[Dha90]{Dhar}
Deepak Dhar.
\newblock Self-organized critical state of sandpile automaton models.
\newblock {\em Phys. Rev. Lett.}, 64(14):1613--1616, 1990.

\bibitem[Ede88]{Edelman}
Alan Edelman.
\newblock Eigenvalues and condition numbers of random matrices.
\newblock {\em SIAM J. Matrix Anal. Appl.}, 9(4):543--560, 1988.

\bibitem[ESY10]{ESY}
L\'{a}szl\'{o} Erd\H{o}s, Benjamin Schlein, and Horng-Tzer Yau.
\newblock Wegner estimate and level repulsion for {W}igner random matrices.
\newblock {\em Int. Math. Res. Not. IMRN}, (3):436--479, 2010.

\bibitem[FG15]{FG}
Jason Fulman and Larry Goldstein.
\newblock Stein's method and the rank distribution of random matrices over
  finite fields.
\newblock {\em Ann. Probab.}, 43(3):1274--1314, 2015.

\bibitem[FJ19]{FJ}
Asaf Ferber and Vishesh Jain.
\newblock Singularity of random symmetric matrices---a combinatorial approach
  to improved bounds.
\newblock {\em Forum Math. Sigma}, 7:Paper No. e22, 29, 2019.

\bibitem[FJLS21]{FJLS}
Asaf Ferber, Vishesh Jain, Kyle Luh, and Wojciech Samotij.
\newblock On the counting problem in inverse {L}ittlewood-{O}fford theory.
\newblock {\em J. Lond. Math. Soc. (2)}, 103(4):1333--1362, 2021.

\bibitem[FJSS21]{FJSS}
Asaf Ferber, Vishesh Jain, Ashwin Sah, and Mehtaab Sawhney.
\newblock Random symmetric matrices: rank distribution and irreducibility of
  the characteristic polynomial, 2021.

\bibitem[FL16a]{FL}
Matthew Farrell and Lionel Levine.
\newblock Co{E}ulerian graphs.
\newblock {\em Proc. Amer. Math. Soc.}, 144(7):2847--2860, 2016.

\bibitem[FL16b]{FL2}
Matthew Farrell and Lionel Levine.
\newblock Multi-{E}ulerian tours of directed graphs.
\newblock {\em Electron. J. Combin.}, 23(2):Paper 2.21, 7, 2016.

\bibitem[GK19]{GK}
Darren Glass and Nathan Kaplan.
\newblock Chip-firing games and critical groups, 2019.

\bibitem[KKS95]{KKS}
Jeff Kahn, J\'{a}nos Koml\'{o}s, and Endre Szemer\'{e}di.
\newblock On the probability that a random {$\pm 1$}-matrix is singular.
\newblock {\em J. Amer. Math. Soc.}, 8(1):223--240, 1995.

\bibitem[KN22]{KNg}
Jake Koenig and Hoi Nguyen.
\newblock Rank of near uniform matrices.
\newblock {\em J. Comb.}, 13(3):397--436, 2022.

\bibitem[KNP21]{KNgP}
Jake Koenig, Hoi~H. Nguyen, and Amanda Pan.
\newblock A note on inverse results of random walks in abelian groups, 2021.

\bibitem[LLT{\etalchar{+}}21]{Tik}
Alexander~E. Litvak, Anna Lytova, Konstantin Tikhomirov, Nicole
  Tomczak-Jaegermann, and Pierre Youssef.
\newblock Circular law for sparse random regular digraphs.
\newblock {\em J. Eur. Math. Soc. (JEMS)}, 23(2):467--501, 2021.

\bibitem[LMN20]{LMNg}
Kyle Luh, Sean Meehan, and Hoi~H. Nguyen.
\newblock Some new results in random matrices over finite fields.
\newblock {\em Journal of the London Mathematical Society}, 103(4):1209–1252,
  Nov 2020.

\bibitem[Lor89]{Lo0}
Dino~J. Lorenzini.
\newblock Arithmetical graphs.
\newblock {\em Math. Ann.}, 285(3):481--501, 1989.

\bibitem[Lor91]{Lo1}
Dino~J. Lorenzini.
\newblock A finite group attached to the {L}aplacian of a graph.
\newblock {\em Discrete Math.}, 91(3):277--282, 1991.

\bibitem[Lor08]{Lorenzini2008}
Dino Lorenzini.
\newblock Smith normal form and {{Laplacians}}.
\newblock {\em Journal of Combinatorial Theory. Series B}, 98(6):1271--1300,
  2008.

\bibitem[M\'20]{Me}
Andr\'{a}s M\'{e}sz\'{a}ros.
\newblock The distribution of sandpile groups of random regular graphs.
\newblock {\em Trans. Amer. Math. Soc.}, 373(9):6529--6594, 2020.

\bibitem[Mac69]{Mac}
Jessie MacWilliams.
\newblock Orthogonal matrices over finite fields.
\newblock {\em Amer. Math. Monthly}, 76:152--164, 1969.

\bibitem[Mac15]{Macdonald2015}
I.~G. Macdonald.
\newblock {\em Symmetric Functions and {{Hall}} Polynomials}.
\newblock Oxford {{Classic Texts}} in the {{Physical Sciences}}. {The Clarendon
  Press, Oxford University Press, New York}, second edition, 2015.

\bibitem[Map13a]{M1}
Kenneth Maples.
\newblock Singularity of random matrices over finite fields, 2013.

\bibitem[Map13b]{Msym}
Kenneth Maples.
\newblock Symmetric random matrices over finite fields announcement, 2013.

\bibitem[McK81]{Mckay}
Brendan~D. McKay.
\newblock Subgraphs of random graphs with specified degrees.
\newblock {\em Congr. Numer.}, 33:213--223, 1981.

\bibitem[Meh67]{Mehta}
M.~L. Mehta.
\newblock {\em Random matrices and the statistical theory of energy levels}.
\newblock Academic Press, New York-London, 1967.

\bibitem[MR]{MR}
M~L Mehta and N~Rosenzweig.
\newblock Distribution laws for the roots of a random antisymmetric hermitian
  matrix.
\newblock {\em Nucl. Phys., A109: 449-56(1968).}

\bibitem[Ngu12]{Ng}
Hoi~H. Nguyen.
\newblock Inverse {L}ittlewood-{O}fford problems and the singularity of random
  symmetric matrices.
\newblock {\em Duke Math. J.}, 161(4):545--586, 2012.

\bibitem[Ngu18]{NgF}
Hoi~H. Nguyen.
\newblock Random matrices: overcrowding estimates for the spectrum.
\newblock {\em J. Funct. Anal.}, 275(8):2197--2224, 2018.

\bibitem[NP20]{NgP}
Hoi.~H. Nguyen and Elliot Paquette.
\newblock Surjectivity of near-square random matrices.
\newblock {\em Combin. Probab. Comput.}, 29(2):267--292, 2020.

\bibitem[NTV17]{NgTV}
Hoi Nguyen, Terence Tao, and Van Vu.
\newblock Random matrices: tail bounds for gaps between eigenvalues.
\newblock {\em Probab. Theory Related Fields}, 167(3-4):777--816, 2017.

\bibitem[NW22]{NgW}
Hoi~H. Nguyen and Melanie~Matchett Wood.
\newblock Random integral matrices: universality of surjectivity and the
  cokernel.
\newblock {\em Invent. Math.}, 228(1):1--76, 2022.

\bibitem[Pas72]{Pastur}
L.~A. Pastur.
\newblock The spectrum of random matrices.
\newblock {\em Teoret. Mat. Fiz.}, 10(1):102--112, 1972.

\bibitem[PR12]{Poonen2012}
Bjorn Poonen and Eric Rains.
\newblock Random maximal isotropic subspaces and {{Selmer}} groups.
\newblock {\em Journal of the American Mathematical Society}, 25(1):245--269,
  2012.

\bibitem[Rus90]{Rush}
Joseph~J. Rushanan.
\newblock Combinatorial applications of the {S}mith normal form.
\newblock In {\em Proceedings of the {T}wentieth {S}outheastern {C}onference on
  {C}ombinatorics, {G}raph {T}heory, and {C}omputing ({B}oca {R}aton, {FL},
  1989)}, volume~73, pages 249--254, 1990.

\bibitem[SS16]{SS}
Philippe Sosoe and Uzy Smilansky.
\newblock On the spectrum of random anti-symmetric and tournament matrices.
\newblock {\em Random Matrices Theory Appl.}, 5(3):1650010, 33, 2016.

\bibitem[TV07]{TV}
Terence Tao and Van Vu.
\newblock On the singularity probability of random {B}ernoulli matrices.
\newblock {\em J. Amer. Math. Soc.}, 20(3):603--628, 2007.

\bibitem[TV09]{TVinverse}
Terence Tao and Van~H. Vu.
\newblock Inverse {L}ittlewood-{O}fford theorems and the condition number of
  random discrete matrices.
\newblock {\em Ann. of Math. (2)}, 169(2):595--632, 2009.

\bibitem[TV10a]{TVcir}
Terence Tao and Van Vu.
\newblock Random matrices: universality of {ESD}s and the circular law.
\newblock {\em Ann. Probab.}, 38(5):2023--2065, 2010.
\newblock With an appendix by Manjunath Krishnapur.

\bibitem[TV10b]{TVbook}
Terence Tao and Van~H. Vu.
\newblock {\em Additive combinatorics}, volume 105 of {\em Cambridge Studies in
  Advanced Mathematics}.
\newblock Cambridge University Press, Cambridge, 2010.
\newblock Paperback edition [of MR2289012].

\bibitem[Ver14]{Ver}
Roman Vershynin.
\newblock Invertibility of symmetric random matrices.
\newblock {\em Random Structures Algorithms}, 44(2):135--182, 2014.

\bibitem[Wig58]{Wigner}
Eugene~P. Wigner.
\newblock On the distribution of the roots of certain symmetric matrices.
\newblock {\em Ann. of Math. (2)}, 67:325--327, 1958.

\bibitem[Woo17]{Wood2017}
Melanie Wood.
\newblock The distribution of sandpile groups of random graphs.
\newblock {\em Journal of the American Mathematical Society}, 30(4):915--958,
  2017.

\bibitem[Woo19]{W1}
Melanie~Matchett Wood.
\newblock Random integral matrices and the {C}ohen-{L}enstra heuristics.
\newblock {\em Amer. J. Math.}, 141(2):383--398, 2019.

\end{thebibliography}

\bibliographystyle{alpha}

\end{document}